\documentclass[11pt,letterpaper]{amsart}

\usepackage{float}

\usepackage[dvipsnames]{xcolor}

\usepackage[all]{xy}

\usepackage[T1]{fontenc}
\usepackage[utf8]{inputenc}

\usepackage{amsmath,amssymb,amsthm,mathtools,enumerate,extarrows,stmaryrd,mathrsfs,textcomp}

\usepackage{mathtools}
\usepackage{extarrows}
\usepackage{tikz}
\usepackage[bbgreekl]{mathbbol}
\DeclareSymbolFontAlphabet{\mathbbm}{bbold}
\DeclareSymbolFontAlphabet{\mathbb}{AMSb}%

\UseRawInputEncoding
\let\oldtocsubsection=\tocsubsection
\renewcommand{\tocsubsection}[2]{\hspace{1em}\oldtocsubsection{#1}{#2}}

\usepackage{hyperref}
\hypersetup{colorlinks=true,linkcolor=blue,citecolor=teal,filecolor=magenta,urlcolor=cyan}

\usepackage{url}

\usepackage[backend=bibtex,style=alphabetic,sorting=nyt,isbn=false,url=false,doi=true,maxalphanames=10,minalphanames=4,mincitenames=4,maxcitenames=10,minnames=4,maxnames=10,giveninits=true,maxbibnames=99]{biblatex}
\usepackage[newparttoc]{titlesec}
\usepackage{appendix}
\usepackage{titletoc}
\titleformat{\section}
{\normalfont\bfseries\large}{\thesection}{1em}{}
\titlecontents{section}
	[3.5em]{}
	{\contentslabel{1.5em}}
	{}
	{\hfill\contentspage}
\titleformat{\subsection}
{\normalfont\large\slshape}{\thesubsection}{1em}{}
\titlecontents{subsection}
	[5.5em]{}
	{\contentslabel{1.5em}\itshape}
	{}
	{\hfill\contentspage}
\titleformat{\subsubsection}
{\normalfont\itshape}{\thesubsubsection}{1em}{}
\titlecontents{subsubsection}
	[7.5em]{}
	{\contentslabel{1.5em}\itshape}
	{}
	{\hfill\contentspage}

\theoremstyle{plain}
\newtheorem{theorem}{Theorem}[section]
\newtheorem{proposition}[theorem]{Proposition}
\newtheorem{corollary}[theorem]{Corollary}
\newtheorem{conjecture}[theorem]{Conjecture}
\newtheorem{lemma}[theorem]{Lemma}
\theoremstyle{definition}
\newtheorem{definition}[theorem]{Definition}
\newtheorem{example}[theorem]{Example}

\theoremstyle{remark}
\newtheorem{remark}[theorem]{Remark}
\newcommand{\bra}{\langle |}
\newcommand{\ket}{|\rangle}
\newcommand{\Bra}{\Big\langle\Big|}
\newcommand{\Ket}{\Big|\Big\rangle}
\newcommand{\PSG}{\mathbb{PS}}
\newcommand{\dd}{\mathrm{d}}

\begin{document}
	
\title{Functional relations for higher-order free cumulants}

\author[G.~Borot]{Ga\"etan Borot}
\address{Humboldt-Universit\"at zu Berlin, Institut f\"ur Mathematik und Institut f\"ur Physik \\ Unter den Linden 6, 10099 Berlin, Germany.}
\email{\href{mailto:gaetan.borot@hu-berlin.de}{gaetan.borot@hu-berlin.de}}

\author[S.~Charbonnier]{S\'everin Charbonnier}
\address{IRIF, Universit\'e de Paris \\ B\^atiment Sophie Germain, Case courrier 7014, 8 Place Aur\'elie Nemours, 75205 Paris Cedex 13, France.}
\email{\href{mailto:charbonnier@irif.fr}{charbonnier@irif.fr}}

\author[E.~Garcia-Failde]{Elba Garcia-Failde}
\address{Institut de Math\'ematiques de Jussieu - Paris Rive Gauche \\ Sorbonne Universit\'e, 4 Place Jussieu, 75252 Paris Cedex 05, France.}
\email{\href{mailto:egarcia@imj-prg.fr}{elba.garcia-failde@imj-prg.fr}}

\author[F.~Leid]{Felix Leid}
\address{Universit\"at des Saarlandes, Fachrichtung Mathematik \\ Postfach 1511505, 66041 Saarbr\"ucken, Germany.}
\email{\href{mailto:felix.leid@math.uni-sb.de}{felix.leid@math.uni-sb.de}}

\author[S.~Shadrin]{Sergey Shadrin}
\address{Korteweg-de Vries Instituut voor Wiskunde, Universiteit van Amsterdam \\ Postbus 94248, 1090 GE Amsterdam, the Netherlands.}
\email{\href{mailto:s.shadrin@uva.nl}{s.shadrin@uva.nl}}

\begin{abstract} We establish the functional relations between generating series of higher-order free cumulants and moments in higher-order free probability, solving an open problem posed fifteen years ago by Collins, Mingo, \'Sniady and Speicher. We propose an extension of free probability theory, which governs the all-order topological expansion in unitarily invariant matrix ensembles, with a corresponding notion of free cumulants and give as well their relation to moments via functional relations. Our approach is based on the study of a master transformation involving double monotone Hurwitz numbers via semi-infinite wedge techniques, building on the recent advances of the last-named author with Bychkov, Dunin-Barkowski and Kazarian. We illustrate our formulae by computing the first few decaying terms of the correlation functions of an ensemble of spiked GUE matrices, going beyond the law of large numbers and the central limit theorem.

\end{abstract}

\maketitle

\tableofcontents	

\section{Introduction}
Voiculescu introduced free probability in the 1980's in the context of operator algebras. In this framework, the notion of independence of variables in classical probability is replaced by a notion of freeness, for which free cumulants are crucial objects. The first-order moments and free cumulants of a random variable $a$ are encoded in generating series:
\begin{equation}
\label{MCdef}
M(X)\coloneqq1+\sum_{k \geq 1}\varphi_{1}(a^k) X^{k}\,, \qquad C(w)\coloneqq 1+\sum_{k \geq 1}\kappa_{k}(a) w^{k}\,,
\end{equation}
which satisfy the following functional relation
\begin{equation}\label{eq:mome:cumu:1}
C(X\,M(X))= M(X)\,.
\end{equation}
It is sometimes referred to as the $R$-transform machinery, originally given in \cite{Voiculescu} in a slightly different form: with the $R$-transform $R(w)$ and the Stieltjes transform $W(x)$ related to \eqref{MCdef} by $C(w) = 1 + w R(w)$ and $W(x) = x^{-1}M(x^{-1})$, where $x=X^{-1}$. After the pioneering work of Voiculescu \cite{Voi91}, these notions brought a better understanding of the law of large numbers for the spectra of large random matrices, which in turn enriched the study of free probability.

Later, the study of random matrices of large size beyond the law of large numbers motivated, in \cite{MingoSpeicher, MMS07, CMSS},  the development of a theory of higher-order freeness. Given a higher-order probability space, these works introduce higher-order moments and free cumulants, the latter being obtained from the former via a convolution product involving non-crossing partitioned permutations. In the same manner as for the first order, one would like to relate the generating series of moments and free cumulants of order $n \geq 2$ of a variable $a$:
\begin{equation*}
\begin{split}
M_n(X_1,\dots,X_n) & \coloneqq \sum_{k_1,\ldots,k_n \geq 1}\varphi_n(a^{k_1},\ldots,a^{k_n}) \prod_{i = 1}^n X_i^{k_i} \,,\\
 C_n(w_1,\dots,w_n) & \coloneqq \sum_{k_1,\ldots,k_n\geq 1} \kappa_{k_1,\ldots,k_n}(a,\ldots,a) \prod_{i = 1}^n w_i^{k_i}\,.
\end{split}
\end{equation*}
In \cite{CMSS}, second-order freeness was studied in detail, and an explicit functional relation for $n=2$ was found. It can be rewritten in the simple form
\begin{equation}
\label{eq:mome:cumu:2}
M_2(X_1,X_2) + \frac{X_1X_2}{(X_1 - X_2)^2} = \frac{\dd \ln w_1}{\dd \ln X_1}\,\frac{\dd \ln w_2}{\dd \ln X_2} \bigg(C_2(w_1,w_2) + \frac{w_1w_2}{(w_1 - w_2)^2}\bigg)\,,
\end{equation}
where $w_i = X_i M(X_i)$, or equivalently $X_i = w_i/C(w_i)$. 
However, due to the complicated combinatorics of partitioned permutations, the problem of obtaining such functional relations for $n\geq 3$ has remained open until now, limiting the practical applicability of higher-order freeness. The present work remedies this by establishing in Theorem \ref{thm:R-transform-GenusZero} the functional relations between generating series of moments and free cumulants to any order $n\geq 1$. They have the following structure.

\begin{theorem}
Consider the change of variables $X_i= w_i/C(w_i)$. For $n \geq 3$, we have:
\begin{equation}
\label{MnCn} M_n(X_1,\ldots,X_n) =  \sum_{r_1,\ldots,r_n \geq 0} \,\,\sum_{T \in \mathcal{G}_{0,n}(\mathbf{r} + 1)}\prod_{i=1}^n \vec{\mathsf{O}}_{r_i}^{\vee}(w_i) \prod_{I \in \mathcal{I}(T)}' C_{\# I}(w_I)\,,
\end{equation}
where:
\begin{itemize}
\item $\mathcal{G}_{0,n}(\mathbf{r} + 1)$ is the set of bicoloured trees with white vertices labelled from $1$ to $n$ having valency $r_1 + 1,\ldots,r_n + 1$, and without univalent black vertices.
\item The weight $\vec{\mathsf{O}}_{r_i}^{\vee}(w_i)$ of the $i$-th white vertex is a differential operator acting on the variable $w_i$. Its expression (Definition~\ref{Or2}) involves only $C(w_i)$.
\item $\mathcal{I}(T)$ is the set of black vertices, identified with the subset of white vertices they connect to.
\item $\prod'$ means that any occurrence of $C_{2}(w_i,w_j)$ should be replaced with 
$$
C_{2}(w_i,w_j) + \frac{w_iw_j}{(w_i - w_j)^2}.
$$
\item For a given monomial in the $X$s, only finitely many terms of the right-hand side contribute.
\end{itemize}
\end{theorem}

Formula \eqref{MnCn} stems from combinatorics, but indirectly and with a convoluted history. In the case where the higher-order probability space arises from the large-size limit of an arbitrary (formal) unitarily-invariant measure of the space of Hermitian matrices, then moments are identified with generating series of ordinary maps \cite{BIPZ}. Later, for the same type of measure, the first and third-named authors \cite[Section 11.2]{BG-F18} identified the free cumulants with generating series of planar fully simple maps. Using those identifications, the functional relations \eqref{eq:mome:cumu:1} and \eqref{eq:mome:cumu:2} admit combinatorial proofs \cite{BG-F18}. Generating series of maps of arbitrary genus $g$ and number of boundaries $n$ are known to satisfy a universal recursion on $2g - 2 + n$, known under the name of topological recursion \cite{EORev,Ebook}. Its initial data is encoded in the plane curve of equation $w = W(x) = \frac{1}{x}M(\frac{1}{x})$, called spectral curve. \cite{BG-F18} observed that \eqref{eq:mome:cumu:1} and \eqref{eq:mome:cumu:2} correspond to a transformation of the spectral curve called ``symplectic exchange'', which plays an important role in the theory of topological recursion \cite{EORev,EO2MM,EOxy}, and then conjectured that the topological recursion for the exchanged spectral curve $x = C(w)/w$ produces the
generating series of fully simple maps of any topology. The functional relation \eqref{MnCn} for $n = 3$ (in the special case of maps) can be derived from this conjecture, but the method could not be pushed to $n \geq 4$. The conjecture has received two proofs recently: using combinatorial tools in \cite{BCG-F} by the three first-named authors; and for more general models of maps, i.e.~more general unitarily-invariant ensembles of random matrices, using the Fock space formalism, in \cite{BDBKSexplicit, BDBKS} by the fifth-named author with Bychkov, Dunin-Barkowski and Kazarian. The latter works are part of a general approach establishing, among other things, topological recursion for Hurwitz theory~\cite{BDBKSHurw}.

\medskip

The starting point of this article will be the \emph{master relation} of the form
\begin{equation}
\label{master} Z(\lambda) = \mathfrak{z}_{\lambda}\sum_{\nu\vdash d} H^{<}(\lambda,\nu) Z^{\vee}(\nu)
\end{equation}
between two functions $Z$ and $Z^{\vee}$ on the set of integer partitions, where $H^{<}(\lambda,\nu)$ are the strictly monotone (double) Hurwitz numbers and $\mathfrak{z}_{\lambda}$ is a symmetry factor (Sections~\ref{Sec:PS} and \ref{Sec:Hurbase}). The weakly monotone Hurwitz numbers, encoded in $H^{\leq}(\lambda,\nu)$, and which are in some sense (that will become explicit in Lemma~\ref{Hinverse}) dual to the strict version, were introduced and studied in \cite{Goulden1,Goulden2,Goulden3,NovakHCIZ}. As proved in \cite{BG-F18} via Weingarten calculus and recalled in Theorem~\ref{tunome}, the master relation \eqref{master} materialises in the context of unitarily invariant ensembles of random hermitian matrices, where $Z(\lambda)$ comes from moments of traces of powers and $Z^{\vee}(\nu)$ from expectation values of products of entries along cycles of a permutation. From the relation between map enumeration and formal matrix integrals, the master relation also relates the enumeration of (stuffed) maps (in $Z(\lambda)$) to the enumeration of fully simple (stuffed) maps (in $Z^{\vee}(\nu)$). This admits two bijective proofs and applies  as well to hypermaps \cite{borot2019relating}. 

\medskip

We show that having the master relation \eqref{master} between two \emph{arbitrary} generating series is equivalent to a moment-cumulant relation where moments are expressed as an extended convolution of the free cumulants with the zeta function on partitioned permutations (Theorem~\ref{thm:hbarstarproduct}). By extended, we mean that we do not throw away the contribution of non-planar partitioned permutations, contrarily to the convolution used in \cite[Section 5]{CMSS}. Representing these arbitrary generating series as expectation values in the Fock space, the universal relation amounts to inserting an operator creating the monotone Hurwitz numbers (Section~\ref{Sec:TransZ}). In Section~\ref{Sec3}, we present the core functional relations stated in Theorem~\ref{thm:R-transform-HigherGenera}. The extended convolution can be truncated to keep only the planar contributions as in \cite{CMSS}, and by truncating accordingly we obtain the functional relations of Theorem~\ref{thm:R-transform-GenusZero} (or in the equivalent form of Theorem~\ref{coeffThm}), summarised above as \eqref{MnCn}. The rest of Section~\ref{Sec3} is devoted to examples and reformulations of the functional relations and a brief discussion of their analytic properties. The  proof of the main theorem (Theorem~\ref{thm:R-transform-HigherGenera}) exploits the tools developed in \cite{BDBKSexplicit} and is exposed in Section~\ref{Sec4}. These functional relations can be inverted to rather express free cumulants in terms of moments (Section~\ref{SecDual}). This strategy is summarised in Figure~\ref{Fig:diagsum}.

\medskip

\begin{figure}[h!]
\begin{center}
\includegraphics[width= \textwidth]{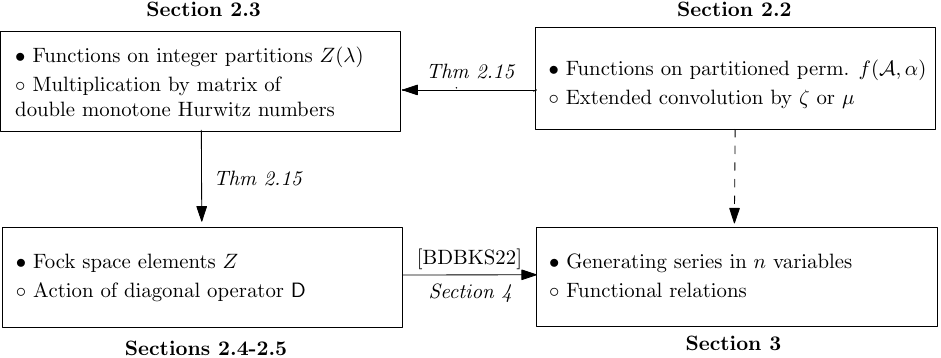}
\caption{\label{Fig:diagsum} Overview of the proof of the main result. The master relation is in the top-left corner. All arrows are in fact equivalences.}
\end{center}
\end{figure}

The application of these results to higher-order free probability is summarised in Sections~\ref{Decpartperm}-\ref{Sec:Higherorderfree}. Besides, this leads us to propose in Section~\ref{SecFreehig} a natural extension of the free probability theory to include non-planar cases, to all orders. It is based on surfaced permutations, which are essentially partitioned permutations with genus information (Section~\ref{Sec:PSG}). In this context we define an extended convolution, a notion of surfaced moments and free cumulants related by extended convolution with the zeta function, and the functional relations are described by Theorem~\ref{thm:R-transform-HigherGenera} (Section~\ref{SecFreehig}). We argue that this extension is the right one:
\begin{itemize}
\item We define $(g,n)$-freeness of variables by vanishing of mixed cumulants up to order $(g,n)$, and this is equivalent to $(g,n)$-freeness of the algebras generated by those variables (Proposition~\ref{SettoAlg}).
\item The $(\infty,\infty)$-asymptotic freeness  of two independent ensembles of random matrices, one of which is unitarily invariant, is guaranteed (Section~\ref{Sec:Randommat}).
\item We stress that, compared to previous work in free probability where an emphasis was put on keeping only non-crossing or planar contributions, our conclusion is that the theory becomes simpler once one extends it to keep all genera. It is in this context that we derive functional relations, which we (only then) truncate to obtain the desired genus-zero relations \eqref{MnCn}.
\item It is convenient to allow the genus to take half-integer values. Then, $(\frac{1}{2},1)$-freeness retrieves the notion of infinitesimal freeness of \cite{NF-infinitesimal} coming from \cite{BS09}, and the known functional relations between the corresponding moments and free cumulants (Corollary~\ref{the1demi}).
\end{itemize} 
In Section~\ref{SecGUEDET}, we illustrate the functional relations by a computation of (the generating series of) moments of order $(g,n) = (0,3),(1,1)$ of a GUE + deterministic matrix, going beyond the known law of large numbers (corresponding to $(0,1)$) and central limit theorem (corresponding to $(0,2)$).

Albeit not used in this article, the theory of the topological recursion whispered to us that keeping all genera was the natural thing to do. It also led us to import for the present purposes the Fock space techniques recently developed to understand better the interplay between Hurwitz theory and topological recursion. Actually, we are led to formulate Conjecture~\ref{ConjTR}, stating that the master relation \eqref{master} and the functional relations of Theorem~\ref{thm:R-transform-HigherGenera} exactly describe the effect of symplectic exchange in topological recursion. 

\subsection*{Added on revision} Several developments based on the results of this paper have appeared since the first version of this paper was released in December 2021, which are relevant for applications in higher-order free probability theory. The key formula for the higher-order moments / free cumulants correspondence (which is a genus-zero formula in our framework stated in Theorem~\ref{thm:R-transform-GenusZero}) has got an alternative formulation by Hock in \cite{hock2022xy}, which is visibly simpler. For the reader's interest we present Hock's formula in Appendix~\ref{sec:Hock}.  In the special case of third-order freeness, Lionni \cite{lionni2022higher} proposed a combinatorial interpretation of the formula relating moments and cumulants. In higher genera, an analogous simplification of the correspondence formula stated in Theorem~\ref{thm:R-transform-HigherGenera} was derived in~\cite{hock2022simple} and~\cite{alexandrov2022universal}.  Moreover, our Conjecture~\ref{ConjTR} has been proved in \cite{alexandrov2022universal}.

\subsection*{Acknowledgements}
We thank the anonymous referees for their comments, which helped improve the exposition. We thank O.~Arizmendi, A.~Bufetov, J. Mingo and R.~Speicher for discussions. We also thank H.~Grosse, M.~Khalkhali, H.~Markwig, J.~Sch\"urmann, and R.~Wulkenhaar, the co-organisers of the workshops ``Non-commuta\-ti\-ve geometry meets topological recursion'' in August 2021 in M\"unster and in April 2023 at the Erwin Schr\"odinger Institute in Vienna: the first workshop led to a breakthrough in our understanding of the structure of the relations between the higher-order free cumulants and moments, the second one allowed us to work on a revision and extension of this paper. S.~C.~was supported by the European  Research Council  (ERC)  under  the  European  Union's Horizon  2020 research and innovation programme  (grant  agreement  No.~ERC-2016-STG 716083  ``CombiTop''). F.~L.~was supported by the SFB-TRR 195 ``Symbolic Tools in Mathematics and their Application'' of the German Research Foundation (DFG). S.~S.~was supported by the Netherlands Organisation for Scientific Research.

\section{Convolution on partitioned permutations and Hurwitz numbers}

\label{S2}
In Section~\ref{Sec:PS}, we set up some preliminary notations in order to review partitioned permutations in Section~\ref{Sec:PSS}, in which we also introduce some new definitions needed for our purposes. In Section~\ref{Sec:Hurbase}, following \cite{ALS16, borot2019relating}, we recall notations and basic facts concerning monotone Hurwitz numbers. This allows us to recast in Section~\ref{Sec:TransZ}  the extended convolution by the zeta function in the algebra of functions on partitioned permutations, in terms of transformations of topological partition functions using monotone Hurwitz numbers.

\subsection{Integer and set partitions}

\label{Sec:PS}
Let $d$ be a nonnegative integer. The set of permutations of $[d]\coloneqq \{1,\ldots,d\}$  is denoted $S(d)$. We say that $\lambda = (\lambda_i)_{i = 1}^{\ell}$ is a partition of $d$ (notation $\lambda \vdash d$) when it is a weakly decreasing sequence of positive integers such that $\sum_{i = 1}^{\ell} \lambda_i = d$. We denote $\ell(\lambda) = \ell$ the length of the sequence. If $(k_1,\ldots,k_{\ell})$ is a sequence of positive integers, we denote $\lambda(\mathbf{k})$ this sequence written in weakly decreasing order. If $\lambda \vdash d$, we  associate to it the permutation $\pi_{\lambda} \in S(d)$:
\[
\pi_{\lambda}\;\coloneqq\; (1\,\,\dots\,\,\lambda_1)(\lambda_1+1\,\,\dots\,\,\lambda_{1}+\lambda_{2})\cdots(\lambda_1+\dots+\lambda_{\ell-1}+1\,\,\dots\,\,\lambda_1+\dots+\lambda_{\ell})\,. \label{gammalamb}
\]
Conversely, if $\sigma$ is a permutation, we denote $\lambda(\sigma)$ the sequence of lengths of the cycles of $\sigma$, in weakly decreasing order. By convention, we define $S(0) = \{\emptyset\}$, and we declare $\emptyset$  to be the (only) partition of $0$.  Given $\lambda \vdash d$, let $C_{\lambda} \subseteq S(d)$ be the conjugacy class of $\pi_{\lambda}$, that is, the set of permutations $\sigma\in S(d)$ such that $\lambda(\sigma)=\lambda$, and
\[
\mathfrak{z}_{\lambda} = \frac{d!}{\# C_{\lambda}} = \prod\limits_{i=1}^{\ell(\lambda)}\, \lambda_i \, \prod\limits_{j\geq 1} m_j(\lambda)!\,,
\]
 where $m_j(\lambda)$ is the number of occurrences of $j$ in the sequence $\lambda$.

For us, a partition of $[d]$ is a set of non-empty and pairwise disjoint subsets of $[d]$ whose union is $[d]$. The set $P(d)$ of partitions of $[d]$ is endowed with the structure of a poset. Namely, let  $\mathcal{A}, \,\mathcal{B}\in P(d)$, $\mathcal{A}=\{A_1,\dots, A_a\},\, \mathcal{B}=\{B_1,\dots,B_{b}\}$. The partial order relation  $\mathcal{A}\leq \mathcal{B}$ holds when every block $A_i$ is contained in some block $B_j$. The trivial partition $\mathbf{1}_d \coloneqq \{[d]\} \in P(d)$ is the maximum of $P(d)$. If $\mathcal{A},\mathcal{B} \in P(d)$, their merging is denoted $\mathcal{A} \vee \mathcal{B} \in P(d)$: it is the smallest $\mathcal{C} \in P(d)$ such that $\mathcal{A} \leq \mathcal{C}$ and $\mathcal{B} \leq \mathcal{C}$. In particular, we have $\mathcal{A} \vee \mathbf{1}_d = \mathbf{1}_d$. If $\sigma \in S(d)$, we define $\mathbf{0}_{\sigma} \in P(d)$ to be the partition of $[d]$ whose elements are the supports of cycles in $\sigma$. For instance $\mathbf{0}_{(126)(35)(4)} = \{\{1,2,6\},\{3,5\},\{4\}\}$. Note that $\mathbf{0}_{{\rm id}}$ is the minimum of $P(d)$, and we have $\mathcal{A} \vee \mathbf{0}_{\rm id} =  \mathcal{A}$. We also use the notation $\mathbf{0}_{\lambda} \coloneqq \mathbf{0}_{\pi_{\lambda}}$.

\subsection{Partitioned permutations}
\label{Sec:PSS}

\begin{definition} A \emph{partitioned permutation}  of $d$ elements is a pair $(\mathcal{A},\alpha)$, where $\mathcal{A}\in P(d)$ and $\alpha \in S(d)$, such that $\mathbf{0}_{\alpha} \leq \mathcal{A}$. Let $PS(d)$ be the set of partitioned permutations of $d$ elements  and $PS=\bigcup_{d\geq 1} PS(d)$.
\end{definition}
 We define the \emph{colength} of $\mathcal{A}\in P(d)$ and of $\alpha \in S(d)$ as
\[
|\mathcal{A}|\coloneqq d-\# \mathcal{A}\;\; \text{ and }\;\; |\alpha|\coloneqq d-\#\mathbf{0}_\alpha\,,\]
where $\#\mathcal{A}$ and $\#\mathbf{0}_\alpha$ are the number of blocks of $\mathcal{A}$ and the number of cycles of $\alpha$, respectively. Then, $|\alpha|$ is the minimal number of transpositions in a factorisation of $\alpha$. The colength of a partitioned permutation $(\mathcal{A},\alpha)\in PS(d)$ is defined as:
\[
|(\mathcal{A},\alpha)|\coloneqq 2|\mathcal{A}|-|\alpha|\,.
\]
It is a nonnegative integer, as the property $\mathbf{0}_{\alpha} \leq \mathcal{A}$ implies $\# \mathbf{0}_{\alpha} \geq \# \mathcal{A}$. Note the special case:
\begin{equation}
\label{spcol} |(\mathbf{0}_{\alpha},\alpha)| = |\alpha|\,,
\end{equation}
and the additivity:
\begin{equation}
\label{Addico}|(\mathcal{A},\alpha)| = \sum_{A \in \mathcal{A}} |(\mathbf{1}_{\# A},\alpha_{|A})|\,,
\end{equation}
where $\alpha_{|A}$ is the restriction of the permutation $\alpha$ to the block $A$. Here we have made a choice of bijections $[\#A] \rightarrow A$ to consider $(\mathbf{1}_{\# A},\alpha_{|A}) \in PS(\# A)$, but its colength appearing on the right-hand side of \eqref{Addico} is independent of these choices.

We recall the definitions of the product of partitioned permutations and the associated convolution, and introduce their extended version which is relevant for us. Hereafter $R$ denotes an arbitrary commutative ring.
\begin{definition}\label{def:extension}
Let $(\mathcal{A},\alpha), (\mathcal{B},\beta),(C,\gamma) \in PS(d)$ and $f_1,f_2 \colon PS(d) \rightarrow R$ be two functions.
\begin{itemize}
	\item The multiplication\footnote{Strictly speaking, $\odot$ is a product on $PS(d)$, while $\cdot$ is not because $0$ is not an element of $PS(d)$ (it can be upgraded to a product on the set $PS(d) \cup \{0\}$, with $0$ declared to be an absorbing element). This distinction is irrelevant for the definition of the convolution product $\ast$ on functions on $PS$.} $\cdot$ and the extended multiplication $\odot$ of partitioned permutations are:
\[
\begin{split}
(\mathcal{A},\alpha)\cdot (\mathcal{B},\beta)&\coloneqq \begin{cases} (\mathcal{A}\vee \mathcal{B},\alpha \circ \beta)  & \textup{if } |(\mathcal{A},\alpha)|+|(\mathcal{B},\beta)|=\\
	& \qquad  |(\mathcal{A}\vee\mathcal{B},\alpha \circ \beta)|\,,\\ \,0 & \textup{otherwise}\,. \end{cases} \\
(\mathcal{A},\alpha)\odot (\mathcal{B},\beta)&\coloneqq (\mathcal{A}\vee \mathcal{B},\alpha \circ \beta)\,.
\end{split}
\]
	\item The convolution product and the extended convolution of two functions are:
\begin{equation}
\label{extconv}
\begin{split}
(f_1\ast f_2) (\mathcal{C},\gamma) &\coloneqq\sum\limits_{(\mathcal{A},\alpha)\cdot(\mathcal{B},\beta) = (\mathcal{C},\gamma)} f_1(\mathcal{A},\alpha) \,f_2(\mathcal{B},\beta)\,, \\
(f_1  \circledast f_2) (\mathcal{C},\gamma) &\coloneqq\sum\limits_{(\mathcal{A},\alpha) \odot (\mathcal{B},\beta) = (\mathcal{C},\gamma)} f_1(\mathcal{A},\alpha) \,f_2(\mathcal{B},\beta)\,.
\end{split}
\end{equation}
\end{itemize}
\end{definition}
It is easy to see that these products are associative, but for $d > 1$ they are not commutative. Elementary combinatorics --- see \textit{e.g.}~\cite[Lemma 4.7]{CMSS} --- show that
\begin{equation}
\label{miniABprod} |(\mathcal{A} \vee \mathcal{B},\alpha \circ \beta)| \leq |(\mathcal{A},\alpha)| + |(\mathcal{B},\beta)|\,.
\end{equation}

\begin{definition}
\label{def:PSPmudel}The following functions on $PS(d)$ will play an important role.
\begin{itemize}
\item The \emph{delta function} is:
$$
\delta(\mathcal{A},\alpha) =\begin{cases} \,1 & \textup{if}\,\,\mathcal{A} = \mathbf{0}_{\rm id} \,\,{\rm and}\,\,\alpha = {\rm id}\,, \\ \,0 & \textup{otherwise}\,. \end{cases}
$$ 
\item The \emph{zeta function} is:
\[
\zeta(\mathcal{A},\alpha) \coloneqq \begin{cases} \,1 & \textup{if}\,\, \mathcal{A}=\mathbf{0}_{\alpha}\,, \\ \,0 & \textup{otherwise}\,. \end{cases} \qquad 
\]
The \emph{extended zeta function} is $\zeta_{\hbar}(\mathcal{A},\alpha) \coloneqq \hbar^{|\alpha|}\zeta(\mathcal{A},\alpha)$ and takes values in $R\llbracket \hbar \rrbracket$.
\vspace{0.1cm}
\item The \emph{M\"obius function} $\mu \colon PS(d) \rightarrow R$ and the \emph{extended M\"obius function} $\mu_{\hbar} \colon PS(d) \rightarrow R\llbracket \hbar \rrbracket$ are uniquely determined by
\[
\begin{split}
\mu\ast \zeta & = \zeta\ast \mu = \delta\,, \\
\mu_{\hbar} \circledast \zeta_{\hbar} & = \zeta_{\hbar} \circledast \mu_{\hbar} = \delta\,.
\end{split}
\]
\end{itemize}
\end{definition} 

\begin{lemma}
\label{existe} The M\"obius functions $\mu$ and $\mu_{\hbar}$ exist.
\end{lemma}
\begin{proof} The existence of $\mu$ is known from \cite{CMSS}. For the existence of $\mu_{\hbar}$, we view functions $PS(d) \rightarrow R\llbracket\hbar\rrbracket$ (and their $\circledast$-convolution) as elements of the monoid algebra $R\llbracket\hbar\rrbracket[PS(d)]$ (with product induced by $\odot$). In particular,  $\zeta_{\hbar}$ can be viewed as $\sum_{ \alpha \in S(d)} \hbar^{|\alpha|}\,(\mathbf{0}_{\alpha},\alpha)$. It is of the form $(\mathbf{0}_{\rm id},{\rm id}) + O(\hbar)$ since the only permutation with zero colength is the identity, and $(\mathbf{0}_{\rm id},{\rm id})$ is the unit for $\odot$ thus invertible. Therefore, $\zeta_{\hbar}$ is invertible, and its inverse defines $\mu_{\hbar}$.
\end{proof} 
We give some examples of computations of the extended M\"obius function $\mu_{\hbar}$ in Appendix~\ref{App:Moebius}.

\begin{definition}
\label{def:multifun} A function $f \colon PS \rightarrow R$ is \emph{multiplicative} if  for any  $d \in \mathbb{Z}_{> 0}$ and $\sigma \in S(d)$, $f(\mathbf{1}_d,\sigma)$ depends only on the conjugacy class of $\sigma$, and for any $(\mathcal{A},\alpha) \in PS$:
$$
f(\mathcal{A},\alpha)=\prod_{A \in\mathcal{A}}f(\mathbf{1}_{\#A},\alpha_{|A})\,.
$$
Here $\alpha_{|A}$ is the pre-composition of the restriction of $\alpha$ to $A$ with some bijection $[\# A] \rightarrow A$, the result being independent of the choice  of that bijection due to the first property. In particular, a multiplicative function is determined by its values $f(\mathbf{1}_d,\pi_{\lambda})$ with $d \in \mathbb{Z}_{> 0}$ and $\lambda \vdash d$.
\end{definition}

The convolution of two functions $f_1,f_2 \colon PS \rightarrow R$ is defined by requiring that $(f_1 \ast f_2)_{|PS(d)} = f_1{}_{|PS(d)} \ast f_2{}_{|PS(d)}$ for all $d > 0$, and likewise for $\circledast$. Then, multiplicative functions are stable under convolution (resp. extended convolution), and  the zeta function, the extended zeta function and their corresponding M\"obius functions are multiplicative. Besides, the convolution of two multiplicative functions is commutative:
\begin{lemma}
\label{commutmult} If $f_1,f_2\colon PS \rightarrow R$ are two multiplicative functions, then
\[
f_1 \ast f_2 = f_2 \ast f_1\,,\qquad {\rm and}\qquad f_1 \circledast f_2 = f_2 \circledast f_1\,.
\]
\end{lemma}
\begin{proof}
Let $(\mathcal{A},\alpha),(\mathcal{B},\beta),(\mathcal{C},\gamma) \in PS$ such that $(\mathcal{A},\alpha) \odot (\mathcal{B},\beta) = (\mathcal{C},\gamma)$. Then $\mathcal{C} = \mathcal{A} \vee \mathcal{B}$  and $\alpha \circ \beta = \gamma$. This can also be written $\mathcal{C} = \mathcal{B} \vee \mathcal{A}$ and $\beta^{-1} \circ \alpha^{-1} = \gamma^{-1}$. The support of cycles of a permutation and its inverse are the same: $\mathbf{0}_{\alpha} = \mathbf{0}_{\alpha^{-1}}$,  etc. Therefore $(\mathcal{A},\alpha^{-1}),(\mathcal{B},\beta^{-1}),(\mathcal{C},\gamma^{-1})$ are still partitioned permutations and $(\mathcal{B},\beta^{-1}) \odot (\mathcal{A},\alpha^{-1}) = (\mathcal{C},\gamma^{-1})$. Since $\alpha$ and $\alpha^{-1}$ are conjugated, a multiplicative function takes the same values on $(\mathcal{A},\alpha)$ and on $(\mathcal{A},\alpha^{-1})$. Then, relabelling $(\mathcal{A},\alpha)$ into $(\mathcal{B},\beta^{-1})$ and $(\mathcal{B},\beta)$ into $(\mathcal{A},\alpha^{-1})$ in the definition \eqref{extconv} of the extended convolution shows that $f_1 \circledast f_2 = f_2 \circledast f_1$. The same reasoning works for $\ast$ as $(\mathcal{C},\gamma)$ and $(\mathcal{C},\gamma^{-1})$ have the same colength.
\end{proof}

We give the following property for later use.

\begin{lemma}
\label{lem25} Let $f_{1,\hbar},f_{2,\hbar} \colon PS \rightarrow R\llbracket\hbar\rrbracket$ be two multiplicative functions such that
$$
\forall (\mathcal{A},\alpha) \in PS,\qquad f_{i,\hbar}(\mathcal{A},\alpha) = \hbar^{|(\mathcal{A},\alpha)|}f_i(\mathcal{A},\alpha) + o(\hbar^{|(\mathcal{A},\alpha)|})\,, \qquad i = 1,2\,,
$$
where $f_i\colon PS \rightarrow R$. The relation~$f_{1,\hbar} = \zeta_{\hbar} \circledast f_{2,\hbar}$ implies $f_1 = \zeta \ast f_2$.
\end{lemma}

\begin{proof}
Let us evaluate $f_{1,\hbar}=\zeta_{\hbar}\circledast f_{2,\hbar}$ at $(\mathcal{C},\gamma) \in PS(d)$. In the left-hand side the leading term with respect to $\hbar$ is $\hbar^{|(\mathcal{C},\gamma)|}f_{1}(\mathcal{C},\gamma)$, by definition. Due to the definition of $\zeta_{\hbar}$, only the factorisations of the type $(\mathbf{0}_{\alpha},\alpha) \odot (\mathcal{B},\beta) = (\mathcal{C},\gamma)$ contribute to the right-hand side, and the leading term with respect to $\hbar$ is $\hbar^{|\alpha| + |(\mathcal{B},\beta)|} \zeta(\mathbf{0}_{\alpha},\alpha) f_{2}(\mathcal{B},\beta)$. Recalling \eqref{spcol} we have $|(\mathbf{0}_{\alpha},\alpha)| = |\alpha|$, and together with the subadditivity of colength for the extended product \eqref{miniABprod}, it shows that $|(\mathcal{C},\gamma)| \leq |\alpha| + |(\mathcal{B},\beta)|$, with equality if and only if $(\mathbf{0}_{\alpha},\alpha) \cdot (\mathcal{B},\beta) = (\mathcal{C},\gamma)$. Therefore:
\[
f_{1}(\mathcal{C},\gamma) = \sum\limits_{(\mathbf{0}_{\alpha},\alpha)\cdot(\mathcal{B},\beta)=(\mathcal{C},\gamma)} f_{2}(\mathcal{B},\beta) = (\zeta\ast f_{2})(\mathcal{C},\gamma)\,.
\]
\end{proof}

\subsection{Monotone Hurwitz numbers}
\label{Sec:Hurbase}
Let $r$ be a nonnegative integer. A sequence $\tau_1,\ldots,\tau_r$ of transpositions in $S(d)$ is called \emph{strictly monotone} (resp.~\emph{weakly monotone}) if $\tau_i=(a_i \; b_i)$ with $a_i<b_i$ and the sequence  $(b_i)_{i = 1}^r$ is strictly (resp.~weakly) increasing.

\begin{definition}
Let $\lambda,\nu \vdash d$. The \emph{strictly monotone Hurwitz number} $H^{<}_{r}(\lambda,\nu)$ is $\frac{1}{d!}$ times the number of tuples $(\alpha,\tau_1,\dots,\tau_r,\beta)$ of permutations in $S(d)$ such that:
\begin{itemize}
	\item $\alpha \in C_{\lambda}$ and $\beta \in C_{\nu}$;
	\item $\tau_1,\dots,\tau_r$ is a strictly monotone sequence of transpositions;
	\item $\alpha \circ \tau_1 \circ \cdots \circ \tau_r \circ \beta = \textup{id}$.
\end{itemize}
The \emph{weakly monotone Hurwitz number} $H_r^{\leq}(\lambda,\nu)$ is defined analogously\footnote{Weakly monotone Hurwitz numbers were first introduced and studied in a series of papers by Goulden, Guay-Paquet and Novak \cite{Goulden1,Goulden2,Goulden3} as coefficients in the large-$N$ asymptotic expansion of the Harish-Chandra--Itzykson--Zuber (HCIZ) matrix integral over the unitary group $U(N)$.}. Summing over all the possible numbers of transpositions, we introduce the generating series
\begin{equation}
\label{monoHur}
\begin{split}
H^{<}(\lambda,\nu) & = \sum\limits_{r=0}^{d-1} \hbar^{r} \, H^{<}_{r}(\lambda,\nu) \in \mathbb{Q}[\hbar]\,, \\
H^{\leq}(\lambda,\nu) & = \sum\limits_{r \geq 0} (-\hbar)^r \,H^{\leq}_r(\lambda,\nu) \in \mathbb{Q}\llbracket\hbar\rrbracket\,.
\end{split}
\end{equation}
\end{definition} 

We recall a classical result, including a proof to be self-contained.
\begin{lemma}
\label{Hinverse} For any $d \in \mathbb{Z}_{\geq 0}$ and $\lambda,\nu \vdash d$, we have:
\begin{equation*}
\begin{split}
\sum_{\rho \vdash d} \mathfrak{z}_{\lambda}\,H^{<}(\lambda,\rho) \cdot \mathfrak{z}_{\rho} H^{\leq}(\rho,\nu) & = \delta_{\lambda,\nu}\,, \\
\sum_{\rho \vdash d} \mathfrak{z}_{\lambda}\,H^{\leq}(\lambda,\rho) \cdot \mathfrak{z}_{\rho} H^{<}(\rho,\nu) & = \delta_{\lambda,\nu}\,.
\end{split}
\end{equation*}

\end{lemma}
\begin{proof}
This relation appears when we present monotone Hurwitz numbers using the Jucys--Murphy elements. These are the elements of the group algebra $\mathbb{Q}[S(d)]$ defined by $J_k = \sum_{i = 1}^{k - 1} (i\,k)$. They have the property \cite{Jucys,Murphy} that symmetric polynomials in the $(J_k)_{k = 2}^d$ belong to the center of $\mathbb{Q}[S(d)]$. In particular, we can use them in the $r$-th elementary symmetric and complete symmetric polynomials:
\begin{align*}
\mathsf{e}_r(x_2,\ldots,x_d) & \coloneqq \sum_{2 \leq k_1 < \ldots < k_r \leq d} x_{k_1}\cdots x_{k_r}\,,
\\
 \mathsf{h}_r(x_2,\ldots,x_d) & \coloneqq \sum_{2 \leq k_1 \leq \cdots \leq k_r \leq d} x_{k_1} \cdots x_{k_r}\,.
\end{align*}
By construction,
\begin{align*}
H_{r}^{<}(\lambda,\nu) & = \frac{1}{d!}\cdot [{\rm id}] \,\,C_{\lambda}C_{\nu}\,\mathsf{e}_r(J_2,\ldots,J_d)\,, 
\\
 H_{r}^{\leq}(\lambda,\nu) & = \frac{1}{d!}\cdot [{\rm id}] \,\,C_{\lambda}C_{\nu}\,\mathsf{h}_r(J_2,\ldots,J_d)\,,
\end{align*}
where the operation $[{\rm id}]$ stands for the extraction of the coefficient of ${\rm id}$ in the group algebra, and $C_{\lambda}$ is used to denote the element $\sum_{\sigma \in C_{\lambda}} \sigma$ of the group algebra. For the generating series \eqref{monoHur} over $r$, this gives:
\begin{equation}
\label{Hjuc}\begin{split}
H^{<}(\lambda,\nu) & = \frac{1}{d!} \cdot [{\rm id}] \,\,C_{\lambda} C_{\nu} \prod_{k = 2}^d (1 + \hbar J_k) \,,\\
H^{\leq}(\lambda,\nu)& = \frac{1}{d!} \cdot [{\rm id}]\,\,C_{\lambda} C_{\nu} \frac{1}{\prod_{k = 2}^d (1 + \hbar J_k)}\,.
\end{split}
\end{equation}
Recall that the center of $\mathbb{Q}[S(d)]$ has a basis given by the conjugacy class elements $C_{\lambda}$ for $\lambda \vdash d$. In general, if $B$ is in the center of $\mathbb{Q}[S(d)]$, recalling that $\mathfrak{z}_{\lambda} = \frac{d!}{\# C_{\lambda}}$, we have
\[
\frac{1}{d!}\cdot [{\rm id}]\,\,C_{\lambda}C_{\nu} B  = [C_{\lambda}]\,\, \frac{C_{\nu}B}{\mathfrak{z}_{\lambda}} = [C_{\nu}]\,\,  \frac{C_{\lambda}B}{\mathfrak{z}_{\nu}}\,,
\]
where $[C_{\lambda}]$ extracts the coefficient of $C_{\lambda}$ in the decomposition on the conjugacy class elements. We use this relation to compute for any $\lambda,\nu \vdash d$:
\begin{equation*}
\begin{split}
\delta_{\lambda,\nu} & = \frac{\mathfrak{z}_{\lambda}}{d!}\cdot [{\rm id}] \,\,C_{\lambda}C_{\nu} = \frac{\mathfrak{z}_{\lambda}}{d!}\cdot  [{\rm id}] \,\,C_{\lambda} \prod_{k = 2}^{d} (1 + \hbar J_k)  \cdot C_{\nu} \frac{1}{\prod_{k = 2}^d (1 + \hbar J_k)} \\
& = \frac{\mathfrak{z}_{\lambda}}{d!}\cdot [{\rm id}]\bigg(\sum_{\rho,\rho' \vdash d} H^{<}(\lambda,\rho)\,\mathfrak{z}_{\rho} C_{\rho} \cdot H^{\leq }(\rho',\nu) \mathfrak{z}_{\rho'} C_{\rho'}\bigg) \\
& = \sum_{\rho \vdash d} \mathfrak{z}_{\lambda} H^{<}(\lambda,\rho) \cdot \mathfrak{z}_{\rho} H^{\leq}(\rho,\nu)\,.
\end{split}
\end{equation*} 
The second relation is proved similarly. 
\end{proof}

An alternative description of strictly monotone Hurwitz numbers will later come in handy.
\begin{definition}[See \textit{e.g.}~\cite{ALS16}]\label{def:free:single}
Given $\lambda,\nu \vdash d$, the \emph{free single Hurwitz number} $H^{|}_{r}(\lambda,\nu)$ is $\frac{1}{d!}$ times the number of triples $(\alpha,\psi,\beta)$ of permutations of $[d]$ such that
    \begin{itemize}
        \item $\alpha \in C_{\lambda}$ and $\beta \in C_{\nu}$;
        \item $\psi \in S(d)$ has colength $r$;
        \item $\alpha \circ \psi \circ \beta =\textup{id}$.
    \end{itemize}
In other words, in the group algebra we have 
\begin{equation}
H^|_r(\lambda,\nu)=\frac{1}{d!}\cdot [{\rm id}] \,\,C_{\lambda}C_{\nu} \sum_{\substack{\rho\, \vdash\, d \\ \ell(\rho) = d - r}}C_{\rho}\,,
\end{equation}
where the summation index $\rho$ is the partition associated to $\psi$.
\end{definition} 
The free single Hurwitz number is directly related to the weighted count $D_{r}(\lambda,\nu)$ of (possibly disconnected) dessins d'enfants having $r$ more edges than there are vertices, namely $D_r(\lambda,\nu) = \mathfrak{z}_{\lambda} H^{|}_{r}(\lambda,\nu)$, see \cite[Definition 2.4]{borot2019relating}. We also recall that dessins d'enfants is another name for bipartite maps (again, up to symmetry considerations which in this case amount to factors of $2$ in the weighted count). The Harnad--Orlov correspondence~\cite{HarnadOrlov,HarnadG} (see also \cite[Proposition 4.4]{ALS16}), or the more elementary observation that  every permutation can be uniquely expressed as the product of transpositions along a strictly monotone sequence (see \textit{e.g.}~\cite[Lemma 2.5]{borot2019relating}), yields the following equality.
\begin{lemma}\label{lem:Harnad:Orlov}
We have $H^{<}_r(\lambda,\nu) = H^{|}_r(\lambda,\nu)$.
\end{lemma}

\subsection{Fock space preliminaries}
\label{Sec:TransZ}

Assume now that the commutative ring $R$ is also a $\mathbb{Q}$-algebra. We introduce the algebras of formal series in countably many variables (aka the bosonic Fock space)
$$
\mathcal{F}_{R} \coloneqq \lim_{\substack{\longleftarrow \\ d \in \mathbb{Z}_{\geq 0}}} R[\![p_1,\ldots,p_d]\!]\,,\qquad \mathcal{F}_{R,\hbar} \coloneqq \mathcal{F}_{R(\!(\hbar)\!)} = \lim_{\substack{\longleftarrow \\ d \in \mathbb{Z}_{\geq 0}}} R(\!(\hbar)\!)[\![p_1,\ldots,p_d]\!].
$$
We introduce the vector $\ket \coloneqq 1 \in \mathcal{F}_{R,\hbar}$ and the linear form $\bra \colon \mathcal{F}_{R,\hbar} \rightarrow R(\!(\hbar)\!)$ which extracts the constant term. As an $R(\!(\hbar)\!)$-module, $\mathcal{F}_{R,\hbar}$ admits the Schur (Schauder) basis $s_{\lambda}$, indexed by $\lambda \vdash d$ and $d \in \mathbb{Z}_{\geq 0}$. One of the many equivalent definitions of the Schur elements is in terms of the power sums
$$
\forall \lambda \vdash d, \qquad  s_{\lambda} = \sum_{\mu \vdash d} \chi_{\lambda}(C_{\mu}) \frac{p_{\mu}}{\mathfrak{z}_{\mu}}\,,
$$
where we define $p_{\mu} = p_{\mu_1}\cdots p_{\mu_{\ell(\mu)}}$, for $\mu \vdash d$.
We say that $(i,j)  \in \mathbb{Z}_{> 0}^2$ belongs to $\lambda$ when $i \leq \ell(\lambda)$ and $j \leq \lambda_i$.

\begin{definition}
We define an $R(\!(\hbar)\!)$-linear operator $\mathsf{D}$ on $\mathcal{F}_{R,\hbar}$ by the formula
\begin{equation}
\mathsf{D}\,s_{\lambda} = \prod_{(i,j) \in \lambda} (1 + \hbar(j - i))\, s_{\lambda}\,.
\end{equation}
\end{definition}
Note that $\mathsf{D}$ has a logarithm, namely the operator defined by
\begin{equation}
\label{lnDs}
\ln \mathsf{D}\,s_{\lambda} = \left(\sum_{\substack{(i,j) \in \lambda \\ k \geq 1}}(-1)^{k+1} \frac{\hbar^k(j - i)^k}{k}\right)  s_{\lambda}\,.
\end{equation}
The Schur element corresponding to $\lambda = \emptyset$ is $1 = \ket$, and in this case the sum in \eqref{lnDs} is empty: we have $\ln \mathsf{D} \ket = 0$. As all Schur elements are eigenvectors of $\ln \mathsf{D}$, we also have $\bra \ln \mathsf{D} = 0$.

There is an alternative description of the operator $\mathsf{D}$ which is relevant for us in relation with Hurwitz numbers. The bosonic Fock space enters in the study of Hurwitz numbers via the Frobenius character map, i.e.~the isomorphism of $R(\!(\hbar)\!)$-modules 
\begin{equation}
\label{Frobcharmap}
{\rm ch}\,:\,\begin{array}{ccc}
\prod_{d \geq 0} \mathsf{Z}R(\!(\hbar)\!)[S(d)] & \longrightarrow & \mathcal{F}_{R,\hbar} \\
C_{\lambda} & \longmapsto & p_{\lambda}/\mathfrak{z}_{\lambda} \end{array},
\end{equation}
where the prefix $\mathsf{Z}$ indicates that we take the center of the group algebra. The algebra $\mathcal{F}_{R,\hbar}$ acts on the left-hand side (and thus on itself by transporting with ${\rm ch}$) via evaluation on Jucys--Murphy elements and multiplication in the center of the group algebra. Then:
\begin{lemma}
\label{chDch} The operator ${\rm ch}^{-1} \circ \mathsf{D} \circ {\rm ch}$ on the center of the $S(d)$-group algebra coincides with multiplication by $\prod_{k = 2}^{d} (1 + \hbar J_k)$.
\end{lemma}
This could equally well be taken as definition for $\mathsf{D}$. We already met this combination of Jucys--Murphy elements in the proof of Lemma~\ref{Hinverse}, and the relevance of this operator for us will become clear in Theorem~\ref{thm:hbarstarproduct}. Lemma~\ref{chDch} is a direct consequence of the Harnad--Orlov correspondence\footnote{In notations of \cite[Section 4]{HarnadOrlov}, the operator $\mathsf{D}$ corresponds to $E(\hbar,\mathcal{J})$ and the idempotents $F_{\lambda}$ in the center of the group algebra are sent by the Frobenius character map to the Schur elements $s_{\lambda}$.} \cite{HarnadOrlov,HarnadG}, which is itself based on the classical results of Jucys \cite{Jucys}.  Other formulae for $\mathsf{D}$ are given in \cite[Section 5]{ALS16}.

In the last decade, the full action of $\mathcal{F}_{R,\hbar}$ on itself has been the starting point for the study of a large class of Hurwitz numbers, but in this article only the operator $\mathsf{D}$ and its inverse will be needed.


\subsection{Topological partition functions}
\label{Sec24}
Let $\mathcal{F}_{R}^0 \subset \mathcal{F}_R$ be the subspace of elements without constant term (i.e.~in the kernel of $\bra$), and likewise $\mathcal{F}_{R,\hbar}^0 \subset \mathcal{F}_{R,\hbar}$. Clearly, if $F \in \mathcal{F}_{R,\hbar}^0$, then $e^{F} \in \mathcal{F}_{R,\hbar}$ is well-defined. Recall that if $\lambda \vdash d$, we define $p_{\lambda} = p_{\lambda_1}\cdots p_{\lambda_{\ell(\lambda)}}$. If $\sigma$ is a permutation we also denote $p_{\sigma} \coloneqq p_{\lambda(\sigma)}$. Last, given $F \in \mathcal{F}_{R,\hbar}$, we denote $[p_{\sigma}]\,\, F$ for the coefficient of $p_{\sigma}$ in $F$, with the convention that $p_{\emptyset} = 1$ and $[p_{\emptyset}] \,\,F = \bra F$. Recall the definition
$$
\mathfrak{z}_{\lambda}\coloneqq\prod_{i=1}^{\ell(\lambda)}\, \lambda_i \, \prod_{j\geq 1} m_j(\lambda)!\,,
$$
where $m_j(\lambda)$ is the number of occurrences of $j$ in $\lambda$.

\begin{definition}
\label{def:phi:Z}
A \emph{topological partition function} is an element of $\mathcal{F}_{R,\hbar}$ of the form $Z = e^{F}$, where
$$
F = \left(\sum_{\substack{g \in \mathbb{Z}_{\geq 0} \\ n \in \mathbb{Z}_{> 0}}} \hbar^{2g - 2 + n}F_{g,n}\right) \in \mathcal{F}_{R,\hbar}^0
$$
and each $F_{g,n} \in \mathcal{F}_{R}^0$ is homogeneous of degree $n$ in the variables $(p_k)_{k \geq 1}$, where each $p_k$ is assigned degree $1$. To a topological partition function $Z = e^{F}$, we associate the unique multiplicative function $\Phi_{Z,\hbar} \colon PS \rightarrow R\llbracket\hbar\rrbracket$ such that
\begin{equation}
\label{Fzzz} F = \sum_{\substack{d \geq 1 \\ \lambda \vdash d}} \hbar^{-d} \Phi_{Z,\hbar}(\mathbf{1}_d,\pi_{\lambda})\,\frac{p_{\lambda}}{\mathfrak{z}_{\lambda}}\,.
\end{equation}
\end{definition}
In other words, for $(\mathcal{A},\alpha) \in PS(d)$:
\begin{equation}
\label{eq1comp}\begin{split}
\Phi_{Z,\hbar}(\mathcal{A},\alpha) & \coloneqq  \hbar^{d} \prod\limits_{A\in\mathcal{A}}  [p_{\alpha_{|A}}]\,\,\mathfrak{z}_{\alpha_{|A}}F  \\
& = \hbar^{|(\mathcal{A},\alpha)|}\sum\limits_{g \colon \mathcal{A} \to \mathbb{Z}_{\geq 0}} \hbar^{2\sum_{A \in \mathcal{A}} g(A)} \prod_{A \in \mathcal{A}} [p_{\alpha_{|A}}]\,\,\mathfrak{z}_{\alpha_{|A}} F_{g(A),\ell(\alpha_{|A})} \\
& \coloneqq \hbar^{|(\mathcal{A},\alpha)|} \sum\limits_{g \colon \mathcal{A} \to \mathbb{Z}_{\geq 0}} \hbar^{2\sum_{A \in \mathcal{A}} g(A)}\,\Phi^{[g]}_Z(\mathcal{A},\alpha)\,.
\end{split}
\end{equation}
Here in the second and the third formulae the sum runs over all possible assignments of nonnegative integers $g(A)$ to $A\in\mathcal{A}$. The multiplicative function completely determines the partition function and vice versa. If one prefers to use $Z = e^{F}$ instead of $F$, one decomposes
\begin{equation}
\label{Z1plus}
Z = 1 + \sum_{\substack{ d \geq 1 \\ \lambda \vdash d}} \hbar^{-d} Z(\lambda)\frac{p_{\lambda}}{\mathfrak{z}_{\lambda}}\,,
\end{equation}
and the relation between the coefficients $Z(\lambda)$ and the multiplicative function is
\begin{equation}
\label{eq:Z:phi}
Z(\lambda) = \sum_{\substack{\mathcal{A} \in P(d) \\ \mathbf{0}_{\lambda} \leq \mathcal{A}}} \Phi_{Z,\hbar}(\mathcal{A},\pi_{\lambda})\,.
\end{equation}
The $1$ in \eqref{Z1plus} may be absorbed by allowing the empty partition in the sum while setting $Z(\emptyset) = 1$. Yet another way to access $Z(\lambda)$ is to use the Frobenius character map, following the spirit of Section~\ref{Sec:Hurbase}. Namely, for $\lambda,\mu \vdash d$ we have
$$
\frac{1}{d!} \,[{\rm id}] \,\,C_{\lambda}C_{\mu} =  \frac{\delta_{\lambda,\mu}}{\mathfrak{z}_{\lambda}}\,.
$$
Therefore,
\begin{equation}
\label{ZlambdachZ} Z(\lambda) = \frac{\hbar^{d} \mathfrak{z}_{\lambda}}{d!}\, [{\rm id}] \,\,C_{\lambda}\, {\rm ch}^{-1}(Z)\,.
\end{equation}
In Remark~\ref{RemMMPart} we explain what topological partition functions correspond to in applications to random matrix theory. 

These definitions involve conventions regarding powers of $\hbar$ and prefactors $\mathfrak{z}$, which are motivated by the following guiding principles from analytic combinatorics. First, the definition of topological partition functions via its coefficients should always include the right symmetry factor. The reason to include $\mathfrak{z}$ in \eqref{Fzzz} is the following identity
$$
\sum_{\substack{n \geq 1 \\k_1,\ldots,k_n \geq 1}} \frac{\Phi_{Z,\hbar}(\mathbf{1}_{k_1 + \cdots + k_n},\pi_{\lambda(\mathbf{k})})}{n!\,k_1 \cdots k_n}\,p_{k_1} \cdots p_{k_n}  = \sum_{\substack{d \geq 1 \\ \lambda \vdash d}} \Phi_{Z,\hbar}(\mathbf{1}_d,\pi_{\lambda})\,\frac{p_{\lambda}}{\mathfrak{z}_{\lambda}}\,.
$$
The ``right'' symmetry factor in the left-hand side is the number of ways to label $n$ cycles of respective order $k_1,\ldots,k_n$ as well as their elements, respecting the cyclic order. Second, the power of $\hbar$ in topological partition functions should control the ``Euler characteristic''; the latter being additive under disjoint union, this is compatible with an interpretation of  $F$ as generating series of connected objects, while $Z = e^{F}$ generates disconnected objects. So, we want the contribution of type $(g,n)$ (``genus $g$, $n$ cycles'') to appear in $F$ with a factor $\hbar^{2g - 2 + n}$. For the multiplicative function the power of $\hbar$ is dictated by Lemma~\ref{lem25}: we want $\Phi_{Z,\hbar}(\mathcal{A},\alpha)$ to have leading order $\hbar^{|(\mathcal{A},\alpha)|}$. In particular, due to $|(\mathbf{1}_d,\pi_{\lambda})| = d - 2 + \ell(\lambda)$ and \eqref{eq1comp}, the $(g,n)$-part contributes to $\Phi_{Z,\hbar}(\mathbf{1}_d,\pi_{\lambda})$ with an $\hbar^{2g - 2 + \ell(\lambda) + d}$. To respect the principle, we killed the $\hbar^{d}$ in \eqref{eq1comp} to define $F$ in \eqref{Fzzz}.

\subsection{The master relation and its avatars}

We say that two topological partition functions $Z_1,Z_2$ satisfy the master relation if  $Z_1 = \mathsf{D}Z_2$. The key observation for us is that the master relation expresses nothing but the convolution by the zeta function for the associated multiplicative functions. We shall prove this by connecting both to monotone Hurwitz numbers.

\begin{theorem}\label{thm:hbarstarproduct} Consider two topological partition functions $Z_1$ and $Z_2$. The following four properties, each required for every $d > 0$, are equivalent:
\begin{itemize}
\item[(i)] $Z_1(\lambda) = \mathfrak{z}_{\lambda}\sum_{\nu\vdash d} H^{<}(\lambda,\nu) Z_2(\nu)$ holds for any $\lambda \vdash d$;
\item[(ii)] $\Phi_{Z_1,\hbar} = \zeta_{\hbar} \circledast \Phi_{Z_2,\hbar}$ holds between functions on $PS(d)$;
\item[(iii)] $Z_2(\nu) = \mathfrak{z}_{\nu} \sum_{\lambda \vdash d} H^{\leq}(\nu,\lambda) Z_1(\lambda)$ holds for any $\nu \vdash d$;
\item[(iv)] $\Phi_{Z_2,\hbar} = \mu_{\hbar} \circledast \Phi_{Z_1,\hbar}$ holds between functions on $PS(d)$.
\end{itemize} 
Besides, these properties are equivalent to $Z_1 = \mathsf{D}Z_2$.
\end{theorem}
\begin{corollary}
\label{thm:starproduct}\label{thm:dual:starproduct} Let $d \in \mathbb{Z}_{>0}$. If one of the four conditions above holds, then the relations $\Phi_{Z_1}^{[0]} = \zeta \ast \Phi_{Z_2}^{[0]}$ and $\Phi_{Z_2}^{[0]} = \mu \ast \Phi_{Z_1}^{[0]}$ between the genus-zero parts of the functions on $PS(d)$ hold.
\end{corollary}

\begin{remark} Due to Lemma~\ref{commutmult}, condition (ii) is also equivalent to $\Phi_{Z_1,\hbar} = \Phi_{Z_2,\hbar} \circledast \zeta_{\hbar}$, etc.
\end{remark}

\begin{proof}[Proof of Theorem~\ref{thm:hbarstarproduct}]

We first prove (ii) $\Rightarrow$ (i).  Let $\lambda \vdash d$ be a partition of $d$. We start from the equality $\Phi_{Z_1,\hbar}=\zeta_{\hbar}\circledast \Phi_{Z_2,\hbar}$, which, by Equation \eqref{eq:Z:phi}, is equivalent to:
\begin{equation*}
\begin{split}
Z_1(\lambda) &= \sum_{\substack{\mathcal{C}\in P(d) \\ \mathbf{0}_{\lambda} \leq \mathcal{C}}} (\zeta_{\hbar} \circledast \Phi_{Z_2,\hbar})(\mathcal{C},\pi_{\lambda}) \\
&=\sum_{\substack{\mathcal{C}\in P(d) \\ \mathbf{0}_{\lambda} \leq \mathcal{C}}} \,\,\sum_{(\mathbf{0}_\alpha,\alpha) \odot (\mathcal{B},\beta)=(\mathcal{C},\pi_{\lambda})} \hbar^{|\alpha|}\,\Phi_{Z_2,\hbar}(\mathcal{B},\beta)  \\ 
& = \sum_{\substack{\alpha,\beta \in S(d) \\ \alpha \circ \beta = \pi_{\lambda}}}  \sum\limits_{\substack{\mathcal{B}\in P(d)\\ \mathbf{0}_{\beta} \leq \mathcal{B}}} \hbar^{|\alpha|} \Phi_{Z_2,\hbar}(\mathcal{B},\beta) \\
& = \sum_{\nu \vdash d} \sum_{\alpha \in S(d)} \hbar^{|\alpha|} \sum_{\substack{\beta \in C_{\nu} \\ \alpha \circ \beta = \pi_{\lambda} }} \Bigg(\sum_{\substack{\mathcal{B} \in P(d) \\ \mathbf{0}_{\beta} \leq \mathcal{B}}} \Phi_{Z_2,\hbar}(\mathcal{B},\beta)\Bigg)\,.
\end{split} 
\end{equation*}
From the second line to the third, we just used the definition of $\odot$. In the last line we clustered the sum by the conjugacy class $C_{\nu}$ to which $\beta = \alpha^{-1}\circ\pi_{\lambda}$ belongs. By multiplicativity of $\Phi_{Z_2,\hbar}$, the sum inside the brackets only depends on the cycle structure of $\beta$. In particular, substituting $\beta$ with $\pi_{\nu}$ does not change the sum, and by comparing with \eqref{eq:Z:phi} we recognise the value of $Z_2(\nu)$. This yields:
\[
Z_1(\lambda) = \sum_{\nu \vdash d} \sum_{\alpha \in S(d)} \hbar^{|\alpha|} \sum_{\substack{\beta \in C_{\nu} \\ \alpha \circ \beta = \pi_{\lambda}}} Z_2(\nu)  = \sum_{\nu \vdash d} Z_2(\nu) \Bigg(\sum_{\substack{\alpha \in S(d), \beta\in C_{\nu} \\ \alpha \circ \beta = \pi_{\lambda}}} \hbar^{|\alpha|} \Bigg)\,.
\]
The last constraint can be written $\pi_{\lambda}^{-1} \circ \alpha \circ \beta = {\rm id}$. By comparison with Definition~\ref{def:free:single}, we recognise the free single Hurwitz numbers:
\[
Z_1(\lambda) = \mathfrak{z}_{\lambda} \sum_{\nu \vdash d} Z_2(\nu) \bigg( \sum_{r = 0}^{d - 1} \hbar^{r} H^{|}_r(\lambda,\nu)\bigg)\,.
\]
Here, $r$ is the colength of $\alpha$, and the factor $\mathfrak{z}_{\lambda} = \frac{d!}{\# C_{\lambda}}$ is explained as follows. The numerator compensates the $\frac{1}{d!}$ in the definition of Hurwitz numbers. The denominator comes from the fact that in the definition of free single Hurwitz numbers, we let the leftmost permutation be any element of the conjugacy class $C_{\lambda}$, so we overcount by a factor of $\# C_{\lambda}$.
Thanks to Lemma~\ref{lem:Harnad:Orlov} and by comparison with the definition of the generating series of strictly monotone Hurwitz numbers in \eqref{monoHur}, we get:
\[
{\rm (i)}\,\,:\,\,Z_1(\lambda) = \mathfrak{z}_{\lambda} \sum_{\nu \vdash d} H^{<}(\lambda,\nu)Z_2(\nu)\,.
\]
As all steps are equivalences, this in fact proves (i) $\Leftrightarrow$  (ii). 
 
 \medskip
 
Extended convolution from the left by the M\"obius function $\mu_{\hbar}$ proves (ii) $\Rightarrow$ (iv). The same operation with the zeta function $\zeta_{\hbar}$ proves the converse direction. The implication (i) $\Rightarrow$ (iii) is obtained by multiplying (i) by $\mathfrak{z}_{\nu} H^{\leq}(\nu,\lambda)$, summing over $\lambda \vdash d$ and using the second line of Lemma~\ref{Hinverse}, while the converse direction is obtained likewise using the first line of Lemma~\ref{Hinverse}. This finishes the proof of all equivalences between (i),(ii),(iii),(iv).

The equivalence between (i) for all $d > 0$ and $Z_1 = \mathsf{D} Z_2$ is a direct consequence of the Harnad--Orlov correspondence. Let us give the details for completeness. Take two topological partition functions $Z_1,Z_2$ related by $Z_1 = \mathsf{D}Z_2$. From the definition of Frobenius character map in \eqref{Frobcharmap} and  \eqref{ZlambdachZ} for $Z_1$, we have for any $d \geq 0$ and $\lambda \vdash d$ 
\begin{equation}
\label{Z1Z2Z1Z2}\begin{split}
Z_1(\lambda)  & = \frac{\hbar^{d}\mathfrak{z}_{\lambda}}{d!}\, [{\rm id}]\,\, C_{\lambda} \cdot {\rm ch}_d^{-1}(Z_1) \\
& = \frac{\hbar^{d}\mathfrak{z}_{\lambda}}{d!}\, [{\rm id}] \,\,C_{\lambda}  \prod_{k = 2}^{d} (1 + \hbar J_k)  \cdot {\rm ch}_d^{-1}(Z_2) \\
& = \sum_{\nu \vdash d} \frac{\mathfrak{z}_{\lambda}}{d!}\, [{\rm id}]\,\, C_{\lambda} \prod_{k = 2}^{d} (1 + \hbar J_k)  \cdot Z_2(\nu) C_{\nu} \\
& = \mathfrak{z}_{\lambda} \sum_{\nu \vdash d} H^{<}(\lambda,\nu) Z_2(\nu)\,.
\end{split}
\end{equation}
Here ${\rm ch}^{-1}_d$ is the restriction of ${\rm ch}^{-1}$ in \eqref{Frobcharmap} to the $d$-th component. In the second line we used Lemma~\ref{chDch}, in the third line again \eqref{ZlambdachZ} for $Z_2$, and in the last line the definition of the strictly monotone double Hurwitz numbers in \eqref{Hjuc}. Since the collection of coefficients $(Z(\lambda))_{\lambda}$ uniquely determines topological partition functions, the last line of \eqref{Z1Z2Z1Z2} implies conversely $Z_1 = \mathsf{D}Z_2$.
\end{proof}

\begin{proof}[Proof of Corollary~\ref{thm:starproduct}] In light of Lemma~\ref{lem25} and the properties of the M\"obius function for $\ast$, this is a direct consequence of the theorem we just proved.
\end{proof}

\section{The functional relations}
\label{Sec3}

\subsection{\texorpdfstring{$n$-point functions}{n-point functions}}
\label{npointSec}
The bosonic Fock space $\mathcal{F}_R$ (and so $\mathcal{F}_{R,\hbar}$) is acted upon by the Heisenberg operators\footnote{They should not be confused with the Jucys--Murphy elements $J_k$ used in the proof of Lemma~\ref{Hinverse}.}
\[
\mathsf{J}_k = \begin{cases} k\partial_{p_k} & \textup{if}\,\,k > 0\,, \\ 0 & \textup{if}\,\,k = 0\,, \\ p_{-k} & \textup{if}\,\,k < 0\,. \end{cases}
\]
We collect them in generating series
\[
\mathsf{J}(X) = \sum_{k > 0} X^{k} \mathsf{J}_k\,,\qquad \widetilde{J}(X) = \sum_{k \in \mathbb{Z}} X^{k} \mathsf{J}_k\,.
\]
A topological partition function $Z = e^{F} \in \mathcal{F}_{R,\hbar}$ as in Definition~\ref{def:phi:Z} can be decomposed as
\begin{equation}
\label{eq2comp} 
\begin{split}
F & =  \sum_{\substack{n \geq 1 \\ g \geq 0}} \frac{\hbar^{2g - 2 + n}}{n!} \sum_{k_1,\ldots,k_n > 0} F_{g;k_1,\ldots,k_n} \prod_{i = 1}^n \frac{p_{k_i}}{k_i} 
\\
& = \sum_{\substack{d \geq 1 \\ g \geq 0}} \sum_{\lambda \vdash d} \hbar^{2g - 2 + \ell(\lambda)}\,F_{g;\lambda_1,\ldots,\lambda_{\ell(\lambda)}}\,\frac{p_{\lambda}}{\mathfrak{z}_{\lambda}}\,,
\end{split}
\end{equation}
with coefficients $F_{g;k_1,\ldots,k_n} \in R$ that are symmetric under permutation of the $k_i$s and related to the $F_{g,n}s$ of Definition~\ref{def:phi:Z} by
$$
F_{g,n}=\frac{1}{n!}\sum_{k_1,\ldots,k_n > 0} F_{g;k_1,\ldots,k_n} \prod_{i = 1}^n \frac{p_{k_i}}{k_i}\,.
$$ 
The operator
\[
\mathsf{F} = \sum_{\substack{n \geq 1 \\ g \geq 0}} \frac{\hbar^{2g - 2 + n}}{n!} \sum_{k_1,\ldots,k_n > 0}  F_{g;k_1,\ldots,k_n} \prod_{i = 1}^n \frac{\mathsf{J}_{-k_i}}{k_i}
\]
is such that $Z = e^{\mathsf{F}} \ket$. For every $n \in \mathbb{Z}_{> 0}$, we define the $n$-point functions $G_n$ and their shifted version $\widetilde{G}_n$:
\begin{equation}
\label{Gntildes}\begin{split}
G_n(X_1,\ldots,X_n) & = \hbar^{-1}\delta_{n,1} + \Bra  \prod_{i = 1}^n \mathsf{J}(X_i) \cdot e^{\mathsf{F}} \Ket^{\circ}\,, \\
\widetilde{G}_n(X_1,\ldots,X_n) & = \hbar^{-1}\delta_{n,1} + \Bra \prod_{i = 1}^n \widetilde{\mathsf{J}}(X_i) \cdot e^{\mathsf{F}} \Ket^{\circ}\,.
\end{split}
\end{equation}
Here, $\bra \cdots \ket{}^{\circ}$ refers to the connected expectation value, defined for any tuple of linear operators $(\mathsf{A}_i)_{i = 1}^n$ by
\[
\bra \mathsf{A}_1 \cdots \mathsf{A}_n \cdot e^{\mathsf{F}} \ket^{\circ} \coloneqq \partial_{t_1 = 0} \cdots \partial_{t_n = 0} \ln\Big( \bra e^{t_1\mathsf{A}_1} \cdots e^{t_n \mathsf{A}_n} \,e^{\mathsf{F}} \ket\Big)\,.
\]
Equivalently, we have the inclusion-exclusion formulae:
\begin{equation*}
\begin{split}
\bra \mathsf{A}_1 \cdots \mathsf{A}_n\,e^{\mathsf{F}} \ket & = \sum_{\mathcal{I} \in P(n)}   \prod_{I \in \mathcal{I}} \Bra \prod_{i \in I} \mathsf{A}_i \cdot e^{\mathsf{F}}\Ket^{\circ}\,, \\
\bra \mathsf{A}_1 \cdots \mathsf{A}_n \cdot e^{\mathsf{F}} \ket^{\circ} & = \sum_{\mathcal{I} \in P(n)} (-1)^{\# \mathcal{I} - 1}\,(\# \mathcal{I} - 1)!\,\prod_{I \in \mathcal{I}} \Bra \prod_{i \in I} \mathsf{A}_i \, e^{\mathsf{F}} \Ket\,. 
\end{split}
\end{equation*}
In concrete terms, we have
\begin{equation}
\label{npointfun}\begin{split}
G_n(X_1,\ldots,X_n) & = \hbar^{-1}\delta_{n,1} + \sum_{g \geq 0} \sum_{k_1,\ldots,k_n > 0} \hbar^{2g - 2 + n}\,F_{g;k_1,\ldots,k_n} X_1^{k_1} \cdots X_n^{k_n}\,, \\
\widetilde{G}_n(X_1,\ldots,X_n) & =  G_n(X_1,\ldots,X_n) + \delta_{n,2}\,\frac{X_1X_2}{(X_1 - X_2)^2}\,.
\end{split}
\end{equation}
The second equation can be obtained by elementary manipulations with the Heisenberg commutation relations, see \cite[Proposition 4.1]{BDBKSexplicit}. Collecting the coefficients of powers of $\hbar$, we obtain a decomposition
\begin{equation}
\label{giGn} G_{n} = \sum_{g \geq 0} \hbar^{2g - 2 + n}\,G_{g,n} = \hbar^{n - 2}\,G_{0,n} + o(\hbar^{n - 2})\,,
\end{equation}
and likewise for $\widetilde{G}$. 

\begin{remark} \label{rem:div} Note that $G_n$ for all $n$ and $\widetilde G_n$ for $n\neq2$ are formal power series in $X_1,\ldots,X_n$, while $\widetilde G_2$ should be considered as an element of $R\llbracket\hbar\rrbracket(\!(X_2)\!)\llbracket X_1\rrbracket$, corresponding to the expansion in the sector $|X_1|<|X_2|$. More generally, series in $n$ variables of this type are understood as expansions in the sector $|X_1|<\cdots<|X_n|$, i.e. as elements of $R(\!(\hbar)\!)(\!(X_n)\!)\cdots(\!(X_1)\!)$.
\end{remark}

To a multiplicative function $f_{\hbar} \colon PS \rightarrow R\llbracket\hbar\rrbracket$ such that
\[
\forall (\mathcal{A},\alpha) \in PS,\qquad f_{\hbar} (\mathcal{A},\alpha) \in \hbar^{|(\mathcal{A},\alpha)|} R\llbracket\hbar^{2}\rrbracket\,,
\]
we can associate a topological partition function $Z$ such that $f_{\hbar}$ coincides with $\Phi_{Z,\hbar}$ constructed in Section~\ref{Sec24}, namely
$$
Z = \exp\left( \sum_{\substack{d \geq 1 \\ \lambda \vdash d}}  \hbar^{-d} f_{\hbar}(\mathbf{1}_d,\pi_{\lambda}) \frac{p_{\lambda}}{\mathfrak{z}_{\lambda}}\right).
$$
To such a multiplicative function $f_{\hbar}$, we can also associate a collection of $n$-point functions \eqref{npointfun} whose coefficients are given for any $g \geq 0$ and $\lambda \vdash d$ of length $n$ by the formula
\[
F_{g;\lambda_1,\ldots,\lambda_{n}} =  [\hbar^{2g - 2 + n + d}]\,\,f_{\hbar}(\mathbf{1}_{d},\pi_{\lambda})\,.
\]
We invite the reader to compare with \eqref{eq1comp}-\eqref{eq:Z:phi} and \eqref{eq2comp}.

In particular, given a multiplicative function $f \colon PS \rightarrow R$, we can put ourselves in the previous situation by defining $f_{\hbar} = \hbar^{{\rm colength}}f$, and thus associate to it $n$-point functions \eqref{npointfun}, in which  the $g > 0$ sector is zero and:
\[
F_{0;\lambda_1,\ldots,\lambda_{n}} = f(\mathbf{1}_{d},\pi_{\lambda})\,.
\]

\subsection{Main formulae}

Topological partition functions, multiplicative functions and collections of $n$-point functions are different ways to encode the same information. In Theorem~\ref{thm:hbarstarproduct} we have given equivalent forms, in terms of multiplicative functions, of the relation $Z = \mathsf{D} Z^{\vee}$ between two topological partition functions $Z$ and $Z^{\vee}$. Our aim is now to translate this relation into functional relations between their respective $n$-point functions $G_n$ and $G_n^{\vee}$. The result is expressed as weighted sums over graphs and the formal power series:
\begin{equation}
\label{varsigmaw} \varsigma(w) = \frac{{\rm sinh}(w/2)}{w/2} = 1 + \frac{w^2}{24} + O(w^4)\,.
\end{equation}

\begin{definition}[Graphs]\label{def:graphs}
If $n > 0$, we let $\mathcal{G}_n$ be the set of connected bicoloured graphs such that
\begin{itemize}
\item the white vertices are labelled from $1$ to $n$;
\item edges only connect vertices of different colour;
\item black vertices have valency $\geq 2$.
\end{itemize}
\end{definition}
 If $n = 1$, the set of black vertices can be empty (for $n > 1$ connectedness requires it is not empty). Black vertices are characterised by multisets\footnote{A multiset $I$ in $[n]$ is a function $f_I\colon [n] \rightarrow \mathbb{Z}_{\geq 0}$. We say that $i$ is an element of $I$ when $f_I(i) > 0$, but elements may have multiplicity $f_I(i)$ greater than $1$. The cardinality is defined to take into account the multiplicity: $\# I = \sum_{i = 1}^n f_I(i)$.}  $I$ in $[n]$ --- also called \emph{hyperedges} --- recording the white vertices they connect to. The last condition implies $\# I \geq 2$.  If $\Gamma \in \mathcal{G}_n$, we denote $\mathcal{I}(\Gamma)$ its multiset\footnote{Hyperedges can appear with multiplicity in $\mathscr{I}(\Gamma)$, since distinct black vertices may connect in the same way to white vertices.} of hyperedges and ${\rm Aut}(\Gamma)$ the automorphism group, consisting of permutations of the edges respecting the structure of $\Gamma$ and the labelling of white vertices. The three conditions imply that the set of graphs in $\mathcal{G}_n$ is infinite, but it is finite if we fix their first Betti number $b_1$. Indeed, there are $\sum_{I \in \mathcal{I}(\Gamma)} \# I$ edges, and $n + \# \mathcal{I}(\Gamma)$ is the total number of vertices, so that we have
$$
1 - b_1(\Gamma)  = n + \# \mathcal{I}(\Gamma) - \# \text{Edge}(\Gamma) = n  - \sum_{I \in \mathcal{I}(\Gamma)} (\# I - 1)
$$
and therefore
\begin{equation}
\label{nmoins0} \# \mathcal{I}(\Gamma) \leq \sum_{I \in \mathcal{I}(\Gamma)} (\# I - 1) = n + b_1(\Gamma) - 1\,.
\end{equation}
Note that for trees, $b_1(\Gamma) = 0$ and we have
\begin{equation}
\label{nmoins1}
n - 1 = \sum_{I \in \mathcal{I}(\Gamma)} (\# I - 1)\,,
\end{equation}
so $n - 1$ is an upper bound for the number of hyperedges.

\begin{example} Here we give a few examples of such graphs:
	\[
\vcenter{
	\xymatrix@!C=5pt@R=30pt{
		*+[o][F-]{{1}}    \\
		 *+[o][F**]{}\ar@{-}@/^/[u]\ar@{-}@/_/[u] \\
	}
}
\vcenter{
	\xymatrix@!C=5pt@R=30pt{
		 & & *+[o][F-]{{1}}  &    \\
		 & *+[o][F**]{}\ar@{-}@/^/[ru] \ar@{-}@/_/[ru]&  *+[o][F**]{}\ar@{-}@/^/[u]\ar@{-}@/_/[u] & *+[o][F**]{}\ar@{-}[lu] \ar@{-}@/_/[lu]  \ar@{-}@/^/[lu] \\
	}
}
\vcenter{
	\xymatrix@!C=5pt@R=30pt{
		& & *+[o][F-]{{1}}  & *+[o][F-]{{2}} &   \\
		&  *+[o][F**]{}\ar@{-}@/^/[ru] \ar@{-}[ru] \ar@{-}[rru] \ar@{-}@/^/[rru]&  *+[o][F**]{}\ar@{-}@/^/[u]\ar@{-}[u]\ar@{-}@/_/[u] & *+[o][F**]{}\ar@{-}[lu] \ar@{-}[u]& *+[o][F**]{}\ar@{-}[llu] \ar@{-}[lu] \\
	}
}
\vcenter{
			\xymatrix@!C=5pt@R=30pt{
				& *+[o][F-]{{1}} & *+[o][F-]{{2}}  & *+[o][F-]{{3}} & *+[o][F-]{{4}}   \\
				& *+[o][F**]{}\ar@{-}[u] \ar@{-}[ru]\ar@{-}[rru]&  *+[o][F**]{}\ar@{-}[lu]\ar@{-}[u]\ar@{-}@/^/[rru] & *+[o][F**]{}\ar@{-}[lu] \ar@{-}[u]& *+[o][F**]{}\ar@{-}[llu] \ar@{-}[lu] \\
			}
		}
	\]
\medskip


The automorphism factors of these graphs are $\# {\rm Aut} =2$ for the first and the fourth ones and $\# {\rm Aut} =48$ for the second and the third ones.
\end{example}

\vspace{0.3cm}

\begin{definition}[Weights] \label{def:weights} Let $Z^{\vee}$ be a topological partition function, with $n$-point functions $G_n^{\vee}$ and shifted $n$-point functions $\widetilde{G}_n^{\vee}$, like in \eqref{npointfun}. Let $(w_i)_{i = 1}^n$ be an $n$-tuple of variables, and if $I$ is a multiset in $[n]$, denote $w_I = (w_i)_{i \in I}$ the corresponding collection of variables with multiplicity.
\begin{itemize}
\item To a hyperedge $I$ of $[n]$ that is not of the form $I = \{j,j\}$, we assign the weight
\[
\mathsf{c}^{\vee}(u_I,w_I) =\Big( \prod_{i \in I} \hbar u_i \,\varsigma(\hbar u_i w_i \partial_{w_i})\Big) \widetilde{G}^{\vee}_{\# I}(w_{I})\,.
\]
These are series depending on variables $(u_i,w_i)_{i \in I}$, and following Remark~\ref{rem:div} they are only considered in the sector $|w_i| < |w_j|$ for $i < j$.
\item For a hyperedge of the form $I = \{j,j\}$, the above expression would be ill-defined due to the double pole in the shifted $2$-point function \eqref{npointfun}. We rather assign the weight:
\[
\mathsf{c}^{\vee}(u_I,w_I) = \big(\hbar u_j \,\varsigma(\hbar u_j w_j \partial_{w_j})\hbar u_j \,\varsigma(\hbar u_j w'_j \partial_{w'_j})\big)\,G_{2}^{\vee}(w_j,w'_j)\Big|_{w_j'=w_j}\,.
\]
\item To the $i$-th white vertex, we attach the operator weight
\begin{equation}
\label{O2Xi}
\begin{split}
 \vec{\mathsf{O}}^{\vee}(w_i) & = \sum_{m \geq 0} \big(P^{\vee}(w_i) w_i\partial_{w_i}\big)^{m} P^{\vee}(w_i) \\
 & \quad \cdot [v_i^{m}]\,\,\sum_{r \geq 0} \Big(\partial_{y} + \frac{v_i}{y}\Big)^{r} \exp\bigg(v_i\,\frac{\varsigma(\hbar v_i \partial_{y})}{\varsigma(\hbar \partial_{y})} \ln y - v_i \ln y\bigg) \Big|_{y = G_{0,1}^{\vee}(w_i)} \\
& \quad \cdot [u_i^{r}]\,\,\frac{\exp\big(\hbar u_i \varsigma(\hbar u_i w_i \partial_{w_i})(G_{1}^{\vee}(w_i) - \hbar^{-1}) - u_i(G_{0,1}^{\vee}(w_i) - 1)\big)}{\hbar u_i\,\varsigma(\hbar u_i)}  \,,
\end{split}
\end{equation}
where $P^{\vee}(w_i)$ is a power series specified later. As $y = G_{0,1}^{\vee}(w_i) = 1 + O(w_i)$, $\ln y $ is a well-defined power series in $w_i$. The $-\hbar^{-1}$ and $-1$ in the last line cancel the conventional constant added in the definition of the $1$-point function \eqref{npointfun}. This operator acts from the left on series depending on the variables $u_i,w_i$ and gives as output a series in $w_i$. 
\end{itemize}
\end{definition}

\begin{theorem}\label{thm:R-transform-HigherGenera} Let $Z,Z^{\vee}$ be two topological partition functions and $G_{g,n},G_{g,n}^{\vee}$ the generating series appearing in the topological expansion of their respective $n$-point functions, cf. Section~\ref{npointSec}. Suppose that $Z = \mathsf{D}Z^{\vee}$. Then, under the substitutions
\[
X_i = \frac{w_i}{G_{0,1}^{\vee}(w_i)}\,,\qquad P^{\vee}(w_i) =  \frac{\dd \ln w_i}{\dd \ln X_i}\,,
\]
we have
\begin{equation}
\label{0102}\begin{split}
G_{0,1}(X_1) & = G^{\vee}_{0,1}(w_1)\,, \\
G_{0,2}(X_1,X_2) & = P^{\vee}(w_1)P^{\vee}(w_2)\bigg(G_{0,2}^{\vee}(w_1,w_2) + \frac{w_1w_2}{(w_1 - w_2)^2}\bigg) 
\\ & \quad - \frac{X_1X_2}{(X_1 - X_2)^2}\,,
\end{split}
\end{equation}
for $2g - 2 + n > 0$:
\begin{equation}
\label{Ggnpass}
\begin{aligned}
	& G_{g,n}(X_1,\ldots,X_n) =
	\\
	 & \quad \delta_{n,1}\Delta^{\vee}_g(X_1) + [\hbar^{2g - 2 + n}]\,\,\sum_{\Gamma \in \mathcal{G}_n}  \frac{1}{\# {\rm Aut}(\Gamma)} \prod_{i = 1}^n \vec{\mathsf{O}}^{\vee}(w_i) \prod_{I \in \mathcal{I}(\Gamma)} \mathsf{c}^{\vee}(u_I,w_I)\,.
\end{aligned}
\end{equation}
The correction term appearing for $n = 1$ is:
\begin{equation}
\label{Ggnpass-n1}
\begin{split}
	& \Delta_{g}^{\vee}(X)  = \\ 
	& \quad [\hbar^{2g}]\,\,\sum_{m\geq 0} \big(P^{\vee}(w) w\partial_{w}\big)^{m} \,
  \\ 
 & \quad [v^{m+1}] \exp\bigg(v\,\frac{\varsigma(\hbar v \partial_{y})}{\varsigma(\hbar \partial_{y})} \ln y - v \ln y\bigg) \Big|_{y = G_{0,1}^{\vee}(w)}
  P^{\vee}(w) w\partial_{w}  G_{0,1}^{\vee}(w)
\,.
\end{split}
\end{equation}
\end{theorem}
\begin{remark}
\label{special02}Equation~\eqref{Ggnpass} remains valid for $(g,n) = (0,2)$ provided the left-hand side is replaced with $\widetilde{G}_{0,2}(X_1,X_2)$. This recovers the second equation in \eqref{0102}.
\end{remark}
\begin{remark}
	\label{special01}Equation~\eqref{Ggnpass} remains also valid for $(g,n) = (0,1)$ provided we extend the summation over $m$ to $m\geq -1$ in the definition of $\Delta_{0}^{\vee}(X)$ and identify $ \big(P^{\vee}(w) w\partial_{w}\big)^{-1} P^{\vee}(w) w\partial_{w} G_{0,1}^{\vee}(w)$ with $G_{0,1}^{\vee}(w)$. This is the only contribution (the sum over graphs does not contribute as it contains only nonnegative powers of $\hbar$). This recovers the first equation in \eqref{0102}.
\end{remark}
This theorem is a special case of~\cite[Theorem 4.14 and Remark 4.15]{BDBKS}, or more precisely, it is a special case of the statement made in the first line of the proof of~\cite[Theorem 4.14]{BDBKS}. For completeness, and since it is of crucial importance for us, we shall explain its proof in Section~\ref{proofg}.

There are several simplifications in the genus-zero sector. We are going to present the result in terms of multiplicative functions.
\begin{definition}
\label{deftree} If $\mathbf{r} = (r_1,\ldots,r_n) \in \mathbb{Z}_{\geq 0}^n$, let $\mathcal{G}_{0,n}(\mathbf{r} + 1)$ be the subset of $\mathcal{G}_n$ consisting of trees in which the $i$-th vertex has valency $r_i + 1$.
\end{definition}
Note that such trees do not have a non-trivial automorphism group. Observe as well that for fixed $n$, the set $\mathcal{G}_{0,n}(\mathbf{r} + 1)$ is non-empty only for finitely many $n$-tuples $\mathbf{r}$. Indeed, the number of edges is $\sum_{I \in \mathcal{I}(\Gamma)} \# I = \sum_{i = 1}^{n} (r_i + 1)$ and by \eqref{nmoins1} it is equal to $n - 1 + \# \mathcal{I}(\Gamma) \leq 2n - 2$.
\begin{definition}
\label{Or2}Let us introduce the $r$-th piece of the genus-zero version of the operator weight of Definition~\ref{def:weights}:
$$
\vec{\mathsf{O}}_r^{\vee}(w) = \sum_{m \geq 0} (P^{\vee}(w) w\partial_{w})^{m} P^{\vee}(w) \cdot [v^m] \,\,\Big(\partial_{y} + \frac{v}{y}\Big)^{r} \cdot 1\Big|_{y = G_{0,1}^{\vee}(w)}\,.
$$
This is an operator acting from the left on series in the variable $w$. 
\end{definition}
\begin{theorem}
\label{thm:R-transform-GenusZero} Let $f,f^{\vee}\colon PS \rightarrow R$ be multiplicative functions and $G_{0,n},G_{0,n}^{\vee}$ their respective $n$-point functions. Suppose that $f= \zeta \ast f^{\vee}$. Then, under the substitutions
\[
X_i = \frac{w_i}{G_{0,1}^{\vee}(w_i)}\,,\qquad P^{\vee}(w_i) =  \frac{\dd \ln w_i}{\dd \ln X_i}\,,
\]
we have \eqref{0102} for $n = 1,2$, and for any $n \geq 3$:
\[
G_{0,n}(X_1,\ldots,X_n) = \sum_{r_1,\ldots,r_n \geq 0} \prod_{i = 1}^{n} \vec{\mathsf{O}}_{r_i}^{\vee}(w_i) \, \sum_{T \in \mathcal{G}_{0,n}(\mathbf{r} + 1)}  \prod_{I \in \mathcal{I}(T)}' G^{\vee}_{0,\# I}(w_I)\,,
\]
where $\prod'$ means that one should replace each occurrence of $G_{0,2}^{\vee}(w_i,w_j)$ for $i \neq j$ with $\widetilde{G}_{0,2}^{\vee}(w_i,w_j)$. 
\end{theorem}
This will be proved in Section~\ref{proofg0}.

\begin{remark} The formula for $(g,n) = (0,1)$ corresponds to Voiculescu's $R$-transform and was derived from the combinatorics of non-crossing partitions in \cite{Spei94}. For $(g,n) = (0,2)$ it was derived from the combinatorics of non-crossing partitioned permutations in \cite{CMSS}, and obtained differently from the combinatorics of fully simple maps in \cite{BG-F18,borot2019relating}. Their generalisation for $g = 0$ and $n \geq 3$ was an open problem from \cite{CMSS},  which Theorem~\ref{thm:R-transform-GenusZero} solves. Its application to free probability and random matrices will be discussed in Sections~\ref{SecAppl} and ~\ref{SecGUEDET}.
\end{remark}

We can present the relation in terms of the coefficients of the $n$-point functions. Although we state it only in genus $0$, the interested reader can easily derive the formula in higher genus from Theorem~\ref{thm:R-transform-HigherGenera} or the more convenient preliminary form Lemma~\ref{lem:KeyComb} appearing later in the text.
\begin{definition}
\label{deftreebis} Let $\mathcal{T}_{n}$ be the set of trees $T$ obtained by connecting to a $T' \in \mathcal{G}_{0,n}$ finitely many univalent black vertices. The difference with $\mathcal{G}_{0,n}$ is therefore that we allow hyperedges $I$ with $\# I = 1$. This makes the set $\mathcal{T}_{n}$ infinite. We denote $\mathcal{T}_{n}(\mathbf{r} + 1) \subset \mathcal{T}_n$ the subset of trees in which the $i$-th white vertex has valency $r_i + 1$. Note that if $\ell_i(T)$ is the number of univalent black vertices incident to the $i$-th white vertex, we have $\# {\rm Aut}(T) = \prod_{i = 1}^n \ell_i(T)!$.
\end{definition}
\begin{theorem} 
 \label{coeffThm}Let $f,f^{\vee}\colon PS \rightarrow R$ be multiplicative functions and $G_{0,n},G_{0,n}^{\vee}$ their respective $n$-point functions. Suppose that $f = \zeta \ast f^{\vee}$. Then, for any $n \geq 3$ and $k_1,\ldots,k_n > 0$ we have:
\[
 F_{0;k_1,\ldots,k_n} =  \Big[\prod_{i = 1}^n w_i^{k_i}\Big] \sum_{\substack{0 \leq r_i \leq k_i \\ i \in [n]}} \prod_{i = 1}^n \frac{k_i!}{(k_i - r_i)!} \sum_{T \in \mathcal{T}_n(\mathbf{r} + 1)} \frac{\prod_{I \in \mathcal{I}(T)}'' G_{0,\# I}^{\vee}(w_I)}{\# {\rm Aut}(T)}\,,
\]
where $\prod''$ means that one should replace each occurrence of $G_{0,1}^{\vee}(w_j)$ with $G_{0,1}^{\vee}(w_j) - 1$, and each occurrence\footnote{Trees do not have hyperedges of the type $\{i,i\}$, so if $2$-point functions occur it is only in their shifted version.} of $G_{0,2}^{\vee}(w_i,w_j)$ with $\widetilde{G}_{0,2}^{\vee}(w_i,w_j)$.  
\end{theorem} 
This will be proved in Section~\ref{coeffgenus0}. The dual statements of these three results, i.e.~the relations giving $G_n^{\vee}$ in terms of $G_n$, have a similar structure and will be given in Section~\ref{SecDual}. Before turning to the proof, we explain how to use the formulae in practice.

\begin{remark} Note that Theorem~\ref{coeffThm} can also be applied for $n=1$ and $n=2$. A computation with the Lagrange inversion formula shows that for $n=1$ it implies 
\[
\frac{\dd}{\dd X} G_{0,1}(X) = \frac{\dd}{\dd X} G^\vee_{0,1}(w),
\]
which recovers Voiculescu's $R$-transform. For $n=2$ we recover the positive-degree part of the Laurent series expansion of~\eqref{0102} thus reproducing the result of~\cite{CMSS}.
\end{remark}

\subsection{Reformulation}
\label{Sec:Exsmall}

It is sometimes more convenient (see~\cite{BG-F18}) to use a different convention for the definition of the $n$-point functions, namely 
we can use another change of variables given by 
\begin{equation}
\label{xXG01} x(w) = \frac{1}{X(w)} =  w^{-1}G_{0,1}^{\vee}(w)\,,
\end{equation}
and 
we introduce the differential forms for $2g - 2 + n \geq 0$:
\begin{equation}
\label{omegdiff}
\begin{split}
\omega_{g,n}(x_1,\ldots,x_n) & = \prod_{i = 1}^n \frac{\dd x_i}{x_i}\,G_{g,n}(x_1^{-1},\ldots,x_n^{-1}) 
\\
& = \sum_{k_1,\ldots,k_n > 0} F_{g;k_1,\ldots,k_n}\,\prod_{i = 1}^n \frac{\dd x_i}{x_i^{k_i + 1}}\,, \\
\omega_{g,n}^{\vee}(w_1,\ldots,w_n) & = \prod_{i = 1}^n \frac{\dd w_i}{w_i}\,G^{\vee}_{g,n}(w_1,\ldots,w_n) 
\\
& = \sum_{k_1,\ldots,k_n > 0} F^{\vee}_{g;k_1,\ldots,k_n}\, \prod_{i = 1}^n w_i^{k_i - 1} \dd w_i\,,
\end{split}
\end{equation}
and their shifted version for $(g,n) = (0,2)$:
\begin{equation}
\label{ome20nu}
\begin{split}
\widetilde{\omega}_{0,2}(x_1,x_2) & = \omega_{0,2}(x_1,x_2) + \frac{\dd x_1 \dd x_2}{(x_1 - x_2)^2}\,,
\\
\widetilde{\omega}^{\vee}_{0,2}(w_1,w_2) & = \omega_{0,2}^{\vee}(w_1,w_2) + \frac{\dd w_1 \dd w_2}{(w_1 - w_2)^2}\,.
\end{split}
\end{equation}
Let us present the equivalent form it gives to Theorem~\ref{thm:R-transform-GenusZero}. The relation $G_{0,1}^\vee(w)=G_{0,1}(x^{-1})$ can be rephrased as  the statement on the functional inverse 
\begin{equation}
	\label{xweq} w(x) = \frac{G_{0,1}(x^{-1})}{x} \qquad \Longleftrightarrow \qquad x(w) = \frac{G^\vee_{0,1}(w)}{w}\,.
\end{equation}
For $(g,n) = (0,2)$, \eqref{0102} becomes:
\begin{equation}
\label{om02tildeequl} \widetilde{\omega}_{0,2}(x_1,x_2) = \widetilde{\omega}^{\vee}_{0,2}(w_1,w_2)\,.
\end{equation}
The alternative convention makes these relations particularly simple to remember. For $n \geq 3$, using the variable $x = X^{-1}$, the operator of Definition~\ref{Or2} becomes:
\[
\vec{\mathsf{O}}_{r}^{\vee}(w) = \sum_{m \geq 0} (-x\partial_{x})^m \frac{- x \dd w}{w \dd x} \cdot [v^m] \Big(\partial_{y} + \frac{v}{y}\Big)^r \cdot 1\Big|_{y = xw}\,,
\]
and we get
\[
\begin{split}
& \omega_{0,n}(x_1,\ldots,x_n) = 
\\ & \quad \sum_{r_1,\ldots,r_n \geq 0} \prod_{i = 1}^n \frac{\dd x_i}{x_i}\, \vec{\mathsf{O}}_{r_i}^{\vee}(w_i) \Big(\frac{w_i}{\dd w_i}\Big)^{r_i + 1} \sum_{T \in \mathcal{G}_{0,n}(\mathbf{r} + 1)} \prod_{I \in \mathcal{I}(T)}' \omega^{\vee}_{0,\# I}(w_I)\,,
\end{split}
\]
where $\prod'$ means that any occurrence of $\omega_{0,2}^{\vee}(w_i,w_j)$ with $i \neq j$ should be replaced with $\widetilde{\omega}_{0,2}^{\vee}(w_i,w_j)$. 

\subsection{Examples}

\subsubsection{\texorpdfstring{$(g,n) = (0,3)$}{(g,n)=(0,3)}}

There are exactly four trees in $\mathcal{G}_{0,3}$ (Figure~\ref{G03fig}), and we get from Theorem~\ref{thm:R-transform-GenusZero}:
\begin{equation*}
\begin{split}
G_{0,3}(X_1,X_2,X_3) & = \Big(\prod_{a = 1}^3 P^{\vee}(w_a)\Big)\Bigg[G_{0,3}^{\vee}(w_1,w_2,w_3) \\
& \quad  - \sum_{i = 1}^3 w_i  \partial_{w_i}\bigg(\frac{\prod_{j \neq i} \big(G_{0,2}^{\vee}(w_i,w_j) + \frac{w_iw_j}{(w_i - w_j)^2}\big)}{G_{0,1}^{\vee}(w_i)\frac{w_i}{x(w_i)}\partial_{w_i} x(w_i)}\bigg)\Bigg]. \\
\end{split}
\end{equation*}
In terms of differential forms, the formula is also slightly simpler to remember:
\begin{equation}
\label{03form}
\omega_{0,3}(x_1,x_2,x_3) = - \omega_{0,3}^{\vee}(w_1,w_2,w_3) + \sum_{i = 1}^3 \dd_{w_i}\bigg(\frac{\prod_{j \neq i} \widetilde{\omega}_{0,2}^{\vee}(w_i,w_j)}{\dd x_i \,\dd w_i}\bigg)\,.
\end{equation}


\begin{figure}[H]
\[
\begin{split}
& 
\vcenter{
	\xymatrix@!C=5pt@R=30pt{
		& *+[o][F-]{{1}} & *+[o][F-]{{2}}  & *+[o][F-]{{3}} \\
		& *+[o][F**]{}\ar@{-}[u] \ar@{-}[ru]\ar@{-}[rru]&  & \\
	}
}
\vcenter{
	\xymatrix@!C=5pt@R=30pt{
		& *+[o][F-]{{1}} & *+[o][F-]{{2}}  & *+[o][F-]{{3}} \\
		& *+[o][F**]{}\ar@{-}[u] \ar@{-}[ru]&  *+[o][F**]{}\ar@{-}[lu]\ar@{-}[ru] & \\
	}
}
\\
{\ }
\\
&\vcenter{
	\xymatrix@!C=5pt@R=30pt{
		& *+[o][F-]{{1}} & *+[o][F-]{{2}}  & *+[o][F-]{{3}} \\
		& *+[o][F**]{}\ar@{-}[u] \ar@{-}[ru]&  *+[o][F**]{}\ar@{-}[u]\ar@{-}[ru] & \\
	}
}
\vcenter{
	\xymatrix@!C=5pt@R=30pt{
		& *+[o][F-]{{1}} & *+[o][F-]{{2}}  & *+[o][F-]{{3}} \\
		& *+[o][F**]{}\ar@{-}[u] \ar@{-}[rru]&  *+[o][F**]{}\ar@{-}[u]\ar@{-}[ru] & \\
	}
}
\end{split}
\]
\caption{\label{G03fig} The trees in $\mathcal{G}_{0,3}$}
\end{figure}

\subsubsection{\texorpdfstring{$(g,n) = (1,1)$}{(g,n)=(1,1)}}

We need to extract the coefficient of $\hbar^{1}$ in the formula of Theorem~\ref{thm:R-transform-HigherGenera}. Recall that the leading order was $\hbar^{-1}$, so this contribution may come from one of the three following possibilities:
\begin{itemize}
\item The vertex without hyperedges. Its weight is the coefficient of $\hbar^1$ in \eqref{O2Xi}. If we pick the leading order $\varsigma(z) = 1 + O(z^2)$ everywhere, we can get a contribution by picking one $G_{1,1}^{\vee}(w)$ in the last line. This comes without powers of $u$, so selects $r = 0$; as $\varsigma$ was replaced by $1$, the exponential in the second line of \eqref{O2Xi} is $1$, thus forcing $m = 0$. 

Another contribution comes from picking the second term $\frac{z^2}{24}$ from one occurrence of $\varsigma(z)$ and replacing the other occurrences of $\varsigma$ by $1$ and picking the term $\hbar^{-1} G^\vee_{0,1}(w)$ in the expansion of $G^\vee_{1}(w)$ in the third line. Here we have three possible cases:
\begin{itemize}
	\item If we pick the second term in one of the two $\varsigma$'s in the second line of the formula, then with these choices the third line turns into $(\hbar u)^{-1}$, and thus the whole expression is equal to $0$. 
	\item 
	If we pick this second term in $\varsigma$ from the denominator of the third line of \eqref{O2Xi}, we will get a linear term in $u$, thus selecting $r = 1$, while in the second line the exponential will be $1$, thus selecting $m = 1$. 
	\item 
	If we rather pick it from the exponential in the third line, we get $u^2$, thus selecting $r = 2$, while in the second line the exponential will be $1$, and as $(\partial_{y} + v/y)^2 \cdot 1 = (v^2 - v)/y^2$, we get contributions from $m = 1$ and $m = 2$.
\end{itemize}
\item The graph consisting of one white vertex connected in two ways to a black vertex (which admits two automorphisms), with $\varsigma$ always replaced by $1$. In this case, we get $\vec{\mathsf{O}}_r^{\vee}(w) u G_{0,2}^{\vee}(w,w)$ (here a factor of $u^2$ comes from the black vertex, further multiplied by $u^{-1}$ coming from the third line in \eqref{O2Xi}; we combine these two exponents of $u$ and replace $\vec{\mathsf{O}}^{\vee}(w)$ by $\vec{\mathsf{O}}_r^{\vee}(w)$), which selects $r = 1$ and subsequently $m = 1$.
\item The correction term $\Delta_1^{\vee}(X_1)$ given by \eqref{Ggnpass-n1}. We must select one term $\frac{z^2}{24}$ in $\varsigma(z)$ either from the numerator or the denominator in the exponential, hence get contributions from $m = 0$ and $m = 2$.
\end{itemize}
This results in:
\begin{equation*} 
\begin{split} 
G_{1,1}(X) & = P^{\vee}(w)G_{1,1}^{\vee}(w) - \frac{1}{24} P^{\vee}(w)w \partial_{w}\bigg(\frac{P^{\vee}(w)}{G_{0,1}^{\vee}(w)}\bigg) \\
& \quad + \frac{1}{24} P^{\vee}(w)w\partial_{w}(P^{\vee}(w)w\partial_{w} - 1)\bigg(\frac{P^{\vee}(w)}{(G_{0,1}^{\vee}(w))^2}\,(w\partial_w)^2 G_{0,1}^{\vee}(w)\bigg) \\
& \quad + \frac{1}{2} P^{\vee}(w)w \partial_w\bigg(\frac{P^{\vee}(w)}{G_{0,1}^{\vee}(w)}\,G_{0,2}^{\vee}(w,w)\bigg) \\
& \quad -\frac{1}{24} \big((P^\vee(w)w\partial_w)^2 -1 \big) \frac{P^\vee(w)w\partial_w G^\vee_{0,1}(w)}{(G^\vee_{0,1}(w))^2}\,.
\end{split}
\end{equation*}
After some tedious algebra, we observe many simplifications:
\begin{equation}
\label{11form}\omega_{1,1}(x) + \omega_{1,1}^{\vee}(w) =  \dd\bigg[\frac{1}{2}\,\frac{\omega_{0,2}^{\vee}(w,w)}{\dd x\,\dd w} - \frac{1}{24}\,\frac{\dd(\partial_w^2 x/\partial_w x)}{\dd x}\bigg]\,.
\end{equation}

\subsubsection{\texorpdfstring{$(g,n) = (0,4)$}{(g,n)=(0,4)}}

As shown in Figure~\ref{G04fig}, there are $29$ trees in $\mathcal{G}_{0,4}$, which can have $4$ different topologies once labellings are forgotten. The weight of a tree $T\in\mathcal{G}_{0,4}$ is given by $\mathcal{W}(T)\coloneqq \prod_{I \in \mathcal{I}(T)}' G^{\vee}_{0,\# I}(w_I)$, recalling that $\prod'$ means that each occurrence of $G_{0,2}^{\vee}(w_i,w_j)$ for $i \neq j$ is replaced with $\widetilde{G}_{0,2}^{\vee}(w_i,w_j)$.  Explicitly, the weight for each tree reads:

\begin{equation*}
\begin{split}
\mathcal{W}(T_0) & = G^{\vee}_{0,4}(w_1,w_2,w_3,w_4)\,, \\
\mathcal{W}(T_1^{(i,j)}) & = \widetilde{G}_{0,2}^{\vee}(w_i,w_j)\, G^{\vee}_{0,3}(w_i,w_k,w_l)\,, \\
\mathcal{W}(T_2^{(i,j)}) & =  \widetilde{G}^{\vee}_{0,2}(w_i,w_j)\,\widetilde{G}^{\vee}_{0,2}(w_i,w_k)\,\widetilde{G}^{\vee}_{0,2}(w_j,w_l)\,, \\
\mathcal{W}(T_3^{(i)}) & = \widetilde{G}^{\vee}_{0,2}(w_i,w_j)\, \widetilde{G}^{\vee}_{0,2}(w_i,w_k)\,\widetilde{G}^{\vee}_{0,2}(w_i,w_l)\,,
\end{split}
\end{equation*}
where $\{i,j,k,l\}= [4]$.

\begin{figure}[H]
\[
\begin{matrix}
	\vcenter{
		\xymatrix@!C=5pt@R=30pt{
			& *+[o][F-]{{1}} & *+[o][F-]{{2}}  & *+[o][F-]{{3}} & *+[o][F-]{{4}}   \\
			& *+[o][F**]{}\ar@{-}[u] \ar@{-}[ru]\ar@{-}[rru] \ar@{-}@/_/[rrru]&  & &  \\
		}
	}
	& 
	\vcenter{
		\xymatrix@!C=5pt@R=30pt{
			& *+[o][F-]{{i}} & *+[o][F-]{{j}}  & *+[o][F-]{{k}} & *+[o][F-]{{l}}   \\
			& *+[o][F**]{}\ar@{-}[u] \ar@{-}[ru] & *+[o][F**]{}\ar@{-}[lu]\ar@{-}[ru] \ar@{-}[rru]  & &  \\
		}
	}
	\\
	{} & {}
	\\
	T_0 & T_1^{(i,j)}
		\\
	{} & {}
			\\
	{} & {}
	\\
		\vcenter{
		\xymatrix@!C=5pt@R=30pt{
			& *+[o][F-]{{i}} & *+[o][F-]{{j}}  & *+[o][F-]{{k}} & *+[o][F-]{{l}}   \\
			& *+[o][F**]{}\ar@{-}[u] \ar@{-}[ru] & *+[o][F**]{} \ar@{-}[lu] \ar@{-}[ru]& *+[o][F**]{} \ar@{-}[lu] \ar@{-}[ru] &  \\
		}
	}
	& 
	\vcenter{
		\xymatrix@!C=5pt@R=30pt{
			& *+[o][F-]{{i}} & *+[o][F-]{{j}}  & *+[o][F-]{{k}} & *+[o][F-]{{l}}   \\
			& *+[o][F**]{}\ar@{-}[u] \ar@{-}[ru] & *+[o][F**]{}\ar@{-}[lu]\ar@{-}[ru]  & *+[o][F**]{}\ar@{-}[llu]\ar@{-}[ru]&  \\
		}
	}
	\\
	{} & {}
	\\
	T_2^{(i,j)} & T_3^{(i)}
\end{matrix}
\]	
\caption{\label{G04fig} For $n=4$, there are $1+12+12+4$ trees in $\mathcal{G}_{0,4}$, which we call $T_0$, $T_1^{(i,j)}$, $T_2^{(i,j)}$ and $T_3^{(i)}$, respectively. In the two middle trees, $i,j \in [4]$ are pairwise distinct; in the last one $i \in [4]$. The missing indices ($k$ and $l$ for the two middle ones; $j, k$ and $l$ for the last one) are chosen so that $\{i,j,k,l\}= [4]$, with $k < l$ for $T_{2}^{(i,j)}$; for the other trees their order does not matter.}
\end{figure}

The functional relation of Theorem~\ref{thm:R-transform-GenusZero} for $n = 4$ yields: 
\begin{equation*}
\begin{split}
& G_{0,4}(X_1,X_2,X_3,X_4)  
\\
& \quad =  
\bigg(\prod_{i = 1}^{4} \vec{\mathsf{O}}_{0}^{\vee}(w_i) \bigg)
\mathcal{W}(T_0) + \sum_{i=1}^4 \vec{\mathsf{O}}_{1}^{\vee}(w_i)\bigg(\prod_{\ell \neq i} \vec{\mathsf{O}}_{0}^{\vee}(w_{\ell})\bigg)\sum_{j \neq i} \mathcal{W}(T_1^{(i,j)}) \\
& \qquad + \sum_{i=1}^4 \vec{\mathsf{O}}_{1}^{\vee}(w_i)\sum_{j\neq i}\vec{\mathsf{O}}_{1}^{\vee}(w_j)\bigg(\prod_{\ell \neq i, j} \vec{\mathsf{O}}_{0}^{\vee}(w_{\ell})\bigg) \mathcal{W}(T_2^{(i,j)}) 
\\ & \qquad
+ \sum_{i=1}^4 \vec{\mathsf{O}}_{2}^{\vee}(w_i)\bigg(\prod_{\ell \neq i} \vec{\mathsf{O}}_{0}^{\vee}(w_{\ell})\bigg) \mathcal{W}(T_3^{(i)})\\
& \quad =\Big(\prod_{a = 1}^4 P^{\vee}(w_a)\Big)\Bigg[\mathcal{W}(T_0)- \sum_{i = 1}^4 w_i  \partial_{w_i}\bigg(\frac{\sum_{j \neq i} \mathcal{W}(T_1^{(i,j)})}{G_{0,1}^{\vee}(w_i)\frac{w_i}{x(w_i)}\partial_{w_i} x(w_i)}\bigg) \\
&\qquad + \sum_{i=1}^4 \sum_{j\neq i}w_iw_j\partial_{w_i}\partial_{w_j}\bigg(\frac{\mathcal{W}(T_2^{(i,j)})}{G_{0,1}^{\vee}(w_i)\frac{w_i}{x(w_i)}\partial_{w_i} x(w_i)G_{0,1}^{\vee}(w_j)\frac{w_j}{x(w_j)}\partial_{w_j} x(w_j)}\bigg) \\
&\qquad + \sum_{i=1}^4 w_i\partial_{w_i} \bigg[ \frac{1}{\frac{w_i}{x(w_i)}\partial_{w_i} x(w_i)}w_i\partial_{w_i}+1\bigg]\bigg(\frac{\mathcal{W}(T_3^{(i)})}{(G_{0,1}^{\vee}(w_i))^2\frac{w_i}{x(w_i)}\partial_{w_i} x(w_i)}\bigg)
\Bigg]\,,
\end{split}
\end{equation*}
with the substitution $X_a = w_a/G_{0,1}^{\vee}(w_a)=1/x_a$ for $a \in [4]$. In terms of the differential forms \eqref{omegdiff}, the relation reads
\begin{equation*}
\begin{split}
& \omega_{0,4}(x_1,x_2,x_3,x_4)  =
\\ & \quad  \omega^{\vee}_{0,4}(w_1,w_2,w_3,w_4)-\sum_{i=1}^4\dd_{w_i}\bigg(\frac{\sum_{j \neq i} \widetilde{\omega}_{0,2}^{\vee}(w_i,w_j)\,\omega_{0,3}^{\vee}(w_i,w_k,w_l)}{\dd x_i \,\dd w_i}\bigg)\\
&\quad +\sum_{i=1}^4\sum_{j\neq i}\dd_{w_i}\dd_{w_j}\bigg(\frac{\widetilde{\omega}_{0,2}^{\vee}(w_i,w_j)\,\widetilde{\omega}_{0,2}^{\vee}(w_i,w_k)\,\widetilde{\omega}_{0,2}^{\vee}(w_j,w_l)}{\dd x_i \,\dd w_i\,\dd x_j \,\dd w_j}\bigg)\\
&\quad +\sum_{i=1}^4\dd_{w_i}\bigg(\big(1+x_i\partial_{x_i}\big)\Big(\frac{\widetilde{\omega}_{0,2}^{\vee}(w_i,w_j)\,\widetilde{\omega}_{0,2}^{\vee}(w_i,w_k)\,\widetilde{\omega}_{0,2}^{\vee}(w_i,w_l)}{x_i \,\dd x_i \,(\dd w_i)^2}\Big)\bigg),
\end{split}
\end{equation*}
where we use the substitution $x_a = G_{0,1}^{\vee}(w_a)/w_a$ for $a \in [4]$ and the convention for $i,j,k,l$ was explained in the caption of Figure~\ref{G04fig}.

\subsection{Relation to topological recursion}
\label{sec:realun}
When $G_{g,n}$ is the generating series of ordinary maps of genus $g$ with $n$ boundaries, $G_{g,n}^{\vee}$ enumerates fully simple maps of the same topology \cite{BG-F18,borot2019relating}. In this case, it is known since Eynard \cite{E1MM,Eformal,Ebook} that $\omega_{g,n}$s are computed by the topological recursion \cite{EO07inv,EORev} applied to the spectral curve
\[
\mathcal{S} = \big(\widehat{\mathbb{C}},x,w,\tfrac{\dd z_1 \dd z_2}{(z_1 - z_2)^2}\big)\,,
\]
while it was conjectured in \cite{BG-F18} and proved in \cite{BDBKS,BCG-F} that $\omega_{g,n}^{\vee}$s are computed by the topological recursion applied to
\[
\mathcal{S}^{\vee} = \big(\widehat{\mathbb{C}},w,x,\tfrac{\dd z_1\dd z_2}{(z_1 - z_2)^2}\big)\,.
\]
Here $x$ and $w$ are two meromorphic functions on $\widehat{\mathbb{C}}$.
More precisely, the topological recursion produces multi-differentials
$$
\Omega_{g,n} = \Omega_{g,n}(z_1,\ldots,z_n)\,,\;  \text{resp.} \; \Omega_{g,n}^{\vee} = \Omega_{g,n}^{\vee}(z_1,\ldots,z_n)\,,
$$
where $z_1,\ldots,z_n$ refer to points on the Riemann sphere $\widehat{\mathbb{C}}$, and their all-order expansions as $z_i \rightarrow \infty$ agree with
$$
\omega_{g,n}(x(z_1),\ldots,x(z_n))\,,\;  \text{resp.} \; \omega_{g,n}^{\vee}(w(z_1),\ldots,w(z_n))\,,
$$
defined in \eqref{omegdiff}-\eqref{ome20nu}. In other words, $\Omega$s are analytic continuations for $\omega$s. We return to analytic continuation statements in Section~\ref{sec:analyticprop}.

This gives a class of examples to test our functional relations numerically. The $(0,3)$ formula \eqref{03form} was indeed shown to hold in this situation in \cite[Section 6]{BG-F18}, and here we see that in fact it holds in the greater generality provided by the master relation \eqref{master}. To test the $(1,1)$ formula \eqref{11form}, let us consider the case of quadrangulations (maps with faces of degree $4$), which corresponds to the spectral curve in parametric form
\[
\begin{split}
	&
	x(z) = c\Big(z + \frac{1}{z}\Big)\,,\qquad w(z) = \frac{z^2 - \tau}{cz^3}\,,\qquad 
	\\ & 
	\frac{\dd z_1\dd z_2}{(z_1 - z_2)^2} = \widetilde{\omega}_{0,2}(z_1,z_2) = \widetilde{\omega}_{0,2}^{\vee}(z_1,z_2) \,,
\end{split}
\]
where $t$ is the weight per quadrangle, and
\[
c = \sqrt{\frac{1 - \sqrt{1 - 12t}}{6t}}\,,\qquad \tau = tc^4 \,.
\]
From this we deduce
\begin{equation*}
\begin{split}
\Omega_{0,2}^{\vee}(z_1,z_2) & = \widetilde{\Omega}_{0,2}^{\vee}(z_1,z_2) - \frac{w'(z_1)w'(z_2)\dd z_1\dd z_2}{(w(z_1) - w(z_2))^2} 
\\ & = \frac{\tau\big(z_1^2z_2^2 + \tau(z_1^2 + 4z_1z_2 + z_2^2)\big) \dd z_1 \dd z_2}{\big(\tau(z_1^2 + z_1z_2 + z_2^2) - z_1^2z_2^2\big)^2}\,, \\
\frac{\Omega_{0,2}^{\vee}(z,z)}{2\,\dd x(z) \,\dd w(z)} & = -\frac{\tau z^4(z^2 + 6\tau)}{2(z^2 - 3\tau)^3(z^2 - 1)}\,,
\end{split}
\end{equation*}
and 
\begin{equation*}
\begin{split}
& \frac{1}{24}\frac{\dd}{\dd x}\bigg(\frac{\partial_w^2 x}{\partial_w x}\bigg) 
\\ & \quad = \frac{1}{24\, x'(z)} \partial_z\bigg(\frac{1}{w'(z)}\partial_z\ln\Big(\frac{x'(z)}{w'(z)}\Big)\bigg) \\
& \quad = - \frac{z^4\big(z^8  - 3z^6(3\tau + 1) + 18\tau z^4 (3\tau + 1) + 3\tau z^2(-33\tau + 1) + 27\tau^2\big)}{12(z^2 - 1)^3(z^2 - 3\tau)^3}\,.
\end{split}
\end{equation*}
Besides, \cite[Section 5.2]{BG-F18} computed from topological recursion\footnote{There is a misprint in \cite{BG-F18}, namely the left-hand side of formula (5.2) should be divided by $x'(z)$ and the left-hand side of formula (5.3) should be divided by $w'(z)$.}:
\begin{equation}
\begin{split}
\Omega_{1,1}(z) & = \frac{z\big(\tau z^4 + z^2(1 - 5\tau) + \tau\big)}{(z^2 - 1)^4(3\tau - 1)^2}\,\dd z\,, \\
\Omega_{1,1}^{\vee}(z) & = \frac{3\tau^2 z\big((3\tau - 2)z^4 + 3\tau(9\tau - 1)z^2 - 27\tau^3\big)}{(z^2 - 3\tau)^4(3\tau - 1)^2}\,\dd z\,,
 \end{split}
\end{equation}
and we checked that \eqref{11form} is indeed satisfied.

\medskip

We suggest that the master relation \eqref{master}, or rather the functional relations given in Theorem~\ref{thm:R-transform-HigherGenera} when written in terms of multi-differentials, in fact describe the effect of the symplectic exchange $x \leftrightarrow w$ in topological recursion. 

\begin{conjecture}
\label{ConjTR} Let $\Sigma$ be a compact Riemann surface, let $x,w$ be meromorphic functions on $\Sigma$ such that $\dd x$ and $\dd w$ do not have common zeroes, and let $\mathfrak{B}$ be a fundamental bidifferential of the second kind. We call $\Omega_{g,n}$ the multi-differentials obtained from the topological recursion with the spectral curve  $(\Sigma,x,w,\mathfrak{B})$, and $\Omega_{g,n}^{\vee}$ the ones associated to the spectral curve $(\Sigma,w,x,\mathfrak{B})$, and define  $\widetilde{\Omega}_{0,2} = \widetilde{\Omega}_{0,2}^{\vee} = \mathfrak{B}$. Then, these multi-differentials will satisfy for all $2g - 2 + n \geq 0$ the functional relations of Theorem~\ref{thm:R-transform-HigherGenera} (after they are converted into relations between meromorphic forms
on products of the spectral curve).
\end{conjecture}

This conjecture holds by definition for $(g,n) = (0,2)$. It holds for $(0,3)$ (formula justified in \cite{BG-F18}) as well as for $(1,1)$ by comparison of \eqref{11form} with \cite[Lemma C.1]{EO07inv} where one can use $w$ as a coordinate. Besides, as we have just discussed, it holds for all $(g,n)$ in the case of spectral curves corresponding to map enumeration. This conjecture should be a step towards understanding the symplectic invariance property of topological recursion envisioned in \cite{EO2MM,EOxy}, which one should extract from the analytic properties of  $\Omega_{g,1} + \Omega_{g,1}^{\vee}$ for all $g \geq 1$.

Furthermore, the conjecture remains true in the more general context of the matrix model with external field. The generalised resolvents of the matrix model with external field were proved to satisfy topological recursion in \cite{EO07inv,EORev}, while the topological recursion for the same spectral curve, but after exchanging the role of $x$ and $w$, produces diagonal correlation functions of the matrix model with external field \cite[Section 4.1]{BCEG-F}, which combinatorially corresponds to the enumeration of the so-called ciliated maps\footnote{Ciliated maps were introduced in \cite{BCEG-F} in order to study $r$-spin intersection numbers on the moduli space of curves in the context of topological recursion and in relation to map enumeration.}. As proved in \cite{BG-F18} and recalled in Theorem~\ref{tunome}, the master relation appears in general unitarily invariant ensembles of random hermitian matrices, so it also applies to the generalised resolvents (of the type $p_{\lambda}$ in Theorem~\ref{tunome}) and the diagonal correlation functions (of the type $\mathcal{P}_{\nu}$ in Theorem~\ref{tunome}) of the matrix model with external field.

\subsection{Analytic properties}
\label{sec:analyticprop}
Whether the generating series are Taylor expansions of analytic functions depends on the topological partition functions that one starts with. However, we can deduce from Theorem~\ref{thm:R-transform-HigherGenera} that the master transformation induced by $\mathsf{D}$ respects the analytic properties of the generating functions, provided one understands well the analytic properties for the $(0,1)$-case.

\begin{proposition}\label{pr313}
Let $Z,Z^{\vee}$ be two topological partition functions such that $Z = \mathsf{D}Z^{\vee}$, and $G_{g,n},G_{g,n}^{\vee}$ the generating series coming from the topological expansion of their $n$-point functions, cf.~Section~\ref{npointSec}. We use the notation $W_{0,1}(x) = x^{-1}G_{0,1}(x^{-1})$ and $X_{0,1}(w) = w^{-1}G_{0,1}^{\vee}(w)$. Suppose that there exists a Riemann surface $\Sigma$ equipped with a marked point $q_{\infty}$ and two meromorphic functions $\mathsf{x}$ and $\mathsf{w}$ such that
\begin{itemize}
\item $q_{\infty}$ is a simple pole of $\mathsf{x}$ and a simple zero of $\mathsf{w}$;
\item $W_{0,1} \circ \mathsf{x}$ is the series expansion of the function $\mathsf{w}$ near $q_{\infty}$, and $X_{0,1} \circ \mathsf{w}$ is the series expansion of the function $\mathsf{x}$ near $q_{\infty}$.
\end{itemize}
Then, the two properties are equivalent:
\begin{itemize}
\item[(i)] for $2g - 2 + n \geq 0$, $G_{g,n}(\frac{1}{\mathsf{x}(z_1)},\ldots,\frac{1}{\mathsf{x}(z_n)})$ is the series expansion of a meromorphic function of $(z_1,\ldots,z_n) \in \Sigma^n$ near $(q_{\infty},\ldots,q_{\infty})$;
\item[(ii)]  for $2g - 2 + n \geq 0$, $G_{g,n}^{\vee}(\mathsf{w}(z_1),\ldots,\mathsf{w}(z_n))$ is the series expansion of a meromorphic function of $(z_1,\ldots,z_n) \in \Sigma^n$ near $(q_{\infty},\ldots,q_{\infty})$.
\end{itemize}
\end{proposition}
\begin{proof}
Assume (ii). We first observe that $\mathsf{x}$ gives a meromorphic analytic continuation on $\Sigma$ of $X_{0,1} \circ \mathsf{w}$, and  $-\frac{\mathsf{x} \dd \mathsf{w}}{\mathsf{w} \dd \mathsf{x}}$ gives a meromorphic analytic continuation on $\Sigma$ of $P^{\vee} \circ \mathsf{w}$.

Second, $G_{0,2}(\frac{1}{\mathsf{x}(z_1)},\frac{1}{\mathsf{x}(z_2)})$ admits an analytic continuation as a meromorphic function of $(z_1,z_2) \in \Sigma^2$ because the right-hand side of \eqref{0102} manifestly does. As $\widetilde{G}_{0,2}(\frac{1}{\mathsf{x}(z_1)},\frac{1}{\mathsf{x}(z_2)})$ is obtained from $G_{0,2}$ by the addition of a manifestly meromorphic term (cf.~\eqref{npointfun}), it also admits an analytic continuation as a meromorphic function of $(z_1,z_2) \in \Sigma^2$.

Third, for any fixed $g \geq 1$, the right-hand side of \eqref{Ggnpass-n1} contains only a finite sum over $m$, consisting of operators that are polynomial in $G_{0,1}^{\vee}(w)$, in $P^{\vee}(w)$, in $w$, in $\partial_{w}$ and in $\partial_{G_{0,1}^{\vee}(w)}$, acting on $P^{\vee}(w)w \partial_w G_{0,1}^{\vee}(w)$. The derivations $\partial_{w}(\cdot)$ and $\partial_{G_{0,1}^{\vee}(w)}(\cdot)$ acting on formal power series of $w$ correspond at the level of meromorphic functions on $\Sigma$ to the operations $\frac{1}{\dd \mathsf{w}} \dd(\cdot)$ and $\frac{1}{\dd(\mathsf{x}\mathsf{w})} \dd (\cdot)$. We deduce that $\Delta_g^{\vee}(\frac{1}{\mathsf{x}(z)})$ admits an analytic continuation as a meromorphic function of $z \in \Sigma$.

Last, for fixed $(g,n)$ with $2g - 2 + n > 0$, $G_{g,n}$ is expressed in terms of $G^{\vee}_{g',n'}$ via \eqref{Ggnpass}, involving  the term $\Delta_g^{\vee}$ that we just handled (present for $n = 1$ only) and a sum over graphs from Definition~\ref{def:graphs}. Only the finitely many graphs with first Betti number bounded by $g$ can contribute. The weight of these graphs is described in Definition~\ref{def:weights}.  As $G_{g,n}$ appears as the coefficient of  $\hbar^{2g - 2 + n}$ in the weighted sum, the formal series $\varsigma$ (cf.~\eqref{varsigmaw}) appearing in the weight of hyperedges $\mathsf{c}^{\vee}$ and in the operator $\vec{\mathsf{O}}$ for white vertices can effectively be truncated to a polynomial of some degree depending on $g$ and $n$. So, the weight of a hyperedge  of cardinality $\# I$ can be replaced by a truncated weight admitting a meromorphic analytic continuation in $\Sigma^{\# I}$. Likewise, in the truncation of the operator-valued weight of white vertices, the sum over $r,m$ (see \eqref{O2Xi}) only receives finitely many contributions: 
\begin{itemize}
\item the sum over $r$ is associated to the coefficient of $u_i^r$. This variable appears to be multiplied by $\hbar$ in the expressions of $\mathsf{c}^{\vee}$ and $\vec{\mathsf{O}}$. Since we extract the coefficient of $\hbar^{2g - 2 + n}$ in the weighted sum, $r$ is bounded by $2g-2+n$.
\item The sum over $m$ is associated to the coefficient of $v_i^m$. Since the sum over $r$ is bounded and we extract the coefficient of $\hbar^{2g - 2 + n}$, the expression for $\vec{\mathsf{O}}$ is actually a polynomial in $v_i$. Therefore the sum over $m$ is finite.
\end{itemize}
Those contributions are polynomials in generating series and derivations that admit meromorphic analytic continuation on $\Sigma$. Therefore $G_{g,n}(\frac{1}{\mathsf{x}(z_1)},\ldots,\frac{1}{\mathsf{x}(z_n)})$ admits an analytic continuation as a meromorphic function of $(z_1,\ldots,z_n) \in \Sigma^n$, and we have checked (i).

The proof of (i) $\Rightarrow$ (ii) is completely analogous based on the dual expressions of $G_{g,n}^{\vee}$ in terms of $G_{g,n}$ given in Section~\ref{SecDual}.
\end{proof}

The assumption on analytic properties of $(0,1)$ in Proposition~\ref{pr313} could be reformulated as: $G_{0,1}(\frac{1}{\mathsf{x}(z)})$ and $G_{0,1}^{\vee}(\mathsf{w}(z))$ both represent the series expansion of the function $\mathsf{x}(z)\mathsf{w}(z)$ of $z \in \Sigma$ near $q_{\infty}$. The triple $(\Sigma,\mathsf{x},\mathsf{w})$ is then called \emph{spectral curve}. In full generality, checking whether there exists or not a spectral curve (and constructing it explicitly enough) is hard, but it can be handled in specific examples, in particular when $W_{0,1}$ is the series expansion of an algebraic function (as in the example of Section~\ref{sec:realun}).

The conclusion of Proposition~\ref{pr313} could be reformulated using the differential form-valued generating series: the $\omega_{g,n}$ extend to meromorphic $n$-forms $\Omega_{g,n}$ on $\Sigma^n$ for $2g -2 + n \geq 0$ if and only if $\omega_{g,n}^{\vee}$ extend to meromorphic $n$-forms $\Omega^{\vee}_{g,n}$ on $\Sigma^n$ for $2g - 2 + n \geq 0$. In these terms, for instance, \eqref{om02tildeequl} translates into an identity between meromorphic bidifferentials on $\Sigma^2$
$$
\Omega_{0,2} + \frac{\dd \mathsf{x}_1 \dd \mathsf{x}_2}{(\mathsf{x}_1 - \mathsf{x}_2)^2} = \Omega_{0,2}^{\vee} + \frac{\dd\mathsf{w}_1\dd \mathsf{w}_2}{(\mathsf{w}_1 - \mathsf{w}_2)^2}\,,
$$
where $\mathsf{x}_i = \mathsf{x} \circ {\rm pr}_i$, $w_i = \mathsf{w} \circ {\rm pr}_i$ and ${\rm pr}_i : \Sigma^2 \rightarrow \Sigma$ is the projection on the $i$-th copy. Likewise, \eqref{11form} translates into an identity between meromorphic $1$-forms on $\Sigma$
\begin{equation}
\label{11formanal}
\Omega_{1,1} + \Omega_{1,1}^{\vee} = \dd\Bigg[ \frac{\iota^* \Omega_{0,2}^{\vee}}{2\,\dd \mathsf{x} \,\dd \mathsf{w}} - \frac{1}{24\,\dd \mathsf{x}}\dd\Big(\frac{\dd\big(\frac{\dd \mathsf{x}}{\dd \mathsf{w}}\big)}{\dd \mathsf{x}}\Big)\bigg]\,,
\end{equation}
where $\iota : \Sigma \rightarrow \Sigma \times \Sigma$ is the diagonal inclusion (specialisation to the two same points).

As the formula in Theorem~\ref{thm:R-transform-HigherGenera} has an explicit structure, the location and order of poles of $\Omega_{g,n}^{\vee}$ knowing the singularities of $\Omega_{g,n}$ (or vice versa) could in principle be described, but we do not pursue it here.

\section{Proof of the main results}

\label{Sec4}

\subsection{All genus: proof of Theorem~\ref{thm:R-transform-HigherGenera}}
\label{proofg}

We rely on the description of the projective representation of $\widehat{\mathfrak{gl}}_{\infty}(R)$ in terms of differential operators on $\mathcal{F}_R$ \cite{KacBook}. Our starting point is the formula found \textit{e.g.}~in \cite[Proposition 3.1]{BDBKSexplicit}, expressing the conjugation of the Heisenberg field with the operator $\mathsf{D}$ or $\mathsf{D}^{-1}$:
\begin{equation}
\label{DJD} \begin{split}
& \mathsf{D}^{\mp 1} \,\widetilde{\mathsf{J}}(X)\, \mathsf{D}^{\pm 1}  = \\
& \quad  \sum_{k \in \mathbb{Z}} X^{k}  [w^{k}] \,\,\sum_{r \geq 0} \partial_{y}^r \exp\bigg(\pm k\, \frac{\varsigma(k\hbar\partial_y)}{\varsigma(\hbar\partial_{y})} \ln(1 + y) \bigg)\Big|_{y = 0}[u^{r}] \frac{1}{\hbar u \varsigma(\hbar u)} \\
& \quad
 \cdot\exp\bigg(\sum_{m > 0} \hbar u\, \varsigma(\hbar u w\partial_{w}) \mathsf{J}_{-m} w^{-m}\bigg) \exp\bigg(\sum_{m > 0} \hbar u\, \varsigma(\hbar u w\partial_{w}) \mathsf{J}_{m}w^m\bigg)\,,
\end{split}
\end{equation}
where we recall $\varsigma(w) = 2w^{-1}{\rm sinh}(w/2) = 1 + \frac{w^2}{24} + O(w^4)$. Let us employ this formula with the sign $+$, substitute it in Equation~\eqref{Gntildes} and then make use of the commutation relations $[\mathsf{J}_l,\mathsf{J}_m] = l \delta_{l + m,0}$, as well as the properties $\bra \mathsf{J}_{-m} = 0$ and $\mathsf{J}_m \ket =0$, for $m>0$. As explained in \cite[Section 2]{BDBKS} (the case discussed in \emph{op.~cit.}~is a bit more general: it depends on a choice of a formal power series $\psi(y,\hbar^2)$ that is equal to $\ln(1+y)$ in our case), we obtain a sum over the graphs of Definition~\ref{def:graphs}, where the weights consist of the following elementary blocks.

\begin{definition}\label{def:AssignGraph} We fix variables $(u_i,w_i)_{i = 1}^n$.
\begin{itemize}
\item To a hyperedge $I$ of $[n]$, we attach the weight (already in Definition~\ref{def:weights}):
\begin{equation*}
\begin{split}
\mathsf{c}^{\vee}(u_I,w_I) &=\Big(\prod_{i \in I} \hbar u_i \varsigma(\hbar u_i w_i \partial_{w_i})\Big) \widetilde{G}^{\vee}_{\# I}(w_{I})\,, \textrm{if } I\neq \{j,j\}, \textrm{ and} \\
\mathsf{c}^{\vee}(u_I,w_I) &=\big(\hbar u_j \,\varsigma(\hbar u_j w_j \partial_{w_j})\hbar u_j \,\varsigma(\hbar u_j w'_j \partial_{w'_j})\big)\,G_{2}^{\vee}(w_j,w'_j)\Big|_{w_j'=w_j}\,, 
\end{split}
\end{equation*}
in the case $I= \{j,j\}$.
\item To the $i$-th white vertex, we attach the operator weight:
\begin{equation}
\label{U2Xi}\begin{split}
\vec{\mathsf{U}}^{\vee}(X_i) & = \sum_{k \in \mathbb{Z}} X_i^{k} \cdot [w_i^{k}] \,\,\sum_{r \geq 0} \partial_y^{r}\exp\bigg(k\,\frac{\varsigma(k\hbar\partial_{y})}{\varsigma(\hbar\partial_{y})}\,\ln(1 + y)\bigg)\Big|_{y = 0} \\
& \phantom{\sum_{k \in \mathbb{Z}} X_i^{k}} \!\! \cdot [u_i^{r}]\,\,\frac{\exp\big(\hbar u_i\, \varsigma(\hbar u_i w_i \partial_{w_i}) (G_{1}^{\vee}(w_i) - \hbar^{-1})\big)}{\hbar u_i \,\varsigma(\hbar u_i)}\,.
\end{split}
\end{equation}
\end{itemize}
This operator acts from the left on series depending on variables $u_i$ and $w_i$ and gives as output a series in $X_i$. The subtraction of $\hbar^{-1}$ in the last line kills the conventional constant added in the definition of the $1$-point function in \eqref{npointfun}.
\end{definition}

\begin{lemma}[Key combinatorial identity]  \label{lem:KeyComb} The relation $Z = \mathsf{D} Z^{\vee}$ is equivalent to the relations:
\begin{equation}
\label{eqn:KeyComb}
\forall n > 0,\qquad \widetilde{G}_n(X_1,\ldots,X_n) - \hbar^{-1}\delta_{n,1} = \sum_{\Gamma \in \mathcal{G}_n} \frac{1}{\# {\rm Aut}(\Gamma)} \prod_{i = 1}^n \vec{\mathsf{U}}^{\vee}(X_i) \prod_{I \in \mathcal{I}(\Gamma)} \mathsf{c}^{\vee}(u_I,w_I)\,,
\end{equation}
where the first product is taken from left to right with $i$ increasing.
\end{lemma}

\begin{proof} For the proof we just give a dictionary that explains how to match Equation~\eqref{eqn:KeyComb} to the statement of~\cite[Lemma 2.1]{BDBKS} assuming that the function $\psi(y,\hbar^2)$ in \emph{op.~cit.}~is set to be $\ln(1+y)$, and the variables $z_i$ there correspond to $w_i$ in our notation. In order to follow the table below one needs to remember the notation and terminology used in~\cite[Lemma 2.1]{BDBKS}.

\vspace{0.4cm}

{ 

\bgroup
\def\arraystretch{1.5}
\noindent
\begin{tabular}{p{6cm} | p{6cm}}
	\hline 
	\hspace*{7mm} Graphs in equation~\eqref{eqn:KeyComb}   & \hspace*{2mm} Graphs in \cite[Lemma 2.1]{BDBKS}  \\ 
	\hline 
	$\bullet\;$the white vertex labelled by $i$ 
	& 
	$\bullet\;$the white vertex labelled by $i$ 
	\\
	$\bullet\;$a hyperedge of cardinality $\geq 3$ & 
	$\bullet\;$a multi-edge of cardinality $\geq 3$  \\
	$\bullet\;$a hyperedge of cardinality $2$ incident to two different white vertices labelled by $i$ and $j$ & 
	$\bullet$~the formal sum of a multi-edge of cardinality $2$	 incident to two different white vertices labelled by  $i$ and $j$ and an ordinary edge  connecting   these two vertices   \\
	$\bullet\;$a hyperedge of cardinality $2$  incident  (twice) to the white vertex labelled by $i$ & 
	$\bullet\;$a multi-edge of cardinality $2$ incident (twice) to the white vertex labelled by $i$  \\
	$\bullet\;$the extra factor for the white vertex labelled by $i$  (the numerator of the second line in~\eqref{U2Xi}) 
	& 
	$\bullet\;$the formal sum of all multi-edges of 
cardinality $1$ incident to the white 
 vertex labelled by $i$
 \\
	\hline
\end{tabular}
\egroup

} 

\vspace{0.6cm}

With this dictionary, we see that the contribution of all multi-edges of cardinality $\geq 3$ is exactly the same in both formulae; one just has to identify the expansions of $G_{k}^{\vee}(w_1,\dots,w_k)$ with $s_{g;t_1,\dots,t_k}$ in~\cite[Equation (11)]{BDBKS}. The same holds for hyperedges (in our terminology) of cardinality $2$ that are incident twice to the same white vertex. 

Consider a hyperedge of cardinality $2$ incident to two different white vertices labelled by $i$ and $j$. By our convention, we assign to it the weight 
\[
\big(\hbar u_i \,\varsigma(\hbar u_i w_i \partial_{w_i})\big)\big(\hbar u_j \,\varsigma(\hbar u_j w_j \partial_{w_j})\big)\bigg(
\,G_{2}^{\vee}(w_i,w_j) +\frac{w_iw_j}{(w_i-w_j)^2}\bigg)\,.
\]
It is possible to say that for each such hyperedge we can choose either the first summand or the second one, and to take the sum over all possible choices for all such edges. In the notation and terminology of~\cite[Lemma 2.1]{BDBKS} the first choice is the weight associated to the respective multi-edge, and the second choice is the weight associated to what is called an ordinary edge in~\emph{op.~cit.}

In the notation of~\cite[Lemma 2.1]{BDBKS}, there can be any nonnegative number of multi-edges of cardinality $1$ attached to a white vertex labelled by $i$. Taking into account the total weight is equal to the inverse order of the automorphism group of the graph, the weights of all such edges can be assembled into the extra exponential factors. These exponential factors are, in our notation, precisely $\exp\big(\hbar u_i\, \varsigma(\hbar u_i w_i \partial_{w_i}) (G_{1}^{\vee}(w_i) - \hbar^{-1})\big)$ for each $i=1,\dots,n$ (the shift by $\hbar^{-1}$ is the result of the convention we have chosen in Equation~\eqref{npointfun}). Note that these factors are exactly the missing ones that we need to match our operators $\vec{\mathsf{U}}^{\vee}(X_i)$ and the corresponding operators in~\cite[Lemma 2.1]{BDBKS}. 

This completes the comparison and establishes~\eqref{eqn:KeyComb} as a special case of~\cite[Lemma 2.1]{BDBKS}.
\end{proof}

To pass from \eqref{eqn:KeyComb} to the statement of Theorem~\ref{thm:R-transform-HigherGenera}, in particular to \eqref{Ggnpass}, we proceed in three steps, which were suggested by Kazarian in an unpublished work (see the talks \cite{Kaz2019,Kaz2020a}). In our exposition we follow~\cite[Section 4.4]{BDBKSexplicit}.

First, if $\Phi(y)$, $\Psi(u)$ and $Y(w)$ are three formal power series with $Y(w) = O(w)$, we observe that
\begin{align}\label{eq:FirstStep-u}
\sum_{r \geq 0} (\partial_y^r\Phi)(0) \cdot [u^r]\,\, e^{uY(w)} \Psi(u) = \sum_{r \geq 0} (\partial_y^r \Phi)(Y(w)) \cdot [u^r]\,\, \Psi(u)\,,
\end{align}
which can be checked by direct expansion of the left-hand side in the variable $u$. This observation applied to $Y(w_i) = G_{0,1}^{\vee}(w_i) - 1$   allows us to remove the otherwise infinitely many\footnote{More precisely, for every fixed degree of $X_i$ this amounted to finitely many contributions, but their number increases with the degree of the monomial in $X$s, resulting in infinitely many contributions in the functional relation.} contributions $\exp\big(u_i (G^{\vee}_{0,1}(w_i) - 1)\big)$  in the second line of \eqref{U2Xi}. Note that the second line of \eqref{U2Xi} starts with $1/u$ for $n=1$. We will discuss this case separately later. For $n\geq 2$, Equation~\eqref{eqn:KeyComb} is turned by this trick into
\begin{equation}
	\label{eqn:KeyComb-2}
\begin{split}
& \widetilde{G}_n(X_1,\ldots,X_n)  =
\\ 
& \quad  \prod_{i=1}^n \Bigg[ \sum_{k \in \mathbb{Z}} X_i^{k}\cdot  [w_i^{k}] \,\,\sum_{r \geq 0} \partial_y^{r}\exp\bigg(k\,\frac{\varsigma(k\hbar\partial_{y})}{\varsigma(\hbar\partial_{y})}\,\ln(1 + y)\bigg)\Big|_{y = G^\vee_{0,1}(w_i)-1} \\
& \quad [u_i^{r}] \, \frac{\exp\big(\hbar u_i\, \varsigma(\hbar u_i w_i \partial_{w_i}) (G_{1}^{\vee}(w_i) - \hbar^{-1})-u_i(G_{0,1}^{\vee}(w_i) - 1)\big)}{\hbar u_i \,\varsigma(\hbar u_i)}\Bigg]
\\ & \quad 
 \sum_{\Gamma \in \mathcal{G}_n} \frac{\prod_{I \in \mathcal{I}(\Gamma)} \mathsf{c}^{\vee}(u_I,w_I)}{\# {\rm Aut}(\Gamma)} \,.
\end{split}
\end{equation}

Second, we have to remove the explicit dependence on $k$s coming from the weight of the white vertices \eqref{U2Xi}. To this end, we use for $r \geq 0$ and $k \in \mathbb{Z}$:
\begin{equation}
 \label{eq:FirstFactorLB}
 \begin{split}
&  \partial_y^{r} \exp\Big(k\frac{\varsigma(k\hbar \partial_y)}{\varsigma(\hbar\partial_y)} \ln(1+y) \Big)\Big|_{y=G^\vee_{0,1}(w)-1} 
\\ &
\quad =  \partial_y^{r} \exp\Big(k\frac{\varsigma(k\hbar \partial_y)}{\varsigma(\hbar\partial_y)} \ln y \Big)\Big|_{y=G^\vee_{0,1}(w)} \\
& \quad =  (G^\vee_{0,1}(w))^{k} \cdot \bigg( \exp(-k\ln y)\partial_y^{r} \exp\Big(k\frac{\varsigma(k\hbar \partial_y)}{\varsigma(\hbar\partial_y)} \ln y \Big)\bigg)\Big|_{y=G^\vee_{0,1}(w)}
 \\
& \quad =  (G^\vee_{0,1}(w))^{k} \cdot \bigg( \Big(\partial_y+\frac{k}{y}\Big)^{r} \exp\Big(k\big(\frac{\varsigma(k\hbar \partial_y)}{\varsigma(\hbar\partial_y)}-1\big) \ln y \Big)\bigg)\Big|_{y=G^\vee_{0,1}(w)}\,.
\end{split}
\end{equation}
The second factor here has a polynomial dependence on $k$ in each degree in $\hbar$, and we can use the following trick to assemble it: if $Q(k)$ is a polynomial in $k$, we have
\begin{equation}\label{eq:RemovingDependenceOnPower}
\sum_{k} Q(k)\,X^{k} =  Q(X\partial_X) \sum_{k} X^{k} = \sum_{j \geq 0} (X\partial_X)^j \,[v^j ]\, \sum_{k} X^{k} Q(v)\,.
\end{equation}
For every given monomial in $\hbar,X_1,\ldots,X_n$, we need to use such a formula when the summation ranges over a set of integers $k$ bounded from below, in which case it makes sense in the framework explained in Remark~\ref{rem:div}. Equation~\eqref{eqn:KeyComb-2} is turned by this trick into
\begin{equation}
	\label{eqn:KeyComb-3}
	\begin{split}
		& \widetilde{G}_n(X_1,\ldots,X_n) 
= \prod_{i=1}^n\Bigg[ \sum_{j \geq 0}  (X_i\partial_{X_i})^j [v^j]  \sum_{k \in \mathbb{Z}} X_i^{k}\cdot  [w_i^{k}] (G^\vee_{0,1}(w_i))^{k}\cdot
\\ & \quad \sum_{r \geq 0} \bigg( \Big(\partial_y+\frac{v}{y}\Big)^{r} \exp\Big(v\big(\frac{\varsigma(v\hbar \partial_y)}{\varsigma(\hbar\partial_y)}-1\big) \ln y \Big)\bigg)\Big|_{y=G^\vee_{0,1}(w_i)} [u_i^{r}] \\
		& \quad \frac{\exp\big(\hbar u_i\, \varsigma(\hbar u_i w_i \partial_{w_i}) (G_{1}^{\vee}(w_i) - \hbar^{-1})-u_i(G_{0,1}^{\vee}(w_i) - 1)\big)}{\hbar u_i \,\varsigma(\hbar u_i)}\Bigg]
		\\ & \quad 
		\sum_{\Gamma \in \mathcal{G}_n} \frac{\prod_{I \in \mathcal{I}(\Gamma)} \mathsf{c}^{\vee}(u_I,w_I)}{\# {\rm Aut}(\Gamma)} \,.
	\end{split}
\end{equation}
 
Third, to remove the factor of $(G^\vee_{0,1}(w_i))^{k}$ in~\eqref{eqn:KeyComb-3}, we use the Lagrange inversion formula: given a Laurent series $L(w)$, we have 
\begin{equation} \label{eq:LagrangeInversion}
\sum_{k \in \mathbb{Z}} X^k \cdot [\tilde{w}^k] \,\,(G^\vee_{0,1}(\tilde w))^{k} L(\tilde{w}) = \frac{\dd \ln w}{\dd \ln X}\,L(w)\,,\qquad {\rm where}\;\;  X = \frac{w}{G^\vee_{0,1}( w)}\,.
\end{equation}

\begin{remark}
This form of the Lagrange inversion formula is equivalent to \cite[Equation~(2.1.4)]{Gessel}; to match it precisely (up to a change of notation) note that for $w=w(X)=XG^\vee_{0,1}( w(X))$ we have 
\begin{equation*}
	\frac{\dd w}{\dd X} = G^\vee_{0,1}( w) + X \frac{\dd G^\vee_{0,1}( w)}{\dd w} \frac{\dd w}{\dd X}
\end{equation*}
and hence
\begin{equation*}
	\frac{\dd \ln w}{\dd \ln X} = \frac{1}{1-X \frac{\dd G^\vee_{0,1}( w)}{\dd w}}\,.
\end{equation*}
\end{remark}

Using this observation for each $i \in [n]$, we obtain for $n \geq 2$ and $g \geq 0$:
\begin{equation}
\label{Ggnpasstilde}
\begin{split}
& \widetilde{G}_{g,n}(X_1,\ldots,X_n) = 
\\ & \quad  [\hbar^{2g - 2 + n}]\,\,\sum_{\Gamma \in \mathcal{G}_n}  \frac{1}{\# {\rm Aut}(\Gamma)} \prod_{i = 1}^n \vec{\mathsf{O}}^{\vee}(w_i) \prod_{I \in \mathcal{I}(\Gamma)} \mathsf{c}^{\vee}(u_I,w_I)\,,
\end{split}
\end{equation}
using the substitutions
\begin{equation}
\label{XPveesub}X_i = \frac{w_i}{G_{0,1}^{\vee}(w_i)}\,,\qquad P^{\vee}(w_i) = \frac{\dd \ln w_i}{\dd \ln X_i}\,,\qquad X_i\partial_{X_i}= P^{\vee}(w_i) w_i\partial_{w_i}\,. 
\end{equation}
The operators $\vec{\mathsf{O}}^{\vee}(w_i)$ are those introduced in \eqref{O2Xi}, and they involve the above choice of $P^{\vee}(w_i)$. For $2g - 2 + n > 0$, the left-hand side is simply $G_{g,n}(X_1,\ldots,X_n)$ and we recover \eqref{Ggnpass}. Up to the comparison of the right-hand side in the case $(g,n) = (0,2)$ to the second formula of \eqref{0102} and the study of the special case $(g,n) = (0,1)$, which will both be carried out in the next section, this concludes the proof of Theorem~\ref{thm:R-transform-HigherGenera} for $n\geq 2$.

In the case $n=1$ we have to make one extra step. There is a graph that contains just one vertex and no hyperedges. The expression that we assign to this graph by Definition~\ref{def:AssignGraph} is equal to the operator~\eqref{U2Xi} applied to the constant $1$. Then, in order to apply Equation~\eqref{eq:FirstStep-u}, we have to remove manually the singularity in $u$ as it is done in~\cite[Section 6.2]{BDBKSexplicit}, that is, the contribution of this graph to $G_1(X) - \hbar^{-1}$ is:
\begin{equation*}
\begin{split}
	& \quad \sum_{k \geq 1} X^{k} \cdot [w^{k}] \,\,\sum_{r \geq 0} \partial_y^{r}\exp\bigg(k\,\frac{\varsigma(k\hbar\partial_{y})}{\varsigma(\hbar\partial_{y})}\,\ln(1 + y)\bigg)\Big|_{y = 0} \\
	& \quad \cdot [u^{r}]\,\,\Bigg(\frac{\exp\big(\hbar u\, \varsigma(\hbar u w \partial_{w}) (G_{1}^{\vee}(w) - \hbar^{-1})\big)}{\hbar u\,\varsigma(\hbar u)} - \frac{\exp\big(u(G^\vee_{0,1}(w)-1)\big)}{\hbar u}\Bigg)
	\\ 
	& \quad 
	  + \sum_{k \geq 1} X^{k} \cdot [w^{k}] \,\,\sum_{r \geq 0} \partial_y^{r}\exp\bigg(k\,\frac{\varsigma(k\hbar\partial_{y})}{\varsigma(\hbar\partial_{y})}\,\ln(1 + y)\bigg)\Big|_{y = 0} 
	  \\ & \quad
	\cdot [u^{r}]\,\,\frac{\exp\big(u\, (G^\vee_{0,1}(w)-1)\big)}{\hbar u}
	\,.
\end{split}
\end{equation*} 
The first summand here can be evaluated by the three steps performed above in the general case, and it gives the contribution of the graph with one vertex and no hyperedges in the statement of Theorem~\ref{thm:R-transform-HigherGenera}. 
Let us show that the second summand is equal to $\hbar^{-1}(G^\vee_{0,1}(w)-1)  +\sum_{g \geq 1} \hbar^{2g-1} \Delta_g^{\vee}(X)$. We follow the computation done in~\cite[Section 6.2]{BDBKSexplicit}, and this computation just implements the same three steps as above with an extra small step to remove the singularity in $u$ and further simplifications related to the particularities of this simpler formula:
\begin{equation} \label{eq:TowardsDelta}
	\begin{split}
	& \sum_{k \geq 1} X^{k} \cdot [w^{k}] \,\,\sum_{r \geq 0} \partial_y^{r}\exp\bigg(k\,\frac{\varsigma(k\hbar\partial_{y})}{\varsigma(\hbar\partial_{y})}\,\ln(1 + y)\bigg)\Big|_{y = 0} 
	\\ &
	\cdot [u^{r}]\,\,\frac{\exp\big(u\, (G^\vee_{0,1}(w)-1)\big)}{\hbar u}
	\\ 
	& \quad = 	\int_0^X \frac {\dd X}{X} \sum_{k \geq 1} X^{k} \cdot [w^{k}] \,\,\sum_{r \geq 0} \partial_y^{r}\exp\bigg(k\,\frac{\varsigma(k\hbar\partial_{y})}{\varsigma(\hbar\partial_{y})}\,\ln(1 + y)\bigg)\Big|_{y = 0} 
	\\
	& \qquad 
	\cdot [u^{r}]\,\,w\partial_w \frac{\exp\big(u\, (G^\vee_{0,1}(w)-1)\big)}{\hbar u}
	\\
	& \quad = 	\int_0^X \frac {\dd X}{X} \sum_{k \geq 1} X^{k} \cdot [w^{k}] \,\,\sum_{r \geq 0} \partial_y^{r}\exp\bigg(k\,\frac{\varsigma(k\hbar\partial_{y})}{\varsigma(\hbar\partial_{y})}\,\ln(1 + y)\bigg)\Big|_{y = 0} 
	\\ & \qquad
\cdot [u^{r}]\,\,\frac{\exp\big(u\, (G^\vee_{0,1}(w)-1)\big)}{\hbar}	w\partial_w G^\vee_{0,1}(w)\,.
	\end{split}
\end{equation}
Here we can use the observation made in Equation~\eqref{eq:FirstStep-u}, that is, we rewrite the latter expression as 
\begin{equation*} 
	\begin{split}
& \int_0^X \frac {\dd X}{X} \sum_{k \geq 1} X^{k} \cdot [w^{k}] \,\,\sum_{r \geq 0} \partial_y^{r}\exp\bigg(k\,\frac{\varsigma(k\hbar\partial_{y})}{\varsigma(\hbar\partial_{y})}\,\ln (1+y)\bigg)\Big|_{y = G^\vee_{0,1}(w)-1} 
\\ & 
\cdot [u^{r}]\,\,\frac{w\partial_w G^\vee_{0,1}(w)}{\hbar}
\\
& \quad = 	\int_0^X \frac {\dd X}{X} \sum_{k \geq 1} X^{k} \cdot [w^{k}] \,\,\sum_{r \geq 0} \partial_y^{r}\exp\bigg(k\,\frac{\varsigma(k\hbar\partial_{y})}{\varsigma(\hbar\partial_{y})}\,\ln y\bigg)\Big|_{y = G^\vee_{0,1}(w)} 
\\ & \qquad 
\cdot [u^{r}]\,\,\frac{w\partial_w G^\vee_{0,1}(w)}{\hbar}	\,.
	\end{split}
\end{equation*}
Since the last factor of this expression does not depend on $u$, the only value of $r$ from the sum over $r\geq 0$ that contributes non-trivially is $r=0$. That is, the latter expression is equal to 
\begin{equation*}
\begin{split}
& \int_0^X \frac {\dd X}{X} \sum_{k \geq 1} X^{k} \cdot [w^{k}]\,\,\exp\bigg(k\,\frac{\varsigma(k\hbar\partial_{y})}{\varsigma(\hbar\partial_{y})}\,\ln y\bigg)\Big|_{y = G^\vee_{0,1}(w)} 
\frac{w\partial_w G^\vee_{0,1}(w)}{\hbar}
\\
& = \int_0^X \frac {\dd X}{X} \sum_{k \geq 1} X^{k} \cdot [w^{k}]\,\,(G^\vee_{0,1}(w))^k\exp\bigg(k\,\frac{\varsigma(k\hbar\partial_{y})}{\varsigma(\hbar\partial_{y})}\,\ln y-k\ln y\bigg)\Big|_{y = G^\vee_{0,1}(w)} 
\\
& \qquad \cdot 
\frac{w\partial_w G^\vee_{0,1}(w)}{\hbar},	
\end{split}
\end{equation*}
and we can continue the computation by isolating the factor $(G^\vee_{0,1}(w))^k$ and using the trick given in Equation~\eqref{eq:RemovingDependenceOnPower}. Thus, we have:
\begin{equation*}
	\begin{split}
		\eqref{eq:TowardsDelta} 
		& = 	\int_0^X \frac {\dd X}{X} \sum_{j \geq 0} (X\partial_X)^j \cdot [v^j] \sum_{k \geq 1} X^{k} \cdot [w^{k}] (G^\vee_{0,1}(w))^k
		\\ & \quad \cdot \exp\bigg(v\,\frac{\varsigma(v\hbar\partial_{y})}{\varsigma(\hbar\partial_{y})}\,\ln y-v\ln y\bigg)\Big|_{y = G^\vee_{0,1}(w)} 
		\frac{w\partial_w G^\vee_{0,1}(w)}{\hbar}	\\
		& = \frac 1\hbar	\int_0^X \frac {\dd X}{X} \sum_{j \geq 0} (X\partial_X)^j  [v^j]  \exp\bigg(v\,\frac{\varsigma(v\hbar\partial_{y})}{\varsigma(\hbar\partial_{y})}\,\ln y-v\ln y\bigg)\Big|_{y = G^\vee_{0,1}(w)} 
		\\ & \quad \cdot
		P^\vee(w)w\partial_w G^\vee_{0,1}(w)
			\\
		& = \frac 1\hbar	\int_0^X \frac {\dd X}{X} \sum_{j \geq 1} (X\partial_X)^j  [v^j]  \exp\bigg(v\,\frac{\varsigma(v\hbar\partial_{y})}{\varsigma(\hbar\partial_{y})}\,\ln y-v\ln y\bigg)\Big|_{y = G^\vee_{0,1}(w)} 
		\\
		& \quad \cdot 	{P^\vee(w)}w\partial_w G^\vee_{0,1}(w)
			+ \frac 1\hbar	\int_0^X \frac {\dd X}{X} 
		{P^\vee(w)}w\partial_w G^\vee_{0,1}(w)
		\,.
	\end{split}
\end{equation*}
Having written it this way, we can directly integrate (recall that $X=w/G^\vee_{0,1}(w)$ and $P^\vee(w)w\partial_w = X\partial_X$). The result of the integration of the first summand gives
\begin{equation*}
	\begin{split}
		\frac 1\hbar \bigg(\sum_{j \geq 1} (X\partial_X)^{j-1}  [v^j]  \exp\bigg(v\,\frac{\varsigma(v\hbar\partial_{y})}{\varsigma(\hbar\partial_{y})}\,\ln y-v\ln y\bigg)\Big|_{y = G^\vee_{0,1}(w)} X\partial_X G^\vee_{0,1}(w)\bigg) \bigg\vert^X_0
	\end{split}
\end{equation*}
which is manifestly $\sum_{g \geq 1} \hbar^{2g-1}\Delta^\vee_g(X)$ as the value at $X=0$ of this expression is equal to $0$. The integration of the second term gives
\begin{equation*}
 \frac 1\hbar (G^\vee_{0,1}(w)-G^\vee_{0,1}(0))=\frac 1\hbar (G^\vee_{0,1}(w)-1)\,.	
\end{equation*}
This completes the proof of Theorem~\ref{thm:R-transform-HigherGenera}.

\subsection{\texorpdfstring{Genus $0$: proof of Theorem~\ref{thm:R-transform-GenusZero}}{Genus 0: proof of Theorem~\ref{thm:R-transform-GenusZero}}}

\label{proofg0}

Let $f,f^{\vee} \colon PS \rightarrow R$ be two multiplicative functions such that $f = \zeta \ast f^{\vee}$. As described at the end of Section~\ref{npointSec}, we associate to $f^{\vee}$ a topological partition function $Z^{\vee}$ whose $n$-point functions $G_{g,n}^{\vee}$ vanish for genus $g > 0$. Then, we \emph{define} a topological partition function $Z = \mathsf{D}Z^{\vee}$. The corresponding multiplicative function $\Phi_{Z,\hbar} \colon PS \rightarrow R\llbracket\hbar\rrbracket$ is a priori different from $f$, but thanks to Corollary~\ref{thm:starproduct} its genus-zero part is
\[
\Phi_{Z}^{[0]} = \zeta \ast \Phi_{Z^{\vee}}^{[0]}= \zeta \ast f^{\vee} = f\,.
\]
Therefore, the $n$-point functions $G_{0,n}$ and $G_{0,n}^{\vee}$, associated to $f$ and $f^{\vee}$ respectively, are related by the restriction to genus $0$ of the expressions found in Section~\ref{proofg}.

Let us come back to \eqref{Ggnpasstilde} (which is valid for $n \geq 2$) and specialise it to genus $0$. Recall the $\hbar$-expansion of the $n$-point functions \eqref{giGn}. We need to isolate the leading order in $\hbar$ in the building blocks from Definition~\ref{def:weights}:
\begin{equation*} 
\begin{split} 
\mathsf{c}^{\vee}(u_I,w_I) & = \hbar^{2\# I - 2}\Big(\prod_{i \in I} u_i\Big) \widetilde{G}_{0,\# I}^{\vee}(w_I) + O(\hbar^{2\# I - 1}) \\
\vec{\mathsf{O}}^{\vee}(w_i) & = \sum_{m \geq 0} (P^{\vee}(w_i)w_i\partial_{w_i})^{m}  P^{\vee}(w_i) \cdot 
 [v_i^m] \sum_{r \geq 0} \Big(\partial_{y} + \frac{v_i}{y}\Big)^{r} 1 \Big|_{y = G_{0,1}^{\vee}(w_i)}\cdot [u_i^{r}]\,\,(u_i\hbar)^{-1} \\ & \quad + O(\hbar^0) \\
& = \hbar^{-1}\sum_{r \geq 0} \vec{\mathsf{O}}_{r}^{\vee}(w_i) \cdot [u_i^{r + 1}] + O(\hbar^0)\,,
\end{split}
\end{equation*}
where, in the last line, we refer to Definition~\ref{Or2}. Substituting this in \eqref{Ggnpasstilde}, we find that for a particular graph $\Gamma\in \mathcal{G}_n$ the minimal degree of $\hbar$ in \eqref{Ggnpasstilde} is equal to
$$
-n+\sum_{I \in \mathcal{I}(\Gamma)}(2 \# I -2)\,.
$$
This value is minimal for trees, for which it is equal to $n-2$ (compare with \eqref{nmoins1}), hence they give the only contribution to the genus-zero part of $\widetilde{G}_{n}$. Note also that trees have no non-trivial automorphisms. The variable $u_i$ only appears from the hyperedge contributions, and its power is the valency of the $i$-th white vertex. Therefore, the extraction of powers of $u_i$ prescribed by the operators at the vertices restricts the sum to the set $\mathcal{G}_{0,n}(\mathbf{r} + 1)$ of trees where the $i$-th white vertex has valency $r_i + 1$. All in all, \eqref{Ggnpasstilde} in genus $0$ and for $n \geq 2$ becomes 
\begin{equation}
\label{Ggn0pass}
\widetilde{G}_{0,n}(X_1,\ldots,X_n) = \sum_{r_1,\ldots,r_n \geq 0} \prod_{i = 1}^n \vec{\mathsf{O}}_{r_i}^{\vee}(w_i) \sum_{T \in \mathcal{G}_{0,n}(\mathbf{r} + 1)} \prod_{I \in \mathcal{I}(T)} \widetilde{G}_{0,\# I}^{\vee}(w_I)
\end{equation}
with the substitutions \eqref{XPveesub}, as announced in Theorem~\ref{thm:R-transform-GenusZero}. For $n = 2$, there is only one tree in $\mathcal{G}_{0,n}$, namely the one where the two white vertices are connected to a single black vertex. Then, only $\vec{\mathsf{O}}_{0}^{\vee}(w_i) = P^{\vee}(w_i)$ for $i = 1,2$ contribute to \eqref{Ggn0pass}:
\[
\widetilde{G}_{0,2}(X_1,X_2) = P^{\vee}(w_1)P^{\vee}(w_2)\widetilde{G}_{0,2}^{\vee}(w_1,w_2)\,.
\]
Coming back to the non-shifted $2$-point functions via \eqref{npointfun} yields the $(0,2)$ case of Theorems~\ref{thm:R-transform-HigherGenera} and ~\ref{thm:R-transform-GenusZero}.

It remains to treat the $(g,n) = (0,1)$ case, and for this we return to Lemma~\ref{lem:KeyComb}. Compared to the previous argument, only the weight of the white vertices is different (it is given by \eqref{U2Xi} instead of \eqref{O2Xi}), and its leading order in $\hbar$ is:
\begin{equation}
\label{Ugenus0}
\begin{split}
 \vec{\mathsf{U}}^{\vee}(X_i) 
&= \sum_{k \in \mathbb{Z}} X_i^k \cdot [w_i^k]\,\,\sum_{r \geq 0} \partial_y^r (1 + y)^{k}\big|_{y = 0} \cdot [u_i^{r}]\,\,\frac{\exp\big(u_i(G_{0,1}^{\vee}(w_i) - 1)\big)}{u_i\hbar} + O(\hbar^0) \\
&  = \hbar^{-1}\sum_{k \geq 0} X_i^k\cdot [w_i^k]\,\,\sum_{r = 0}^{k} \frac{k!}{(k - r)!}  \sum_{l \geq -1} \frac{(G_{0,1}^{\vee}(w_i) - 1)^{l + 1}}{(l + 1)!}\cdot [u_i^{r - l}] + O(\hbar^0)\,.
\end{split}
\end{equation}
It is still true that only the trees will contribute to \eqref{eqn:KeyComb} for $g =0$, and for $n = 1$ there is a single tree, namely the one without hyperedges. Therefore, we obtain:
\begin{equation*}
\begin{split}
G_{0,1}(X_1)-1 & = \widetilde{G}_{0,1}(X_1)-1 
\\
& = \sum_{k \geq 0} X_1^k \cdot [w_1^k]\,\,\sum_{r \geq 0} \frac{k!}{(k - r)!} \frac{(G_{0,1}^{\vee}(w_1) - 1)^{r + 1}}{(r + 1)!} \\
& = \sum_{k \geq 0} X_1^k \cdot [w_1^k]\,\,\sum_{r = 0}^{k} \frac{k!}{(k - r)! r!} \frac{(G_{0,1}^{\vee}(w_1) - 1)^{r + 1}}{r + 1} \\
& = \sum_{k \geq 0} X_1^k\cdot [w_1^k] \,\,\frac{(G_{0,1}^{\vee}(w_1))^{k + 1}-1}{k + 1} \\
& = G_{0,1}^{\vee}(w_1)-1\,,\qquad {\rm where}\quad X_1 = \frac{w_1}{G_{0,1}^{\vee}(w_1)}\,,
\end{split}
\end{equation*}
with the help of the Lagrange inversion formula in the last line. This also completes the proof of the $(0,1)$ case, and of Theorems~\ref{thm:R-transform-HigherGenera} and~\ref{thm:R-transform-GenusZero}.

\begin{remark} To show that the last line is reduced to the Lagrange inversion formula, we match the constant terms and apply $X_1\partial_{X_1}$ on both sides. On the one hand, we have:
\begin{align*}
X_1\partial_{X_1}\sum_{k \geq 0} X_1^k\cdot [w^k] \,\,\frac{(G_{0,1}^{\vee}(w))^{k + 1}-1}{k + 1} 
&  = \sum_{k \geq 0} X_1^k\cdot [w^k] \,\, w\partial_{w}\frac{(G_{0,1}^{\vee}(w))^{k + 1}-1}{k + 1} 
\\
& = \sum_{k \geq 0} X_1^k\cdot [w^k] \,\,
(G_{0,1}^{\vee}(w))^{k} w\partial_{w}(G_{0,1}^{\vee}(w))
\end{align*}
On the other hand, under the condition $X_1 = \frac{w_1}{G_{0,1}^{\vee}(w_1)}$, we have
\begin{align*}
X_1\partial_{X_1} (G_{0,1}^{\vee}(w_1)-1) &  = 
\frac{\dd\ln w_1}{\dd\ln X_1} w_1\partial_{w_1} (G_{0,1}^{\vee}(w_1)-1)
\\ 
& = \frac{\dd\ln w_1}{\dd\ln X_1} w_1\partial_{w_1}(G_{0,1}^{\vee}(w_1))\,,
\end{align*}
which allows us to apply Equation~\eqref{eq:LagrangeInversion} to identify these two expressions.
\end{remark}

\begin{remark}
We encourage the readers who want to understand better the three tricks of Section~\ref{proofg} to derive by themselves the $(0,2)$ case of Lemma~\ref{lem:KeyComb} as we did here with $(0,1)$.
\end{remark}

\subsection{\texorpdfstring{Genus $0$ coefficient-wise: proof of Theorem~\ref{coeffThm}}{Genus 0 coefficient-wise: proof of Theorem~\ref{coeffThm}}}
\label{coeffgenus0}
Our last task is to extract the coefficient of the monomial $\prod_{i = 1}^n X_i^{k_i}$ in $\widetilde{G}_{0,n}(X_1,\ldots,X_n)$. For $k_i > 0$ this coefficient is the same in $G_{0,n}(X_1,\ldots,X_n)$ and was called $F_{0;k_1,\ldots,k_n}$ in \eqref{npointfun}. We return to the specialisation of Lemma~\ref{lem:KeyComb} in genus $0$. The arguments already at work in Section~\ref{proofg0} together with the leading order \eqref{Ugenus0} of the operators $\vec{\mathsf{U}}^{\vee}$ yield:
\[
\begin{split}
 F_{0;k_1,\ldots,k_n} = 
& \quad \Big[\prod_{i = 1}^n w_i^{k_i}\Big] \sum_{\substack{r_1,\ldots,r_n \geq 0 \\ \ell_1,\ldots,\ell_n \geq -1}} \prod_{i = 1}^n \frac{k_i!}{(k_i - r_i)!}\,\frac{(G_{0,1}^{\vee}(w_i) - 1)^{\ell_i + 1}}{(\ell_i + 1)!} 
\\ & \quad \cdot \sum_{T \in \mathcal{G}_{0,n}(\mathbf{r} - \boldsymbol{\ell})}  \prod_{I \in \mathcal{I}(T)} \widetilde{G}_{0,\# I}^{\vee}(w_I)\,,
\end{split} 
\]
where $\frac{k_i!}{(k_i - r_i)!}$ is understood as $r_i! {k_i \choose r_i}$, which vanishes for $r_i > k_i$. For a given $n$-tuple $\boldsymbol{\ell}$, we can add $\ell_i + 1$ univalent black vertices to the $i$-th vertex. Then, the factor $(G_{0,1}^{\vee}(w_i) - 1)^{\ell_i+1}$ becomes part of the product over hyperedges. The $(\ell_i+1)!$ in the denominator accounts for the automorphisms of the tree at the $i$-th vertex, and we recognise Theorem~\ref{coeffThm}.

\subsection{Dual formulations}
\label{SecDual}

We can easily obtain dual functional relations expressing $G_n^{\vee}$ in terms of $G_n$. Writing $Z^{\vee} = \mathsf{D}^{-1}Z$, our starting point would be formula \eqref{DJD} with a minus sign. The only difference, apart from exchanging the roles of $X_i$ and $w_i$, is an opposite sign for $v_i$ in the operator weight \eqref{O2Xi} for white vertices. Noticing that $\varsigma$ is even, the operator weights are now:
\begin{equation*}
\begin{split}
\vec{\mathsf{O}}(X_i) & = \sum_{m \geq 0} (P(X_i)X_i\partial_{X_i})^m P(X_i) \\
& \quad \cdot [v_i^m]\,\,\sum_{r \geq 0} \Big(\partial_y - \frac{v_i}{y}\Big)^r \exp\bigg( -v_i \frac{\varsigma(\hbar v_i \partial_y)}{\varsigma(\hbar \partial_{y})}\ln y + v_i \ln y\bigg)\Big|_{y = G_{0,1}(X_i)} \\
& \quad \cdot [u_i^{r}]\,\,\frac{\exp\big(\hbar u_i \varsigma(\hbar u_i X_i \partial_{X_i})(G_1(X_i) - \hbar^{-1}) - u_i(G_{0,1}(X_i) - 1)\big)}{\hbar u_i \varsigma(\hbar u_i)}\,,
\end{split}
\end{equation*}
and the hyperedge weights are
\[
\mathsf{c}(u_I,X_I) = \left\{\begin{array}{lll} \big(\prod_{i \in I} \hbar u_i \varsigma(\hbar u_i X_i \partial_{X_i})\big) \widetilde{G}_{\# I}(X_I) & & {\rm if}\,\,I \neq \{j,j\}\,, \\[1ex] \big(\hbar u_j \varsigma(\hbar u_j X_j \partial_{X_j}) \hbar u_j \varsigma(\hbar u_j X_j' \partial_{X_j'})\big) \,G_2(X_j,X_j')|_{X_j = X_j'} & & {\rm if}\,\,I = \{j,j\}\,. \end{array}\right.
\]
With the substitutions
\[
w_i = X_iG_{0,1}(X_i)\,,\qquad P(X_i) = \frac{\dd \ln X_i}{\dd \ln w_i} = \frac{1}{P^{\vee}(w_i)}\,,
\]
the dual of Theorem~\ref{thm:R-transform-HigherGenera} for $2g - 2 + n > 0$ is
\begin{equation*}
	\begin{split}
		&
		G_{g,n}^{\vee}(w_1,\ldots,w_n) = 
		\\ & \quad 
		\delta_{n,1}\Delta_g(w_1) + [\hbar^{2g - 2 + n}]\,\,\sum_{\Gamma \in \mathcal{G}_n} \frac{1}{\# {\rm Aut}(\Gamma)} \prod_{i = 1}^n \vec{\mathsf{O}}(X_i) \prod_{I \in \mathcal{I}(\Gamma)} \mathsf{c}(u_I,X_I)\,,
	\end{split}
\end{equation*}
with the correction term for $n = 1$ now given by
\begin{equation*}
\begin{split}
\Delta_g(w_1) & = [\hbar^{2g}]\,\,\sum_{m\geq 0} \big(P(X_1) X_1\partial_{X_1}\big)^{m} \,
 [v_1^{m+1}]\\
 & \quad  \exp\bigg(-v_1\,\frac{\varsigma(\hbar v_1 \partial_{y})}{\varsigma(\hbar \partial_{y})} \ln y + v_1 \ln y\bigg) \Big|_{y = G_{0,1}(X_1)} \cdot
  P(X_1) X_1\partial_{X_1}  G_{0,1}(X_1)\,.
\end{split}
\end{equation*}

In genus $0$, the relevant operator weight is
\[
\vec{\mathsf{O}}_{r_i}(X_i) = \sum_{m \geq 0} (P(X_i)X_i\partial_{X_i})^m P(X_i) \cdot [v_i^m]\,\,\bigg(\partial_y - \frac{v_i}{y}\bigg)^{r_i}\cdot 1 \Big|_{y = G_{0,1}(X_i)}\,,
\]
and the dual of Theorem~\ref{thm:R-transform-GenusZero} reads, for $n \geq 3$:
\[
G_{0,n}^{\vee}(w_1,\ldots,w_n) = \sum_{r_1,\ldots,r_n \geq 0} \prod_{i = 1}^n \vec{\mathsf{O}}_{r_i}(X_i) \sum_{T \in \mathcal{G}_{0,n}(\mathbf{r} + 1)} \prod_{I \in \mathcal{I}(T)}' G_{0,\# I}(X_I)\,,
\]
and $\prod'$ means that each occurrence of $G_{0,2}(X_i,X_j)$ for $i \neq j$ should be replaced with $\widetilde{G}_{0,2}(X_i,X_j)$.

To get the coefficient-wise relation dual to Theorem~\ref{coeffThm}, we observe that the operator \eqref{Ugenus0} is replaced  at leading order with
\begin{equation*}
\begin{split}
\vec{\mathsf{U}}(w_i) & = \sum_{k \geq 1} w_i^{k} \cdot [X_i^k]\,\,\sum_{r \geq 0} \partial_y^r (1 + y)^{-k}\big|_{y = 0} \cdot [u_i{^r}] \frac{\exp\big(u_i(G_{0,1}(X_i) - 1)\big)}{u_i\hbar} + O(\hbar^0) \\
& =  \hbar^{-1} \sum_{k \geq 1} w_i^k \cdot [X_i^k]\,\,\sum_{r \geq 0} \frac{(-1)^r (r + k - 1)!}{(k - 1)!} \sum_{l \geq -1} \frac{(G_{0,1}(X_i) - 1)^{l + 1}}{(l + 1)!} \cdot [u_i^{r - l}] + O(\hbar^0)\,.
\end{split}
\end{equation*}
The only difference with \eqref{Ugenus0} is the combinatorial factor, which is non-vanishing for all nonnegative $r$. This leads to the formula
\[
\begin{split}
F_{0;k_1,\ldots,k_n}^{\vee} & = \Big[\prod_{i = 1}^n X_i^{k_i}\Big] \sum_{r_1,\ldots,r_n \geq 0} \prod_{i = 1}^n \frac{(-1)^{r_i}(r_i + k_i - 1)!}{(k_i - 1)!} 
\\
& \quad \cdot \sum_{T \in \mathcal{T}_n(\mathbf{r} + 1)} \frac{\prod''_{I \in \mathcal{I}(T)} G_{0,\# I}(X_I)}{\# {\rm Aut}(T)}\,,
\end{split}
\]
where $\prod''$ means that one should replace each occurrence of $G_{0,1}(X_j)$ with $G_{0,1}(X_j) - 1$, and each occurrence of $G_{0,2}(X_i,X_j)$ for $i \neq j$ with $\widetilde{G}_{0,2}(X_i,X_j)$. The form is in fact similar to Theorem~\ref{coeffThm} when one observes that
\[
\frac{k!}{(k-r)!} = r! \,\binom{k}{r}\,,\qquad \frac{(-1)^{r}(r + k - 1)!}{(k - 1)!} = r! \,\binom{-k}{r}\,.
\]

\section{Applications to free probability}
\label{SecAppl}

We now present the interpretation of the results of Section~\ref{Sec3} in the context of (higher-order) free probability. For a general introduction to free probability we refer to \cite{MingoSpeicher}.

\subsection{Preliminary: decorated partitioned permutations}
\label{Decpartperm}
If $\mathscr{A}$ is an associative algebra, we can consider partitioned permutations decorated by elements in $\mathscr{A}$. It means we want to study the set $PS(\mathscr{A}) \coloneqq \bigcup_{d \geq 1} PS(d) \times \mathscr{A}^d$. If $f_1\colon PS \rightarrow R$ and $f_2\colon PS(\mathscr{A}) \rightarrow R$ are two functions, their convolution is defined by
\begin{equation*}
\begin{split} 
(f_1 \ast f_2)(\mathcal{C},\gamma)[a_1,\ldots,a_d] & \coloneqq (f_1\ast f_2[a_1,\ldots,a_d])(\mathcal{C},\gamma) \\
&  =  \sum_{(\mathcal{A},\alpha) \cdot (\mathcal{B},\beta)  = (\mathcal{C},\gamma)}  f_1(\mathcal{A},\alpha)\,f_2(\mathcal{B},\beta)[a_1,\ldots,a_d] \,,
\end{split}
\end{equation*} 
for $(\mathcal{C},\gamma) \in PS(d)$ and $a_1,\ldots,a_d \in \mathscr{A}$. A similar definition can be made for the extended convolution $f_1 \circledast f_2$. A function $f\colon PS(\mathscr{A}) \rightarrow R$ is \emph{multiplicative} if for any $d > 0$ and $a_1,\ldots,a_d \in \mathscr{A}$:
\begin{itemize}
\item for any $\sigma,\pi \in S(d)$, we have $$f(\mathbf{1}_{d},\pi^{-1} \circ \sigma \circ \pi)[a_{\pi(1)},\ldots,a_{\pi(d)}] = f(\mathbf{1}_d,\sigma)[a_1,\ldots,a_d];$$
\item for any $(\mathcal{A},\alpha) \in PS(d)$, we have 
$$
f(\mathcal{A},\alpha)[a_1,\ldots,a_d] = \prod_{A \in \mathcal{A}} f(1_{\# A},\alpha_{|A})[(a_i)_{i \in A}],
$$ 
where bijections $[\#A] \rightarrow A$ have been chosen to make sense of the right-hand side, which is independent of this choice due to the first condition.
\end{itemize}

\subsection{Higher-order free probability}
\label{Sec:Higherorderfree}

\begin{definition}
Let $\mathscr{A}$ be an associative algebra. We say that an $n$-linear form $\tau \colon \mathscr{A}^n \rightarrow \mathbb{C}$ is \emph{tracial} if
\[
\begin{split}
& \forall b,a_1,\ldots,a_n \in \mathscr{A},\quad \forall i \in [n],\qquad
\\ &  \tau(a_1,\ldots,a_ib,\ldots,a_n) = \tau(a_1,\ldots,ba_i,\ldots,a_n)\,.
\end{split}
\]
\end{definition} 

\begin{definition}
A \emph{higher-order (non-commutative) probability space} (HOPS) is the data $(\mathscr{A},\boldsymbol{\varphi})$ consisting of a unital associative (maybe non-commutative) algebra $\mathscr{A}$ over $\mathbb{C}$ and a family $\boldsymbol{\varphi} = (\varphi_n)_{n \geq 1}$ of tracial symmetric $n$-linear forms such that 
$$
\varphi_1(1) = 1 \quad \text{and} \quad  
\varphi_n(1,a_2,\ldots,a_n) = 0\,,
$$ 
for any $n \geq 2$ and $a_2,\ldots,a_n \in \mathscr{A}$.
 \end{definition}

Given a  HOPS $(\mathscr{A},\boldsymbol{\varphi})$, there is a natural multiplicative function $$\phi\colon PS(\mathscr{A}) \rightarrow \mathbb{C}$$ encoding the (higher-order) moments, specified for $\lambda \vdash d$ of length $n$ by
\begin{equation}
\label{phidef1d}
\phi(\mathbf{1}_d,\pi_{\lambda})[a_1,\ldots,a_d] = \varphi_{n}\Big(\prod_{j = 1}^{\lambda_1} a_j\,,\, \prod_{j = 1}^{\lambda_2} a_{\lambda_1 + j}\,,\ldots,\,\prod_{j = 1}^{\lambda_{n}} a_{\lambda_1 + \cdots + \lambda_{n - 1} + j}\Big)\,.
\end{equation}
We then define another multiplicative function $\phi^{\vee} \colon PS(\mathscr{A})\rightarrow \mathbb{C}$ by
\[
\phi^{\vee} \coloneqq \mu \ast \phi \qquad  \Longleftrightarrow\qquad   \phi = \zeta \ast \phi^{\vee}\,.
\]

\begin{definition} (Adapted\footnote{In \cite{CMSS}  the convolution of $f_1$ with $f_2$ is defined when $f_1$ is a function on $PS(\mathscr{A})$ and $f_2$ a function on $PS$, and the higher-order free cumulants are defined by right-convolution of $\phi$ with the M\"obius function on $PS$. In the present paper we rather introduced (Section~\ref{Decpartperm}) the convolution $f_1 \ast f_2$ when $f_1$ is a function on $PS$ and $f_2$ a function on $PS(\mathscr{A})$, and define the higher-order free cumulants by left-convolution of $\phi$ (defined from a HOPS in \eqref{phidef1d}) with $\mu$. As $\mu$ is multiplicative and in light of Lemma~\ref{commutmult}, these definitions are equivalent.} from \cite{CMSS}) The free cumulants at order $n$ are encoded in the collection of $d$-linear forms $\kappa_{k_1,\ldots,k_{n}}\colon \mathscr{A}^{d} \rightarrow \mathbb{C}$ indexed by $k_1,\ldots, k_n > 0$ such that $d = k_1+\cdots+k_n$, given by:
\[
\kappa_{k_1,\ldots,k_{n}}(a_1,\ldots,a_{d}) = \phi^{\vee}(\mathbf{1}_{d},\pi_{\lambda(\mathbf{k})})[a_1,\ldots,a_d]\,.
\]
As in Section~\ref{Sec:PS}, $\lambda(\mathbf{k})$ stands for the partition of $d$ of length $n$ consisting of putting $k_1,\ldots,k_n$ in the weakly decreasing order.
\end{definition}

If $a \in \mathscr{A}$, let us define the generating series of moments and cumulants at order $n$ of $a$:
\begin{equation}
\begin{split}
M_n(X_1,\ldots,X_{n}) & = \delta_{n,1} + \sum_{k_1,\ldots,k_n > 0} \varphi_n(a^{k_1},\ldots,a^{k_n}) \prod_{i =1}^n X_i^{k_i}\,,  \\
C_n(w_1,\ldots,w_n) & = \delta_{n,1} + \sum_{k_1,\ldots,k_n > 0} \kappa_{k_1,\ldots,k_n}(a,\ldots,a) \prod_{i = 1}^n w_i^{k_i} \,,
\end{split}
\end{equation}
and let $M_1 = M$, $C_1 = C$. At first order, a fundamental result of Speicher \cite{Spei94} states that
\[
C(XM(X)) = M(X)\,.
\] 
A second equivalent formulation is that $W(x) = \frac{1}{x}M(\frac{1}{x})$ and $X(w) = \frac{1}{w} C(w)$ are functional inverses of each other
$$
X(W(x)) = x\,.
$$
A third equivalent formulation using the $R$-transform defined by $C(w) = 1 + wR(w)$ is
$$
\frac{1}{M(X)} + XR(XM(X)) = 1\,.
$$
Collins, Mingo, \'Sniady and Speicher established in \cite{CMSS} the formula at second order:
\[
\begin{split}
M_2(X_1,X_2) & = \frac{X_1 \dd w_1}{\dd X_1}\,\frac{X_2 \dd w_2}{\dd X_2} \bigg(\frac{C_2(w_1,w_2)}{w_1w_2} + \frac{1}{(w_1 - w_2)^2}\bigg) - \frac{X_1X_2}{(X_1 - X_2)^2}\,,
\\
w_i & = X_iM(X_i)\,,	
\end{split}
\]
and asked for the generalisation of these functional relations to order $n\geq 3$. Converting the notations into $G_{0,n} = M_n$, $G_{0,n}^{\vee} = C_n$, Theorem~\ref{thm:R-transform-GenusZero} answers this question (the formulae for $n = 1$ and $n = 2$ match). Knowing the formulae, it seems rather complicated to prove them directly from the combinatorics of partitioned permutations --- even in the $n = 3$ case \eqref{03form}. We had to make a detour via the master relation (Theorem~\ref{thm:hbarstarproduct}) and algebraic manipulations in the bosonic Fock space in order to derive the results.

\subsection{Surfaced permutations}
\label{Sec:PSG}
Another way to work with the extended product is to add the data of a genus function to partitioned permutations, i.e.~an arbitrary assignment of a number to each block $A$ of the partition $\mathcal{A}$. If the number is a nonnegative integer, we get the notion of surfaced permutations proposed in \cite[Appendix]{CMSS}. Instead of nonnegative integers, we are going to allow nonnegative half-integers. It has a minimal effect on the presentation (the reader uncomfortable with this can restrict to integers, see also Remark~\ref{evenremark}) while allowing infinitesimal freeness to find a natural place in the framework (see end of Section~\ref{SecFreehig}).

\begin{definition}\label{def:PSG}
A \emph{surfaced permutation} of $[d]$ is a triple $(\mathcal{A},\alpha,g)$ where $(\mathcal{A},\alpha)\in PS(d)$ and  $g \colon \mathcal{A} \rightarrow \frac{1}{2}\mathbb{Z}_{\geq 0}$ is a function. We denote $\PSG(d)$ the set of surfaced permutations of $[d]$, and $\PSG = \bigcup_{d \geq 1} \PSG(d)$.
\end{definition}
The colength of $(\mathcal{A},\alpha,g)\in \PSG(d)$ is defined by:
\[ 
|(\mathcal{A},\alpha,g)| \coloneqq |(\mathcal{A},\alpha)| + \sum\limits_{A\in\mathcal{A}} 2 g(A)\,.
\]
\begin{definition}
\label{def:extprodsurf}The extended product of $(\mathcal{A},\alpha,g)$, $(\mathcal{B},\beta,h) \in \PSG(d)$ is defined as $(\mathcal{A},\alpha,g) \odot (\mathcal{B},\beta,h) = (\mathcal{A}\vee \mathcal{B}, \alpha \circ \beta,k)$ in which the genus function at $C \in \mathcal{A} \vee \mathcal{B}$ takes the value:
\begin{equation}
\label{eq:prod:PSG}
k(C)\coloneqq \frac{|(\mathcal{A}_{|C},\alpha_{|C},g_{|C})|+|(\mathcal{B}_{|C},\beta_{|C},h_{|C})|-|(\mathcal{A}\vee \mathcal{B}\,{}_{|C},\alpha \circ \beta{}\,{}_{|C})|}{2}\ .
\end{equation} 
Here, if $\mathcal{D} = \{D_1,\ldots,D_l\} \in P(d)$ and $C \subseteq [d]$, the notation $\mathcal{D}_{|C}$ stands for $\{D_1 \cap C,\ldots,D_l \cap C\}$ from which one removes the elements which are empty sets. The convolution of two functions $\mathbbm{f}_1,\mathbbm{f}_2 \colon \PSG(d) \rightarrow R$ is
\[
(\mathbbm{f}_1 \circledast \mathbbm{f}_2) (\mathcal{C},\gamma,k) = \sum\limits_{(\mathcal{A},\alpha,g) \odot (\mathcal{B},\beta,h) = (\mathcal{C},\gamma,k)} \mathbbm{f}_1(\mathcal{A},\alpha,g) \,\, \mathbbm{f}_2(\mathcal{B},\beta,h)\,.
\]
It is easy to check that $\circledast$ is associative.
\end{definition}

\begin{remark}\label{Blockad} There are two alternative ways to think about the formula \eqref{eq:prod:PSG} for the genus. On the one hand, we observe that it is such that we have a block-additivity of the colength under products of surfaced permutations:
\[
|(\mathcal{A}_{|C},\alpha_{|C},g_{|C})| + |(\mathcal{B}_{|C},\beta_{|C},h_{|C})| = |(\mathcal{A} \vee \mathcal{B}\,{}_{|C},\alpha \circ \beta\,{}_{|C},k_{|C})|\,.
\]
On the other hand, we also have
\begin{equation}
\label{iniguzb}\begin{split}
k(C) & = \sum_{\substack{A \in \mathcal{A} \\ A \subseteq C}} g(A) + \sum_{\substack{B \in \mathcal{B} \\ B \subseteq C}} h(B) 
\\
& 
\quad + \frac{1}{2}\Big(|(\mathcal{A}_{|C},\alpha_{|C})| + |(\mathcal{B}_{|C},\beta_{|C})| - |(\mathcal{A} \vee \mathcal{B}\,{}_{|C},\alpha \circ \beta\,{}_{|C})|\Big)\, \\
& \geq \sum_{\substack{A \in \mathcal{A} \\ A \subseteq C}} g(A) + \sum_{\substack{B \in \mathcal{B} \\ B \subseteq C}} h(B) \,.
\end{split} 
\end{equation}
Since $|\alpha|+|\beta|-|\alpha \circ \beta|$ is nonnegative and even, the $\frac{1}{2}$ is multiplying an even integer, so results in an integer. We interpret this equation by saying that the product of surfaced permutations can create genus in integer units.
\end{remark}

The relation with the setting of Section~\ref{Sec:PSS} (justifying that we keep the same notations $\odot$ and $\circledast$) is that, if we associate to the multiplicative functions $\mathbbm{f}_{i}\colon \PSG \rightarrow R$, for $i = 1,2$, the multiplicative functions $\hat{\mathbbm{f}}_{i}\colon PS \rightarrow R\llbracket\hbar\rrbracket$ given for $(\mathcal{A},\alpha)$ by
\begin{equation}
\widehat{\mathbbm{f}}_{i}(\mathcal{A},\alpha) = \sum_{g \colon \mathcal{A} \rightarrow \frac{1}{2}\mathbb{Z}_{\geq 0}} \hbar^{|(\mathcal{A},\alpha,g)|}\, \mathbbm{f}_i(\mathcal{A},\alpha,g)\,,
\end{equation}
then we have
\[
\widehat{\,\,\mathbbm{f}_{1} \circledast \mathbbm{f}_{2}\,\,} = \widehat{\mathbbm{f}}_{1} \circledast \widehat{\mathbbm{f}}_{2}\,,
\]
where on the left-hand side (resp.~on the right-hand side) this is the extended convolution on $\PSG$ (resp.~$PS$). This is due to the block-additivity of the colength of surfaced permutations mentioned in Remark~\ref{Blockad}. 

We have an injection $\iota\colon PS(d) \rightarrow \PSG(d)$ consisting in completing a partitioned permutation by taking  zero as genus function. Any function $f$ on $PS(d)$ can be considered as a function on $\PSG(d)$ by extending it by $0$ outside $\iota(PS(d))$; we denote it $\iota_*f$. Clearly, $\bbdelta = \iota_*\delta$ is the unit for $\circledast$, while $\bbzeta = \iota_*\zeta$ is the zeta function, also characterised by $\hat{\bbzeta} = \zeta_{\hbar}$. By the previous discussion, it admits as inverse for $\circledast$ the function $\bbmu$ characterised by $\hat{\bbmu} = \mu_{\hbar}$, whose existence comes from Lemma~\ref{existe}. 

\begin{remark} Observe that $\PSG(d)$ admits a poset structure, by declaring that
\[
\begin{split}
(\mathcal{A},\alpha,g) \preceq (\mathcal{C},\gamma,k) \qquad \Longleftrightarrow  \qquad & \exists (\mathcal{B},\beta,h) \in \PSG(d),
\\
&  (\mathcal{A},\alpha,g) \odot (\mathcal{B},\beta,h) = (\mathcal{C},\gamma,k)\,.
\end{split}
\]
One could be tempted to invoke the general fact that posets admit M\"obius functions \cite{Rota64}. However, the zeta function associated to this poset structure is not the one we consider.
\end{remark}

\begin{definition}
A function $\mathbbm{f} \colon \PSG \rightarrow R$ is \emph{multiplicative} if  for any $(d,h) \in \mathbb{Z}_{> 0} \times \frac{1}{2}\mathbb{Z}_{\geq 0}$ and $\sigma \in S(d)$, $\mathbbm{f}(\mathbf{1}_d,\sigma,h)$ depends only on the conjugacy class of $\sigma$, and for any $(\mathcal{A},\alpha,g) \in \PSG$:
\[
\mathbbm{f}(\mathcal{A},\alpha,g)=\prod_{A \in\mathcal{A}} \mathbbm{f}(\mathbf{1}_{\#A},\alpha_{|A},g_{|A})\,.
\]
The previously introduced functions $\bbdelta$, $\bbzeta$ and $\bbmu$ are examples of multiplicative functions.
\end{definition}

\begin{remark} \label{evenremark} We say that $\mathbbm{f} \colon \PSG \rightarrow R$ is \emph{even} if $\mathbbm{f}(\mathcal{A},\alpha,g) = 0$ for any $(\mathcal{A},\alpha,g) \in \mathbb{PS}$ such that there exists a block $A \in \mathcal{A}$ having $g(A) \notin \mathbb{Z}_{\geq 0}$. Due to the last point in Remark~\ref{Blockad}, the extended convolution of two even functions is even. In the previous sections we only considered functions with integer genus: it fits with the present setting by restricting to  even functions on $\PSG$.
\end{remark}

We may also consider the analogue of the product $\cdot$ between surfaced permutations, by keeping only the cases in $\odot$ with no genus creation.

\begin{definition}
\label{def49}If $(\mathcal{A},\alpha,g),(\mathcal{B},\beta,h)$ are two surfaced permutations of $[d]$, we define their product $(\mathcal{A},\alpha,g) \cdot (\mathcal{B},\beta,h)$ by
\begin{equation*}  
\left\{\begin{array}{lll}
(\mathcal{A}\vee \mathcal{B},\alpha \circ \beta,k) & & \textup{if}\,\, |(\mathcal{A},\alpha)|+|(\mathcal{B},\beta)|= |(\mathcal{A}\vee \mathcal{B},\alpha \circ \beta)|, \\
0 && \textup{otherwise}. \end{array}\right.
\end{equation*}
Note that in the first case we have
\[
k(C) = \sum_{\substack{A \in \mathcal{A} \\ A \subseteq C}} g(A) + \sum_{\substack{B \in \mathcal{B} \\ B \subseteq C}} h(B)\,.
\]
The convolution between functions $\mathbbm{f}_1,\mathbbm{f}_2\colon \PSG(d) \rightarrow R$ is defined as
\[
(\mathbbm{f}_1 \ast \mathbbm{f}_2) (\mathcal{C},\gamma,k) = \sum_{(\mathcal{A},\alpha,g) \cdot (\mathcal{B},\beta,h) = (\mathcal{C},\gamma,k)} \mathbbm{f}_1(\mathcal{A},\alpha,g) \,\mathbbm{f}_2(\mathcal{B},\beta,h)\,.
\]
\end{definition}
The extraction of leading order in Lemma~\ref{lem25} can be upgraded to include the first subleading order (encoded in genus $\frac{1}{2}$).  To describe this, we need more notations.

\begin{definition}
We say that two multiplicative functions $\mathbbm{f}_1,\mathbbm{f}_2 \colon \PSG \rightarrow R$ agree infinitesimally if their values coincide on $(\mathcal{A},\alpha,g)$ for any $g \colon \mathcal{A} \rightarrow \frac{1}{2}\mathbb{Z}_{\geq 0}$ such that $\sum_{A \in \mathcal{A}} g(A) \leq \frac{1}{2}$. In that case we denote $\mathbbm{f}_1 \approx \mathbbm{f}_2$.
\end{definition}

\begin{lemma}
\label{lem121}
Let $\mathbbm{f}_{1},\mathbbm{f}_{2}\colon \PSG \rightarrow R$ be two multiplicative functions. The relation $\mathbbm{f}_1 = \bbzeta \circledast \mathbbm{f}_2$ implies the infinitesimal agreement $\mathbbm{f}_1 \approx \bbzeta \ast \mathbbm{f}_2$.
\end{lemma}
\begin{proof} Same as in Lemma~\ref{lem25}, taking into account that the creation of genus occurs by integer units only (Remark~\ref{Blockad}).
\end{proof}
\begin{remark} \label{dualrem} This has an equivalent presentation via the ring of dual numbers\footnote{This ring formalises in an algebraic way formal expansions in $\hbar$ up to (including) order $1$.} $R' = R\llbracket\hbar\rrbracket/\hbar^2R\llbracket\hbar\rrbracket$.  Namely, let us write:
\[
\widehat{\mathbbm{f}}_i(\mathcal{A},\alpha) = \hbar^{|(\mathcal{A},\alpha)|}\big(f_i(\mathcal{A},\alpha,0) + \hbar f'_i(\mathcal{A},\alpha) + o(\hbar)\big)\,,
\]
and define multiplicative functions
\[
{}^{\flat}f_i \colon PS \rightarrow R',\qquad {\rm by}\quad {}^{\flat} f_i(\mathcal{A},\alpha) = f_i(\mathcal{A},\alpha,0) + \hbar f_i'(\mathcal{A},\alpha)\,.
\]
Then, the relation  $\mathbbm{f}_1 = \bbzeta \circledast \mathbbm{f}_2$ between $R$-valued functions on surfaced permutations implies the relation ${}^{\flat} f_1 = \zeta \ast {}^{\flat} f_2$ between $R'$-valued functions on partitioned permutations. Observe that $f_i(-,-,0)$ is a multiplicative function on $PS$, but $f_i'$ is not. Instead, we have
\[
f_i'(\mathcal{A},\alpha) = \sum_{A \in \mathcal{A}} f'_i(\mathbf{1}_{\#A},\alpha_{|A}) \prod_{\substack{A' \in \mathcal{A} \\ A' \neq A}} f_i(\mathbf{1}_{\# A'},\alpha_{|A'},0) \,.
\]
\end{remark}

\medskip

\begin{remark} As in Section~\ref{Decpartperm}, we may also work with the set $\PSG(\mathscr{A})$ of surfaced permutations decorated by elements of an associative algebra $\mathscr{A}$.
\end{remark}

\begin{remark}
The notation $\bbmu$ will exclusively refer to the M\"obius function for the extended convolution product of Definition~\ref{def:extprodsurf}, i.e.~the inverse for $\circledast$ of the zeta function $\bbzeta$, which in terms of multiplicative functions $PS  \rightarrow R\llbracket\hbar\rrbracket$ is such that $\hat{\bbmu} = \mu_{\hbar}$. We are not going to consider the M\"obius function for the convolution product of Definition~\ref{def49}, i.e.~the inverse for $\ast$ of the zeta function $\bbzeta$, so we do not introduce a notation for it.
\end{remark}

\subsection{Extension of the main formulae}
\label{Sec:extdemi}
Section~\ref{Sec24} and Section~\ref{Sec3} extend without effort to topological partition functions allowing $g \in \frac{1}{2}\mathbb{Z}_{\geq 0}$, while the monotone Hurwitz numbers (implementing the master relation) are unchanged and have integer genus. In particular, the functional relations of Theorem~\ref{thm:R-transform-HigherGenera} for the $n$-point functions of multiplicative functions $\mathbbm{f},\mathbbm{f}^{\vee}\colon \PSG \rightarrow R$ satisfying  $\mathbbm{f} = \bbzeta \circledast \mathbbm{f}^{\vee}$ hold in the same form. Taking into account Lemma~\ref{lem121}, the proofs of Theorem~\ref{thm:R-transform-GenusZero} and Theorem~\ref{coeffThm} can be easily adapted to obtain functional relations in genus $0$ and $\frac{1}{2}$. 

\begin{definition}
Let $\mathcal{G}_{0,n}'$ be the set of bicoloured trees as in Definition~\ref{deftree} except that exactly one of their black vertices is special, the others being ordinary, and the hyperedge $I'$ corresponding to the special black vertex may be univalent. Let $\mathcal{T}_{n}'$ be the set of bicoloured trees obtained by connecting to a $T' \in \mathcal{G}_{0,n}'$ (see Definition~\ref{deftreebis}) finitely many univalent black vertices. In $\mathcal{G}_{0,n}'(\mathbf{r} + 1)$ and $\mathcal{T}_n'(\mathbf{r} + 1)$ we require the $i$-th vertex to have valency $r_i + 1$.
\end{definition}

\begin{remark} The special black vertex is preserved by the automorphisms of a tree $T\in \mathcal{T}_{n}'$ even if its hyperedge is univalent. Thus, if the hyperedge $I'$ of the special black vertex of $T$ is univalent and $f_{I'}^{-1}(1)=\{j\}$ (that is, the special black vertex is univalent and is connected to the $j$-th white vertex), and $\ell_i(T)$ is the total number of univalent black vertices incident to the $i$-th white vertex (including the special one), we have $\# {\rm Aut}(T) = \prod_{i = 1}^n (\ell_i(T)-\delta_{ij})!$.
\end{remark}
 
\begin{theorem}
\label{thmdemimain}Let $f,f' ,f^{\vee},f'{}^{\vee}\colon PS \rightarrow R$ be functions so that
\[
{}^{\flat} f = f + \hbar f' \colon PS \rightarrow R',\qquad {}^{\flat}f^{\vee} = f^{\vee} + \hbar f'{}^{\vee} \colon PS \rightarrow R'
\]
are multiplicative. Introduce the $n$-point functions 
\begin{equation}
\label{G12nc9}
\begin{split}
G_{0,n}(X_1,\ldots,X_n) & =  \delta_{n,1} + \sum_{k_1,\ldots,k_n > 0} f(\mathbf{1}_{k_1 + \cdots + k_n},\pi_{\lambda(\mathbf{k})}) \prod_{i = 1}^n X_i^{k_i}\,, \\
G_{\frac{1}{2},n}(X_1,\ldots,X_n) & = \sum_{k_1,\ldots,k_n > 0} f'(\mathbf{1}_{k_1 + \cdots + k_n},\pi_{\lambda(\mathbf{k})}) \prod_{i = 1}^n X_i^{k_i}\,,
\end{split}
\end{equation}
and likewise $G_{0,n}^{\vee}$ and $G_{\frac{1}{2},n}^{\vee}$. 

Suppose that we have ${}^{\flat}f = \zeta \ast {}^{\flat} f^{\vee}$. Then, the genus-zero functional relations given in Theorems~\ref{thm:R-transform-GenusZero} and~\ref{coeffThm} hold, and with the same substitutions and notations we have for any $n \geq 1$:
\begin{equation}
\label{G12nc}
\begin{split}
& G_{\frac{1}{2},n}(X_1,\ldots,X_n)  = 
\\
& \quad \sum_{r_1,\ldots,r_n \geq 0} \prod_{i = 1}^n \vec{\mathsf{O}}_{r_i}^{\vee}(w_i) \sum_{T \in \mathcal{G}_{0,n}'(\mathbf{r} + 1)}G^{\vee}_{\frac{1}{2},\# I'}(w_{I'}) \prod'_{I \neq I'}G_{0,\# I}^{\vee}(w_I)\,.
\end{split}
\end{equation}
Equivalently, for any $n \geq 1$, and $\lambda \vdash d$ of length $n$ we have:
\begin{equation}
\label{demicoeff}
\begin{split}
f'(\mathbf{1}_d,\pi_{\lambda}) & = \Big[\prod_{i = 1}^n w_i^{\lambda_i} \Big] \,\,\sum_{\substack{0 \leq r_i \leq \lambda_i \\ i \in [n]}} \prod_{i = 1}^n \frac{\lambda_i!}{(\lambda_i - r_i)!} 
\\
& \quad \cdot \sum_{T \in \mathcal{T}'_{n}(\mathbf{r} + 1)} \frac{G^{\vee}_{\frac{1}{2},\# I'}(w_{I'}) \prod''_{I \neq I'} G^{\vee}_{0,\# I}(w_I)}{\# {\rm Aut}(T)}\,.
\end{split}
\end{equation}
\end{theorem}
\begin{corollary}
\label{the1demi} We have
\begin{equation}
\label{iuniunasf} G_{\frac{1}{2},1}(X) = P^{\vee}(w)\,G_{\frac{1}{2},1}^{\vee}(w)\,,\quad X = \frac{w}{G_{0,1}^{\vee}(w)}\,,\qquad P^{\vee}(w) = \frac{\dd \ln w}{\dd \ln X}\,.
\end{equation}
Equivalently:
\[
G_{\frac{1}{2},1}(X) \frac{\dd X}{X} = G_{\frac{1}{2},1}^{\vee}(w) \frac{\dd w}{w}\,.
\]
\end{corollary}
\begin{proof}[Proof of Theorem~\ref{thmdemimain}]
Let $f_{\hbar}^{\vee} \colon PS \rightarrow R\llbracket\hbar\rrbracket$ be the unique multiplicative function which for any $d > 0$ and $\alpha \in S(d)$ satisfies
\[
f_{\hbar}^{\vee}(\mathbf{1}_d,\alpha) = \hbar^{|(\mathbf{1}_d,\alpha)|}\big(f^{\vee}(\mathbf{1}_d,\alpha) + \hbar f'{}^{\vee}(\mathbf{1}_d,\alpha)\big)\,.
\]
We then introduce the multiplicative function $f_{\hbar} = \zeta_{\hbar} \circledast f_{\hbar}^{\vee}$. Let $G_n$ and $G_n^{\vee}$ be the $n$-point functions corresponding to $f_{\hbar}$ and $f_{\hbar}^{\vee}$: they have an $\hbar$-expansion of the form \eqref{giGn} with half-integer $g$ allowed, and Theorem~\ref{thm:R-transform-HigherGenera} applies. As we are in the situation of Lemma~\ref{lem121}, $G_{0,n}$ and $G_{\frac{1}{2},n}$ must coincide with the right-hand sides of \eqref{G12nc9} defined in terms of $f$ and $f'$. We shall focus on them but rewrite Theorem~\ref{thm:R-transform-HigherGenera}  differently. Namely, we decide to replace $G_1^{\vee}(w_i)$ in the operator weight $\vec{\mathsf{O}}^{\vee}(w_i)$ with
\[
G_{{\rm even},1}^{\vee}(w) = \sum_{g \in \mathbb{Z}_{\geq 0}} \hbar^{2g - 1}\,G_{g,1}^{\vee}(w) = \hbar^{-1}G_{0,1}^\vee(w) + O(\hbar)\,.
\]
In order to do this and still keep track of the extra factors in the operator weight $\vec{\mathsf{O}}^{\vee}(w_i)$ corresponding to the summand 
\[
G_{{\rm odd},1}^{\vee}(w) = \sum_{g \in \frac{1}{2} + \mathbb{Z}_{\geq 0}} \hbar^{2g - 1} G_{g,1}^{\vee}(w) = G_{\frac 12,1}^\vee(w) + O(\hbar^2)
\]
in the decomposition $G_1^{\vee} = G_{{\rm even},1}^{\vee}+ G_{{\rm odd},1}^{\vee}$, we introduce below the univalent special black vertices. 
We also split the weights of the black vertices into the sum of two terms according to the decomposition of $G_{n}^\vee = G_{\mathrm{even},n}^\vee + G_{\mathrm{odd},n}^\vee$, where 
\[
\begin{split}
{G}_{\mathrm{even},n}^\vee & = \sum_{g \in \mathbb{Z}_{\geq 0}} \hbar^{2g - 2 + n}\,G^\vee_{g,n} = \hbar^{n - 2}\,G_{0,n}^{\vee} + o(\hbar^{n-1})
\\
G_{\mathrm{odd},n}^\vee & = \sum_{g \in \frac 12 + \mathbb{Z}_{\geq 0}} \hbar^{2g - 2 + n}\,G^\vee_{g,n}  = \hbar^{n - 1}\,G^{\vee}_{\frac{1}{2},n} + o(\hbar^{n-1})\,,
\end{split}
\]
and likewise for $\tilde G_n^\vee$ (note that $G_{\mathrm{odd},n}^\vee = \widetilde{G}_{\mathrm{odd},n}^\vee $ also for $n=2$).

Following the proof of Theorem~\ref{thm:R-transform-HigherGenera} in Section~\ref{proofg}, we see that the functional relation \eqref{Ggnpass} still holds with this modified operator weight and the decomposition of the weights of the black vertices provided we sum over the set $\mathcal{G}'_n$ of bicoloured graphs like in Definition~\ref{def:graphs} but now allowing \emph{multiple} special black vertices.
There are two sources of special black vertices:
\begin{itemize}
	\item Univalent special black vertices. They are coming from the modification of the operators $\vec{\mathsf{O}}^{\vee}(w_i)$ as an interpretation for the contribution of $G^\vee_{\mathrm{odd},1}$. When connecting to the $i$-th white vertex the univalent special black vertex receives the weight   
	\[
	\mathsf{c}^{\vee}(u_i,w_i) = \hbar u_i \varsigma(\hbar u_i w_i\partial_{w_i}) G_{{\rm odd},1}^{\vee}(w_i)\,.
	\]
	\item Special black vertices of valency $\geq 2$. They are coming from the odd part of the decomposition of $G_n^\vee$ and receive the weight
	\[
	\mathsf{c}_{\mathrm{odd}}^{\vee}(u_I,w_I) =\Big( \prod_{i \in I} \hbar u_i \,\varsigma(\hbar u_i w_i \partial_{w_i})\Big) \widetilde{G}^{\vee}_{\mathrm{odd},\# I}(w_{I})\,.
	\]
\end{itemize}
The ordinary black vertices must still have valency $\geq 2$, and they receive the weight 
\[
\mathsf{c}_{\mathrm{even}}^{\vee}(u_I,w_I) =\Big( \prod_{i \in I} \hbar u_i \,\varsigma(\hbar u_i w_i \partial_{w_i})\Big) \widetilde{G}^{\vee}_{\mathrm{even},\# I}(w_{I})
\]
(or the same formula with ${G}^{\vee}_{\mathrm{even},2}$ instead of $\widetilde{G}^{\vee}_{\mathrm{even},2}$ in the corresponding special case). Note that now the automorphisms of the graphs can permute special black vertices but never mix them with the ordinary black vertices.

When extracting the genus-$\frac{1}{2}$ part of the formula, only the leading power in $\hbar$ of each of the weights and only the trees in $\mathcal{G}_n'$ which contain exactly one 
special black vertex will contribute. This is because the series $\varsigma$ where $\hbar$ occurs is even, corresponding to the fact that monotone Hurwitz numbers have integer genus. The outcome is then \eqref{G12nc}, and adapting the proof of Section~\ref{coeffgenus0} gives \eqref{demicoeff}.
\end{proof}
\begin{proof}[Proof of Corollary~\ref{the1demi}]
We specialise \eqref{demicoeff} to $n = 1$. The set $\mathcal{T}'_1(r+ 1)$ contains a single tree, namely the white vertex connected to the special vertex and $r$ other univalent black vertices. We therefore find for $k > 0$:
\begin{equation*}
\begin{split}
f'(\mathbf{1}_k,(1 2 \cdots k)) & = [X^k]\,\,G_{\frac{1}{2},1}(X) 
\\ & = [w^k]\,\, \sum_{r = 0}^{k} \frac{k!}{(k - r)!}\, G_{\frac{1}{2},1}^{\vee}(w) \,\frac{(G_{0,1}^{\vee}(w) - 1)^{r}}{r!} \\
& = [w^k]\,\,G_{\frac{1}{2},1}^{\vee}(w) G_{0,1}^{\vee}(w)^{k}\,,
\end{split}
\end{equation*}
which leads to \eqref{iuniunasf} thanks to the Lagrange inversion formula.
\end{proof}

\subsection{Surfaced free probability}
\label{SecFreehig}
Section~\ref{Sec3} suggests a natural generalisation of the notion of higher-order free cumulants that includes information about higher genera, using the extended convolution instead of the convolution. Just as free probability describes the leading-order behaviour of random hermitian matrix ensembles as $N \rightarrow \infty$, the surfaced version we propose captures corrections of order $N^{-2g}$ beyond the leading order (we will make this precise in Section~\ref{Sec:Randommat}). With Sections~\ref{Sec:PSG}-\ref{Sec:extdemi} and Remark~\ref{evenremark} under our belt, it does not cost anything to allow half-integer $g$. This is also relevant in ensembles of random hermitian matrices with $N$-dependent potentials, see Remark~\ref{Remark29}.

\begin{definition}
A \emph{surfaced probability space} (SPS) is the data $(\mathscr{A},\boldsymbol{\varphi})$ consisting of a unital associative algebra $\mathscr{A}$ over $\mathbb{C}$ and a family of tracial symmetric $n$-linear forms on $\mathscr{A}$
$$
\boldsymbol{\varphi} = \big(\varphi_{g,n}\,\,:\,\,g \in \tfrac{1}{2}\mathbb{Z}_{\geq 0},\,\,n \in \mathbb{Z}_{> 0}\big)
$$
such that $\varphi_{0,1}(1) = 1$ and $\varphi_{g,n}(1,a_2,\ldots,a_n) = 0$ for any $(g,n) \neq (0,1)$ and $a_2,\ldots,a_n \in \mathscr{A}$.
\end{definition}

Given a SPS, we encode the moments into the multiplicative function on $\PSG(\mathscr{A})$ valued in $R = \mathbb{C}$ and given, for $\lambda \vdash d$ of length $\ell(\lambda) = n$, genus $g \in \frac{1}{2}\mathbb{Z}_{\geq 0}$ and $a_1,\ldots,a_d \in \mathscr{A}$, by
\begin{equation}
\label{BBHOPS}
\bbphi(\mathbf{1}_{d},\pi_{\lambda},g)[a_1,\ldots,a_d] = \varphi_{g,n}\Big(\prod_{j = 1}^{\lambda_1} a_{j}\,,\,\prod_{j = 1}^{\lambda_2} a_{\lambda_1 + j}\,,\ldots,\,\prod_{j = 1}^{\lambda_{n}} a_{\lambda_1 + \cdots + \lambda_{n- 1} + j} \Big)\,.
\end{equation}
Then we consider the multiplicative function $\bbphi^{\vee}$ on $\PSG(\mathscr{A})$ by
\begin{equation}
\label{phiveunu}\bbphi^{\vee} = \bbmu \circledast \bbphi \qquad \Longleftrightarrow \qquad \bbphi = \bbzeta \circledast \bbphi^{\vee}\,,
\end{equation}
and we define the genus $g$, order $n$ free cumulants to be the $d$-linear forms $\kappa_{g;k_1,\ldots,k_n}\colon \mathscr{A}^d \rightarrow \mathbb{C}$ given by
\[
\kappa_{g;k_1,\ldots,k_n}(a_1,\ldots,a_d) =\bbphi^{\vee}(\mathbf{1}_{d},\pi_{\lambda(\mathbf{k})},g)[a_1,\ldots,a_{d}]\,,
\]
where $d = k_1 + \cdots + k_n$. For a given $a \in \mathscr{A}$, the generating series of ``surfaced moments'' and ``surfaced free cumulants'' of $a$:
\begin{equation}
\label{GgnGnun}\begin{split}
G_{g,n}(X_1,\ldots,X_n) & = \delta_{g,0}\delta_{n,1} + \sum_{k_1,\ldots,k_n > 0} \varphi_{g,n}(a^{k_1},\ldots,a^{k_n}) \prod_{i = 1}^n X_i^{k_i}\,, \\
G_{g,n}^{\vee}(w_1,\ldots,w_n) & =  \delta_{g,0}\delta_{n,1} + \sum_{k_1,\ldots,k_n > 0} \kappa_{g;k_1,\ldots,k_n}(a,\ldots,a) \prod_{i = 1}^n w_i^{k_i}\,,
\end{split}
\end{equation}
are then related by Theorem~\ref{thm:R-transform-HigherGenera} where half-integer genera are allowed. In particular the $g = 0$ case is covered by Theorem~\ref{thm:R-transform-GenusZero} and the $g = \frac{1}{2}$ case by Theorem~\ref{thmdemimain}.

 \begin{definition}
We define a partial order on $\mathsf{Typ} = \frac{1}{2}\mathbb{Z}_{\geq 0} \times \mathbb{Z}_{> 0}$ by declaring $(g',n') \preceq (g,n)$ when $g' \leq g$ and $g' + n'\leq g + n$. We also define
\[
\overline{\mathsf{Typ}} = \mathsf{Typ} \cup \bigg(\bigcup_{g \in \frac{1}{2}\mathbb{Z}_{\geq 0}} \{(g,\infty)\}\bigg) \cup \{(\infty,\infty)\}\,,
\]
and extend the partial order in the following natural way. For all $g',g \in \frac{1}{2}\mathbb{Z}_{\geq 0}$ such that $g' \leq g$ and any $n' \in \mathbb{Z}_{> 0} \cup \{\infty\}$ we have
\[
(g',n') \preceq (g,\infty) \preceq (\infty,\infty)
\]
\end{definition}
This is not a total order. For instance $(1,1)$ and $(0,3)$ are not comparable, see also Figure~\ref{Fig:Levelsets}.

\begin{figure}[h!]
\begin{center}
\includegraphics[width=\textwidth]{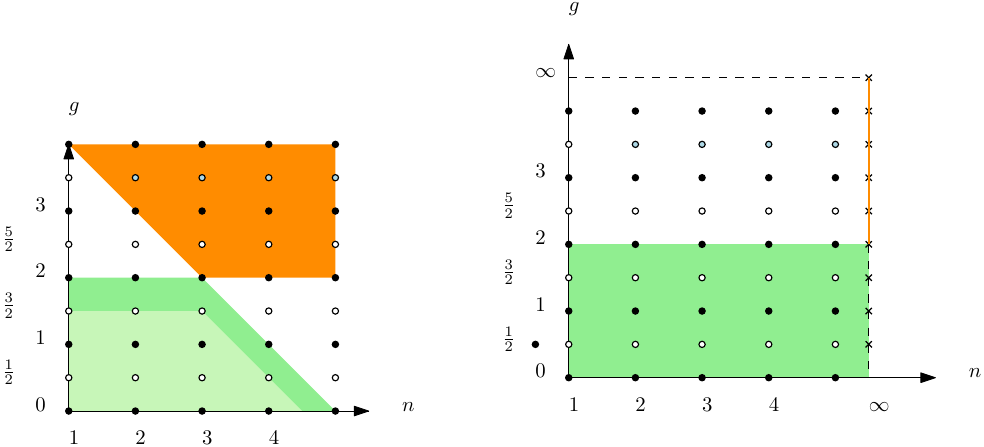}
\caption{\label{Fig:Levelsets} Left panel: in green (resp.~orange), the lower (resp.~upper) level set of $(g,n) = (2,3)$; elements outside these regions are not comparable with $(2,3)$; in light green, the lower level set of $(g,n) = (\frac{3}{2},3)$. Right panel: in green (resp.~orange), the lower (resp.~upper) level set of $(g,n) = (2,\infty)$.}
\end{center}
\end{figure}
The following lemma shows that this order reflects the structure of the relations \eqref{phiveunu} or of Theorem~\ref{thm:R-transform-HigherGenera}.
\begin{lemma}
Let $(g_0,n_0) \in \overline{\mathsf{Typ}}$. The knowledge of $\varphi_{g,n}$ for all $(g,n) \in \mathsf{Typ}$ such that $(g,n) \preceq (g_0,n_0)$ is equivalent to the knowledge of $\kappa_{g;k_1,\ldots,k_n}$ for all $(g,n) \in \mathsf{Typ}$ such that $(g,n) \preceq (g_0,n_0)$ and $k_1,\ldots,k_n > 0$.
\end{lemma}
\begin{proof} This can be in principle extracted by elementary means from the definition \eqref{phiveunu}, but, for illustration, we derive it here for tuples of the form $(a,\ldots,a)$ using Theorem~\ref{thm:R-transform-HigherGenera}. The claim clearly holds for $(g,n) = (0,1)$, $(0,2)$ (see \eqref{0102}) and for $(\frac{1}{2},1)$ (see Corollary~\ref{the1demi}). Now take $a \in \mathscr{A}$ and $(g,n) \in \mathsf{Typ}$ with $2g - 2 + n > 0$: we consider \eqref{Ggnpass} expressing $G_{g,n}$ in terms of $G^{\vee}_{g',n'}$, both defined as in \eqref{GgnGnun}. The summand associated to a graph $\Gamma \in \mathcal{G}_n$ contains --- besides $G_{0,1}^{\vee}$ --- contributions from the hyperedges involving $G_{g_I,\# I}^{\vee}$ for some $g_I \in \frac{1}{2}\mathbb{Z}_{\geq 0}$, and contributions from the $i$-th vertex which give either a $1$ or a product $\prod_{p = 1}^{q_i} G_{g_{i,p},1}^{\vee}$ with $q_i > 0$ and $g_{i,p} \in \frac{1}{2}\mathbb{Z}_{> 0}$. The extraction of the right power of $\hbar$ in \eqref{Ggnpass} shows that 
\[ 
2g - 2 + n \geq \sum_{i = 1}^n \Big(-1 + \sum_{p = 1}^{q_i} 2g_{i,p}\Big)+ \sum_{I \in \mathcal{I}(\Gamma)} 2(g_I - 1 + \# I)\,.
\]
The
inequality comes from the fact that the series $\varsigma$ also generate powers of $\hbar$.
In other words:
\begin{equation}
\label{dind1}
g - 1 + n \geq \sum_{i = 1}^n \sum_{p = 1}^{q_i} g_{i,p} + \sum_{I \in \mathcal{I}(\Gamma)} (g_I - 1 + \# I)\,.
\end{equation}
As already remarked in \eqref{nmoins0}, Euler counting gives $\sum_{I \in \mathcal{I}(\Gamma)} (-1 + \# I) \geq n - 1$ and thus
\begin{equation}
\label{dind2}
g \geq \sum_{i = 1}^n \sum_{p = 1}^{q_i} g_{i,p} + \sum_{I\in \mathcal{I}(\Gamma)} g_I\,.
\end{equation}
As in the right-hand side of \eqref{dind1}-\eqref{dind2} all terms of the sums are nonnegative, we deduce that only $G_{g',n'}^{\vee}$ with $g' + n' \leq g + n$ and $g' \leq g$ are involved in the sum over graphs, that is $(g',n') \preceq (g,n)$ by definition. Noticing that the correction term $\Delta_g^{\vee}$ in \eqref{Ggnpass-n1} only involves $G_{g',1}^{\vee}$ for $g' \leq g$, we deduce that $G_{g,n}$ is expressed as a function of $G_{g',n'}^{\vee}$ with $(g',n') \preceq (g,n)$. 
\end{proof}

The raison d'\^etre of cumulants is to formulate a notion of freeness for elements (or more generally, subalgebras) of $\mathscr{A}$, by the vanishing of mixed free cumulants. We can propose a similar definition here.

\begin{definition}
Let $(g_0,n_0) \in \overline{\mathsf{Typ}}$. A family $(\mathscr{X}_i)_{i \in I}$ of subsets of $\mathscr{A}$ is called \emph{$(g_0,n_0)$-free} if for any $(g,n)  \in \mathsf{Typ}$ such that $(g,n) \preceq (g_0,n_0)$, for any $d > 0$, any map $i : [d] \rightarrow I$, any $(a_1,\ldots,a_d) \in \prod_{p = 1}^d \mathscr{X}_{i(p)}$ and $k_1,\ldots,k_n > 0$ so that $k_1 + \cdots + k_n = d$, we have $$\kappa_{g;k_1,\ldots,k_n}(a_1,\ldots,a_d) = 0\,,$$ whenever there exist $p,q \in [d]$ such that $i(p) \neq i(q)$.
\end{definition}
\begin{remark}
Voiculescu's freeness is $(0,1)$-freeness, the second-order freeness of \cite{MingoSpeicher} is $(0,2)$-freeness, the all-order freeness of \cite{CMSS} is $(0,\infty)$-freeness. Due to Lemma~\ref{lem121} along with Remark~\ref{dualrem} or looking at the definition of the convolution $\ast$ on $\PSG$ (Definition~\ref{def49}), $(\frac{1}{2},1)$-freeness coincides with the notion of infinitesimal freeness of \cite{NF-infinitesimal}. We note that it involves only the free cumulants of type $(g,n) = (0,1)$ and $(\frac{1}{2},1)$. More precisely, in terms of generating series we have from Corollary~\ref{the1demi} the functional relations
\[
G_{0,1}(X) = G_{0,1}^{\vee}(w)\,,\qquad G_{\frac{1}{2},1}(X) = \frac{\dd \ln w}{\dd \ln X}\,G_{\frac{1}{2},1}^{\vee}(w)\,,\qquad X = \frac{w}{G_{0,1}^{\vee}(w)}\,,
\]
which are indeed the ones known for infinitesimal free cumulants. Besides, the $k$-th order infinitesimal freeness of \cite{Fevrierhigher} corresponds to $(0,1)$-freeness using multiplicative functions valued in the ring $R$ of upper triangular Toeplitz matrices of size $k + 1$ (instead of $\mathbb{C}\llbracket\hbar\rrbracket/\hbar^2\mathbb{C}\llbracket\hbar\rrbracket$, which corresponds to $k = 1$).
\end{remark}
An immediate consequence is that, if $a_1 \in \mathscr{X}_1$, $a_2 \in \mathscr{X}_2$ and $(\mathscr{X}_1,\mathscr{X}_2)$ is $(g_0,n_0)$-free, the free cumulants of $a_1 + a_2$ are additive for orders $(g,n) \preceq (g_0,n_0)$, that is:
\begin{equation}
\label{kappaadd}
\kappa_{g;k_1,\ldots,k_n}(a_1 + a_2,\ldots,a_1 + a_2) = \kappa_{g;k_1,\ldots,k_n}(a_1,\ldots,a_1) + \kappa_{g;k_1,\ldots,k_n}(a_2,\ldots,a_2)\,.
\end{equation}
A good notion of freeness for sets should pass to subalgebras. It was shown to be the case for $(0,n)$-freeness in \cite[Section 7.3]{CMSS}. We now extend those arguments to higher genera, only stressing what is new compared to \cite{CMSS}.

\begin{definition}
The type of a surfaced permutation $(\mathcal{A},\alpha,g)$ is $(\mathsf{g},\mathsf{n}) \in \mathsf{Typ}$ where $\mathsf{n}$ is the number of cycles of $\alpha$ and $\mathsf{g} = \sum_{A \in \mathcal{A}} g(A)$. 
\end{definition}

\begin{lemma}
\label{Lhigh} Let $(\mathscr{A},\boldsymbol{\varphi})$ be a SPS and consider the two multiplicative functions $\bbphi,\bbphi^{\vee}$ defined in \eqref{BBHOPS}-\eqref{phiveunu}, and $\mathscr{X}_1,\mathscr{X}_2 \subset \mathscr{A}$, and $(g_0,n_0) \in \overline{\mathsf{Typ}}$. We denote $\mathscr{X}_i^+ = \mathscr{X}_i \cup \{1\}$. The following statements are equivalent:
\begin{itemize}
\item[(i)] $\mathscr{X}_1$ and $\mathscr{X}_2$ are  $(g_0,n_0)$-free.
\item[(ii)] $\mathscr{X}_1^+$ and $\mathscr{X}_2^+$ are $(g_0,n_0)$-free.
\item[(iii)] For any $d > 0$, any $(\mathcal{C},\gamma,k) \in \PSG(d)$ of type $(\mathsf{g},\mathsf{n}) \preceq (g_0,n_0)$, any $a_1,\ldots,a_d \in \mathscr{X}_1^+$ and $b_1,\ldots,b_d \in \mathscr{X}_2^+$, we have
\[
\begin{split} 
&\bbphi(\mathcal{C},\gamma,k)[a_1b_1,\ldots,a_db_d] = \\
&
\quad \sum_{(\mathcal{A},\alpha,g) \odot (\mathcal{B},\beta,h) = (\mathcal{C},\gamma,k)} \bbphi(\mathcal{A},\alpha,g)[a_1,\ldots,a_d]\,\bbphi^{\vee}(\mathcal{B},\beta,h)[b_1,\ldots,b_d]\,.
\end{split}
\] 
\item[(iv)] For any $d > 0$, any $(\mathcal{C},\gamma,k) \in \PSG(d)$ of type $(\mathsf{g},\mathsf{n}) \preceq (g_0,n_0)$, any $a_1,\ldots,a_d \in \mathscr{X}_1^+$ and $b_1,\ldots,b_d \in \mathscr{X}_2^+$, we have
\[
\begin{split}
	& \bbphi^{\vee}(\mathcal{C},\gamma,k)[a_1b_1,\ldots,a_db_d] =
	\\ & \quad  \sum_{(\mathcal{A},\alpha,g) \odot (\mathcal{B},\beta,h) = (\mathcal{C},\gamma,k)} \bbphi^{\vee}(\mathcal{A},\alpha,g)[a_1,\ldots,a_d]\,\bbphi^{\vee}(\mathcal{B},\beta,h)[b_1,\ldots,b_d]\,.
\end{split}
\]
\end{itemize}
\end{lemma}
\begin{proposition}
\label{SettoAlg}
Let $(\mathscr{X}_i)_{i \in I}$ be a family of subsets of $\mathscr{A}$, and $\mathscr{A}_i$ the subalgebra generated by $\mathscr{X}_i$. Let $(g_0,n_0) \in \overline{\mathsf{Typ}}$. The $(g_0,n_0)$-freeness of $(\mathscr{X}_i)_{i \in I}$ is equivalent to the $(g_0,n_0)$-freeness of $(\mathscr{A}_i)_{i \in I}$. 
\end{proposition}
\begin{proof}[Proof of Lemma~\ref{Lhigh}] This is the higher-genus generalisation of 
	\cite[Theorem 7.9]{CMSS}. Here we only explain (i) $\Rightarrow$ (iii). The converse direction is then proved as in \cite{CMSS}. The implication (iii) $\Rightarrow$ (iv) comes from extended convolution with the M\"obius function and (iv) $\Rightarrow$ (iii) from extended convolution with the zeta function. The equivalence between (i) and (ii) comes from a direct adaptation of \cite[Proposition~7.8]{CMSS}.

\medskip

Let $(\mathcal{C},\gamma,k) \in \PSG(d)$ of type $(\mathsf{g},\mathsf{n})$. We take a second copy $[\bar{d}]$ of the set $[d]$ and interleave their elements
\[
 [d,\bar{d}]\coloneqq \{1,\,\bar{1},\,2,\, \bar{2},\, 3,\,\bar{3},\dots,\, d,\, \bar{d}\} \cong [2d]\,.
\] 
We call $\psi\colon [d] \rightarrow [\bar{d}]$ the canonical identification. For a set $C=\{i_1,\dots,i_{\ell}\}\subseteq [d]$, we denote by $\bar{C} = \{\bar{i_1},\dots,\bar{i_{\ell}}\} \subseteq [\bar{d}]$. We define the surfaced permutation $(\hat{\mathcal{C}}, \hat{\gamma}, \hat{k}) \in\mathbb{PS}(2d)$ such that the blocks of $\hat{\mathcal{C}}$ are of the form $\hat{C} \coloneqq C\cup \bar{C}$ where $C \in \mathcal{C}$, the permutation $\hat{\gamma}$ is characterised by $\hat{\gamma}|_{[d]} = \psi$ and $\hat{\gamma}_{|[\bar{d}]} = \gamma \circ \psi^{-1}$ and the genus function is $\hat{k}(C\cup \bar{C}) = k(C)$. We have
\begin{equation*}
\label{phiphivee}
\begin{split}
\bbphi(\mathcal{C},\gamma,k)[a_1b_1,\ldots,a_db_d] & = \bbphi(\hat{\mathcal{C}},\hat{\gamma},\hat{k})[a_1,b_1,\ldots,a_d,b_d] \\
& = \sum_{\substack{(\mathbf{0}_{\alpha},\alpha,0) \odot (\mathcal{B},\beta,h) \\ = (\hat{\mathcal{C}},\hat{\gamma},\hat{k})}} \bbphi^{\vee}(\mathcal{B},\beta,h)[a_1,b_1,\ldots,a_d,b_d]\,. 
\end{split}
\end{equation*}

Assume that $(\mathscr{X}_1,\mathscr{X}_2)$ is $(g_0,n_0)$-free. The vanishing of the mixed surfaced free cumulants up to order $(g_0,n_0)$ means that the only terms remaining in the right-hand side of the previous display come from surfaced permutations $(\mathcal{B},\beta,h)$ where blocks $B \in \mathcal{B}$ are included either in $[d]$ or in $[\bar{d}]$. Since $\mathbf{0}_{\beta} \leq \mathcal{B}$, the permutation $\beta$ leaves $[d]$ and $[\bar{d}]$ stable and we introduce:
\[
\beta_1 = \beta_{|[d]}\,,\qquad \beta_2 = \psi^{-1} \circ \beta_{|[\bar{d}]} \circ \psi\,,\qquad \text{both}\,\,\text{in}\,\,S(d)\,.
\]
By multiplicativity of $\bbphi^{\vee}$, we obtain:
\begin{equation}
\label{phiphivee2}
\begin{split}
& \bbphi(\mathcal{C},\gamma,k)[a_1b_1,\ldots,a_db_d] 
\\
&
\quad = \bbphi(\hat{\mathcal{C}},\hat{\gamma},\hat{k})[a_1,b_1,\ldots,a_d,b_d] \\
& \quad = \sum_{\substack{(\mathbf{0}_{\alpha},\alpha,0) \odot (\mathcal{B},\beta,h) \\ = (\hat{\mathcal{C}},\hat{\gamma},\hat{k})}} \bbphi^{\vee}(\mathcal{B}_1,\beta_1,h_1)[a_1,\ldots,a_d]\,\bbphi^{\vee}(\mathcal{B}_2,\beta_2,h_2)[b_1,\ldots,b_d]\,. 
\end{split}
\end{equation}

The condition $\alpha \circ \beta = \hat{\gamma}$ implies that $\alpha$ sends $[d]$ to $[\overline{d}]$ and vice versa. Introducing
\[
\alpha_1 = \psi^{-1} \circ \alpha_{|[d]}\,,\qquad \alpha_2 =  \alpha_{|[\bar{d}]} \circ \psi\,,\qquad  \text{both}\,\,\text{in}\,\,S(d)\,,
\]
we get $\alpha_1 \circ \beta_1 = \mathrm{id}$ and $\alpha_2 \circ \beta_2 = \gamma$. If we let $\tilde{\alpha} = \alpha_2 \circ \beta_1^{-1} = \alpha_2 \circ \alpha_1 = \alpha^2_{|[d]}$, for any block $C \in \mathcal{C}$ the permutations $\tilde{\alpha}_{|C}$ and $\alpha_{|\hat{C}}$ have the same number of cycles. Therefore
\begin{equation}
\label{alphalpha}
|\tilde{\alpha}_{|C}| = |\alpha_{|\hat{C}}| - \# C.
\end{equation}
We denote $(\mathcal{B}_1,\beta_1,h_1),(\mathcal{B}_2,\beta_2,h_2) \in \PSG(d)$ which are obtained by restricting $(\mathcal{B},\beta,h)$ to $[d]$ and $[\bar{d}]$, using again the canonical identification $[d] \cong [\bar{d}]$.  We claim that
\begin{equation}
\label{lap1} (\mathbf{0}_{\alpha},\alpha,0) \odot (\mathcal{B},\beta,h) = (\hat{\mathcal{C}},\hat{\gamma},\hat{k})
\end{equation}
implies
\begin{equation}
\label{lap2}(\mathbf{0}_{\tilde{\alpha}},\tilde{\alpha},0) \odot (\mathcal{B}_1,\beta_1,h_1) \odot (\mathcal{B}_2,\beta_2,h_2) = (\mathcal{C},\gamma,k)\,.
\end{equation}
This amounts to checking that $\mathbf{0}_{\tilde{\alpha}} \vee \mathcal{B}_1 \vee \mathcal{B}_2 = \mathcal{C}$ and matching the genus functions. The former is already in \cite{CMSS} and we focus on the latter which is new to our setting. By Remark~\ref{Blockad}, we have to check block-additivity of colengths. Given a block $C \in \mathcal{C}$, 
\begin{equation*}
\begin{split}
&  |(\mathbf{0}_{\tilde{\alpha}}{}_{|C},\tilde{\alpha}_{|C})| + |(\mathcal{B}_1{}_{|C},\beta_1{}_{|C},h_1{}_{|C})|  + |(\mathcal{B}_2{}_{|C},\beta_2{}_{|C},h_2{}_{|C})| \\
& \quad = |\tilde{\alpha}_{|C}| + |(\mathcal{B}_{|C},\beta_{|C},h_{|C})| + |(\mathcal{B}_{|\bar{C}},\beta_{|\bar{C}},h_{|\bar{C}})| \\
& \quad = |\alpha_{|\hat{C}}| - \# C + |(\mathcal{B}_{|\hat{C}},\beta_{|\hat{C}},h_{|\hat{C}})| \\
& \quad = 2\hat{k}(\hat{C}) + |(\mathbf{1}_{\hat{C}},\hat{\gamma}_{|\hat{C}})| - \# C
\\ \notag
& \quad = 2k(C) + |(\mathbf{1}_{C},\gamma_{|C})|.
\end{split}
\end{equation*}
We have used \eqref{spcol} in the second line, \eqref{alphalpha} in the third line, the definition \eqref{eq:prod:PSG} of the genus function for $\odot$ in the fourth line, and $|(\mathbf{1}_{\hat{C}},\hat{\gamma}_{|\hat{C}})| = |(\mathbf{1}_{C},\gamma_{|C})| + \# C$ in the last line. This justifies the claim, and a closer look at this argument shows that \eqref{lap1} and \eqref{lap2} are in fact equivalent.

Since the first factor in \eqref{lap2} is exactly the kind of surfaced permutation in the support of the zeta function, we can transform \eqref{phiphivee2} into
\begin{equation*}
\begin{split}
&  \bbphi(\mathcal{C},\gamma,k)[a_1b_1,\ldots,a_db_d] \\
& \quad = \!\!\!\sum_{\substack{(\tilde{\mathcal{A}},\tilde{\alpha},\tilde{g}) \odot (\mathcal{B}_1,\beta_1,h_1) \\  \odot (\mathcal{B}_2,\beta_2,h_2) = (\mathcal{C},\gamma,k)}} \!\!\! \bbzeta(\tilde{\mathcal{A}},\tilde{\alpha},\tilde{g}) \,\bbphi^{\vee}(\mathcal{B}_1,\beta_1,h_1)[a_1,\ldots,a_d] \,
\\ & \qquad  \cdot
 \bbphi^{\vee}(\mathcal{B}_2,\beta_2,h_2)[b_1,\ldots,b_d] \\
& \quad = \!\!\! \sum_{\substack{(\tilde{\mathcal{B}}_1,\tilde{\beta}_1,\tilde{h}_1) \odot (\mathcal{B}_2,\beta_2,h_2) \\ = (\mathcal{C},\gamma,k)}} \!\!\! \bbphi(\tilde{\mathcal{B}}_1,\tilde{\beta}_1,\tilde{h}_1)[a_1,\ldots,a_d]\,\bbphi^{\vee}(\mathcal{B}_2,\beta_2,h_2)[b_1,\ldots,b_d]\,,
\end{split}
\end{equation*} 
where we have recognised convolution by the zeta function to get the second line.
\end{proof}

\begin{proof}[Proof of Proposition~\ref{SettoAlg}] By multilinearity, the linear spans of two free sets are free. The only thing that deserves a check is that freeness of $\mathscr{X}_1$ and $\mathscr{X}_2$ implies freeness of $\mathscr{X}_1 \cup \mathscr{X}_1\mathscr{X}_1$ and $\mathscr{X}_2$. Given Lemma~\ref{Lhigh}, the proof is identical to \cite[Theorem 7.12]{CMSS}.
\end{proof}

\subsection{Application in random matrix theory}
\label{Sec:Randommat}
The formalism of surfaced free cumulants and freeness up to order $(g_0,n_0)$ can be directly applied in random matrix theory, generalising the known cases of order $(0,1)$ \cite{Voi91} and $(0,2)$ \cite{MMS07}.

\begin{definition}
If $A$ is a matrix of size $N$ and $\lambda \vdash d$ of length $n$, with $d \leq N$, we denote
\begin{equation}
p_{\lambda}(A) = \prod_{i = 1}^{n} {\rm Tr}(A^{\lambda_i})\,,\qquad \mathcal{P}_{\lambda}(A) = \prod_{c = 1}^{d} A_{c,\pi_{\lambda}(c)}\,.
\end{equation}
\end{definition}

We recall the following relation between expectation values of $p_{\lambda}$ and $\mathcal{P}_{\lambda}$, which comes from Weingarten calculus. It has exactly the same form as the transformation studied in Theorem~\ref{thm:hbarstarproduct}, and this is in fact the motivation for the abstract study of these transformations that we carried out in the previous sections.

\begin{theorem} (\cite[Theorem 4.4]{CMSS}, \cite[Theorem 8.8]{BG-F18})
\label{tunome}Let $A$ be a random hermitian matrix of size $N$, whose law is invariant under unitary conjugation. Then for any $\lambda \vdash d$ with $d \leq N$:
\begin{equation*}
\begin{split}
\mathbb{E}[p_{\lambda}(A)] = \mathfrak{z}_{\lambda} \sum_{\nu \vdash d} N^{d}\,H^{<}(\lambda,\nu)\big|_{\hbar = 1/N}\,\mathbb{E}[\mathcal{P}_{\nu}(A)]\,, \\
\mathbb{E}[\mathcal{P}_{\lambda}(A)] = \mathfrak{z}_{\lambda} \sum_{\nu \vdash d} N^{-d} H^{\leq}(\lambda,\nu)\big|_{\hbar = 1/N}\,\mathbb{E}[p_{\nu}(A)]\,.
\end{split}
\end{equation*}
\end{theorem}

Recall that if $Y_1,\ldots,Y_n$ are random variables, their cumulant expectation value is defined as
$$
\mathbb{E}^{\circ}[Y_1,\ldots,Y_n] = \partial_{s_1} \cdots \partial_{s_n} \ln \mathbb{E}[e^{s_1 Y_1 + \cdots + s_n Y_n}]\big|_{s_1 = \cdots = s_n = 0}\,.
$$
A similar notion has been used in the Fock space (Section~\ref{npointSec}) where instead of $\mathbb{E}$ we have the linear form $\langle\cdot \rangle$. In particular, the inclusion-exclusion relation holds
\begin{equation}
\begin{split}
\mathbb{E}\bigg[\prod_{i = 1}^{n} Y_i\bigg] & = \sum_{\mathcal{I} \in P(n)} \prod_{I \in \mathcal{I}} \mathbb{E}^{\circ}\big[(Y_i)_{i \in I}\big]\,, \\
\mathbb{E}^{\circ}[Y_1,\ldots,Y_n] & = \sum_{\mathcal{I} \in P(n)} (-1)^{\# \mathcal{I} - 1} (\# \mathcal{I} - 1)!  \prod_{I \in \mathcal{I}} \mathbb{E}\bigg[\prod_{i \in I} Y_i\bigg]\,. \\
\end{split}
\end{equation}

\begin{definition}
Let $(A_N)_N$ be a sequence of random hermitian matrices of size $N$. We say that it admits a limit distribution up to order $(g_0,n_0)$ if there exists $F_{g;k_1,\ldots,k_n}$ indexed by $(g,n) \preceq (g_0,n_0)$ and $k_1,\ldots,k_n > 0$, independent of $N$, such that for any $n \in [\lfloor g_0 \rfloor  + n_0]$ and any $k_1,\ldots,k_n > 0$, we have when $N \rightarrow \infty$
\[
\begin{split}
& \mathbb{E}^{\circ}\big[{\rm Tr}(A_N^{k_1}),\ldots,{\rm Tr}(A_N^{k_n})\big] =  
\\
&
\quad  \!\!\!\!\sum_{\substack{g \in \frac{1}{2}\mathbb{Z}_{\geq 0} \\  g \leq g_0 + \min(0,n_0 - n)}} \!\!\!\! N^{2 - 2g - n}\,F_{g;k_1,\ldots,k_n} + o(N^{2 - 2g_0 - n_0 + |n_0 - n|})\,,	
\end{split}
\]
where $\mathbb{E}^{\circ}$ denotes the cumulant expectation value. In this expression, the order of the $o(\cdots)$  is adjusted to be the next subleading term compared to the sum. When $g_0 = \infty$, we ask for the existence of such an asymptotic expansion to an arbitrary order $o(N^{-K})$ for all $n \leq n_0$, and in that case we use the notation
\begin{equation}
\label{cumulantofAN}
\mathbb{E}^{\circ}\big[{\rm Tr}(A_N^{k_1}),\ldots,{\rm Tr}(A_N^{k_n})\big] = \sum_{g \in \frac{1}{2}\mathbb{Z}_{\geq 0}} N^{2 - 2g - n}\,F_{g;k_1,\ldots,k_n}^{{\rm A}} + o(N^{-\infty})\,.
\end{equation}
\end{definition}  
By Theorem~\ref{tunome}, if $(A_N)_N$ has a limit distribution up to order $(g_0,n_0)$, then for any partition $\lambda \vdash d$ of length $n \leq \lfloor g_0 \rfloor + n_0$ we have when $N \rightarrow \infty$:
 \begin{equation}
\label{expEP} \mathbb{E}^{\circ}[\mathcal{P}_{\lambda}(A_N)] =  \!\!\!\!\! \sum_{\substack{g \in \frac{1}{2}\mathbb{Z}_{\geq 0} \\ g \leq g_0 + \min(0,n_0 - n)}} \!\!\!\!\! N^{2 - 2g - n - d}\,\kappa_{g;\lambda_1,\ldots,\lambda_n}^{{\rm A}} + o(N^{2 - 2g_0 - n_0 + |n_0 - n| - d})\,.
\end{equation}
We obtain the structure of a SPS on the algebra $\mathscr{A} = \mathbb{C}[a]$  by taking as moments:
\[
\forall (g,n) \preceq (g_0,n_0)\,,\qquad \varphi_{g,n}(a^{k_1},\ldots,a^{k_n}) = F^{{\rm A}}_{g;k_1,\ldots,k_n}\,.
\]
Combining Theorem~\ref{tunome} and the expansion \eqref{expEP} with Theorem~\ref{thm:hbarstarproduct} indicates that $\kappa^{{\rm A}}_{g;k_1,\ldots,k_n}$ are the free cumulants at order $(g,n) \preceq (g_0,n_0)$.

\begin{theorem}\label{ABfreethm}
Let $(A_N)_N$ and $(B_N)_N$ be two sequences of ensembles of random matrices of size $N$, at least one of them being unitarily invariant, and such that for each $N$, $A_N$ is independent of $B_N$. Assume that both ensembles have a limit distribution up to order $(g_0,n_0)$ (possibly $\infty$), and consider the algebra $\mathscr{A} = \mathbb{C}\langle a,b \rangle$ of non-commutative polynomials in two letters. Then, for any $Q \in \mathscr{A}$, $Q(A_N,B_N)$ admits a limit distribution up to order $(g_0,n_0)$, so $\mathscr{A}$ can be upgraded to a SPS. Besides, the subalgebras $\mathbb{C}[a]$ and $\mathbb{C}[b]$ are  $(g_0,n_0)$-free. 
\end{theorem}
 
 \begin{proof}
Let $(A_N)_N$ and $(B_N)_N$ as in the theorem. For any $k,k' > 0$, $(A_N^{k})_N$ and $(B_N^{k'})_N$ clearly have a limit distribution up to order $(g_0,n_0)$. Examining the finite-$N$ formula in \cite[Theorem~4.4, (2)]{CMSS}, the products $(A_N^{k} B_N^{k'})_N$ also have a limit distribution up to order $(g_0,n_0)$. Take $\sigma \in S(d)$,  $N \geq d$ and a map $T \colon [d] \rightarrow \{A_N,B_N\}$. Due to the independence of $A_N$ and $B_N$, and the unitary invariance of the law of one of the matrices (say $A_N$), we have
\begin{equation}
\label{nunfun}\begin{split} 
& \mathbb{E}\Bigg[\prod_{c = 1}^d (T(c))_{c,\sigma(c)}\Bigg] 
\\ & \quad = \mathbb{E}\Bigg[\prod_{c \in T^{-1}(A_N)} (A_N)_{c,\sigma(c)}\Bigg]\,\mathbb{E}\Bigg[\prod_{c \in T^{-1}(B_N)} (B_N)_{c,\sigma(c)}\Bigg] \\
& \quad = \int_{U(N)} \!\!\! \mathrm{d}U\ \mathbb{E}\Bigg[\prod_{c \in T^{-1}(A_N)} (UA_NU^{-1})_{c,\sigma(c)}\Bigg]\,\mathbb{E}\Bigg[\prod_{c \in T^{-1}(B_N)} (B_N)_{c,\sigma(c)}\Bigg] \\
& \quad = \sum_{\substack{i_c,j_c \in [N] \\ c \in T^{-1}(A_N)}} \bigg(\int_{U(N)} \!\!\! \mathrm{d}U  \prod_{c \in T^{-1}(A_N)} U_{c,i_c} U^{-1}_{j_c,\sigma(c)} \bigg) \mathbb{E}\Bigg[\prod_{c \in T^{-1}(A_N)} (A_N)_{i_c,j_c}\Bigg] \,
\\
& \qquad \cdot \mathbb{E}\Bigg[\prod_{c \in T^{-1}(B_N)} (B_N)_{c,\sigma(c)}\Bigg]\,.
\end{split}  
\end{equation}
By Weingarten calculus \cite{CollinsWeingarten} the integral over $U(N)$ vanishes unless there exist two permutations $\alpha,\beta \in S(T^{-1}(A_N))$ such that $c = \sigma(\alpha(c))$ and $i_c = j_{\beta(c)}$ for all $c \in T^{-1}(A_N)$. These conditions force the map $T$ to be constant on the cycles of $\sigma$. As a result, the connected observables
\[
\mathbb{E}^{\circ}\bigg[\bigg(\prod_{c \in C} (T(c))_{c,\sigma(c)}\bigg)_{C \in \mathbf{0}_{\sigma}}\bigg]
\] 
vanish unless $T$ is constant on $[d]$. Since the surfaced free cumulants evaluated on $(T(c))_{c \in [d]}$ are extracted from the asymptotic expansion of these connected observables as $N \rightarrow \infty$ (recall \eqref{expEP}), all mixed surfaced free cumulants between $\mathbb{C}[a]$ and $\mathbb{C}[b]$ vanish up to order $(g_0,n_0)$ (which from the assumption is the order up to which the asymptotic expansions exist), i.e.~these two algebras are $(g_0,n_0)$-free in the surfaced probability space $\mathbb{C}\langle a,b \rangle$.
\end{proof}
  
\begin{remark} \label{Remark29} The combinatorics underlying infinitesimal free cumulants is the truncation keeping the leading (genus $0$) and the first subleading (genus $\frac{1}{2}$) order of the master relation involving monotone Hurwitz numbers, whose origin can be traced back to the Weingarten calculus for the unitary group. Typically, for topological expansions in unitarily invariant random hermitian matrices the genus-$\frac{1}{2}$ order (corresponding to a term of order $N^{-1}$ compared to the leading term) vanishes. An example of situation where it does not vanish is the 1-hermitian matrix model with an $N$-dependent potential of the form $V_0 + N^{-1}V_1$. Although the non-vanishing of the genus-$\frac{1}{2}$ order is typical in the topological expansion for orthogonally invariant random symmetric matrices, the observation of Mingo \cite{Mingoinf} that infinitesimal freeness cannot describe the relations between moments and cumulants in such models is not a surprise, as they should rather be governed by the Weingarten calculus for the orthogonal group.
\end{remark}

\begin{remark} \label{RemMMPart} The topological partition function associated with \eqref{cumulantofAN} is in fact the asymptotic series of the formal series in the variables $X_1,X_2,\ldots$ as $N \rightarrow \infty$ 
\begin{equation*}
\begin{split}
\mathbb{E}\bigg[\prod_{i} \frac{1}{\det(1 - X_iA_N)}\bigg] & = \mathbb{E}\bigg[\exp\Big(\sum_{k \geq 1} \frac{{\rm Tr}(A_N^k) p_k}{k}\Big)\bigg] \\
& = \exp\bigg(\sum_{n \geq 1} \sum_{k_1,\ldots,k_n \geq 1} \mathbb{E}^{\circ}\big[{\rm Tr}(A_N^{k_1}),\ldots,{\rm Tr}(A_N^{k_n})\big]\,\frac{p_{k_1} \cdots p_{k_n}}{n!\,k_1 \cdots k_n}\bigg) \\
& = Z\big|_{\hbar = 1/N} \cdot \big(1 + O(N^{-\infty})\big)\,,
\end{split}
\end{equation*}
where $p_k = \sum_i X_i^k$ is the $k$-th power sum. The first equality is Cauchy's identity, the second one is the definition of the connected expectation value, and the last one comes from comparing \eqref{cumulantofAN} and Section~\ref{npointSec}.
\end{remark}

\subsection{Remarks on analyticity}
\label{Sec:analrecm}
Several important classes of ensembles of random hermitian matrices exhibit limit distributions to all orders \cite{APS01,BG11,BGK}, and it is typical that only integer $g$ appear. All the cases handled in \textit{op.~cit.}~satisfy the assumptions of Proposition~\ref{pr313}: for all $g,n$, the generating series $G_{g,n}$ admit analytic continuations as meromorphic functions on $\Sigma^n$, where $\Sigma$ is a Riemann surface equipped with two meromorphic functions $\mathsf{x},\mathsf{w}$ determined by the equation $\mathsf{w} = W_{0,1}(\mathsf{x}) = \frac{1}{\mathsf{x}} G_{0,1}(\frac{1}{\mathsf{x}})$ (the so-called \emph{spectral curve}). These assumptions are also satisfied in all models governed by the Eynard--Orantin topological recursion \cite{BEO13} --- where it is more common to use the differential forms $\omega_{g,n}$. Proposition~\ref{pr313} then guarantees that the $G_{g,n}^{\vee}$ also admit analytic continuation on $\Sigma^n$.

This has a practical consequence of interest. In the situation of Theorem~\ref{ABfreethm}, assume that $(A_N)_N$ and $(B_N)_N$ admit respective spectral curves $(\Sigma^{{\rm A}},\mathsf{x}^{{\rm A}},\mathsf{w}^{{\rm A}})$ and $(\Sigma^{{\rm B}},\mathsf{x}^{{\rm B}},\mathsf{w}^{{\rm B}})$, and that for all $g,n$, $G_{g,n}^{*}$ admits an analytic continuation as a meromorphic function on $(\Sigma^*)^{n}$ for each $* \in \{{\rm A},{\rm B}\}$.  Then, we define $\Sigma^{{\rm A}+{\rm B}}$ as the subset of $\Sigma^{{\rm A}} \times \Sigma^{{\rm B}}$ where $\mathsf{w}^{{\rm A}} = \mathsf{w}^{{\rm B}}$, which we equip with the two meromorphic functions
$$
\mathsf{w}^{{\rm A} + {\rm B}} = \mathsf{w}^{{\rm A}} = \mathsf{w}^{{\rm B}}\,,\qquad \mathsf{x}^{{\rm A} + {\rm B}} = \mathsf{x}^{{\rm A}} + \mathsf{x}^{{\rm B}} - \frac{1}{\mathsf{w}^{{\rm A} + {\rm B}}}\,.
$$
By the R-transform machinery (cf.~\eqref{xweq}) and additivity of the (first-order) free cumulants in the situation of Theorem~\ref{ABfreethm}, this is a spectral curve for $(A_N + B_N)_N$. Furthermore, the additivity of all-order free cumulants (cf.~\eqref{kappaadd}) yields for $2g - 2 + n \geq 0$
$$
G_{g,n}^{{\rm A} + {\rm B},\vee} = G_{g,n}^{{\rm A},\vee} + G_{g,n}^{{\rm B},\vee}\,.
$$
Our assumptions and Proposition~\ref{pr313} imply that the right-hand side admits an analytic continuation as a meromorphic function on $(\Sigma^{{\rm A}+{\rm B}})^n$. Using again Proposition~\ref{pr313}, we deduce that $G_{g,n}^{{\rm A}+{\rm B}}$ admits an analytic continuation as a meromorphic function on $(\Sigma^{{\rm A} + {\rm B}})^n$.

\section{Example: GUE + deterministic}
\label{SecGUEDET}

We illustrate the previous formulae with computations for the sum of a GUE matrix and an independent deterministic matrix. This ensemble was first considered by Pastur, who showed that the $(0,1)$-moments exist (law of large numbers) and computed them \cite{Pastur}. It attracted interest since the discovery of the Baik--Ben Arous--P\'ech\'e (BBP) phase transition describing a separation of the maximum eigenvalue for deterministic parts of small rank \cite{BBPphaset}. This was revisited using free probability \cite{Capitaine}, see also the review \cite{DonatiCapitaine}. Our results permit the systematic study of the topological expansion in these ensembles. We will restrict ourselves to deterministic matrices with two distinct eigenvalues, where the computations remain relatively explicit, and derive explicit formulae for the generating series of $(0,2)$, $(0,3)$, and $(1,1)$-moments. In general $(g,n)$-moments for $2g - 2 + n > 0$ do not come from a positive measure, but from finite-order distributions.

\subsection{Setting}
 \label{GUESetting}
Fix parameters $s,L \in \mathbb{R}_{> 0}$ and $t \in (0,1)$. Let $A_N$ be a random hermitian matrix of size $N$ with law of density
$$
\dd\mathbb{P}(A_N) =  c_N \exp\Big(-\frac{N}{2s^2} {\rm Tr}(A_N^2)\Big) \dd A_N\,,
$$
with constant $c_N$  making $\mathbb{P}$ a probability measure, and $$B_N = {\rm diag}(L,\ldots,L,0,\ldots,0)$$ a deterministic matrix with $tN$ non-zero eigenvalues. To simplify the statements we assume that $t$ is a fixed rational number not depending on $N$, and the denominator of $t$ divides $N$. As all expansions we consider will be uniform for $t \in (0,1)$, this restriction is not essential. It is obvious for $(B_N)_N$, and a classical consequence of Wick's theorem for $(A_N)_N$, that these ensembles admit a limit distribution to all orders with integer genera only. According to Theorem~\ref{ABfreethm}, $(A_N)_N$ and $(B_N)_N$ are asymptotically free to all orders. The analytic considerations of Section~\ref{Sec:analrecm} will also appear relevant.

We write the asymptotic expansion of the connected moments as $N \rightarrow \infty$
\begin{equation}
\label{EANBN} \begin{split}
\mathbb{E}^{\circ}\big[{\rm Tr}(A_N^{k_1}),\ldots,{\rm Tr}(A_N^{k_n})\big] & = \sum_{g \geq 0} N^{2 - 2g - n}\,F_{g;k_1,\ldots,k_n}^{{\rm A}} + o(N^{-\infty})\,, \\
\mathbb{E}^{\circ}\big[{\rm Tr}(B_N^{k_1}),\ldots,{\rm Tr}(B_N^{k_n})\big] & = \sum_{g \geq 0} N^{2 - 2g - n}\,F_{g;k_1,\ldots,k_n}^{{\rm B}} + o(N^{-\infty})\,.
\end{split}
\end{equation} 
By Theorem~\ref{ABfreethm}, $(A_N + B_N)_N$ also admits a limit distribution to all orders as $N \rightarrow \infty$
\begin{equation}
\label{EANplusBN}
\begin{split}
& \mathbb{E}^{\circ}[{\rm Tr}((A_N + B_N)^{k_1}),\ldots,{\rm Tr}((A_N + B_N)^{k_n})] 
\\
&
\quad = \sum_{g \geq 0}  N^{2 - 2g - n}\,F_{g;k_1,\ldots,k_n}^{{\rm A} + {\rm B}} + o(N^{-\infty})\,.
\end{split}
\end{equation}
We will show how to compute the first few terms in \eqref{EANplusBN}, going beyond the known $(0,1)$ case.
 
 We remind the reader the notation $|\lambda|$ stands for $\sum_{i=1}^{\ell(\lambda)}\lambda_i$, where here $\ell(\lambda)=n$.
The associated free cumulants $\kappa_{g;k_1,\ldots,k_n}^{*}$ with $* \in \{{\rm A},{\rm B},{\rm A} + {\rm B}\}$ are instrumental in this computation. For the GUE, they provide the $N \rightarrow \infty$ asymptotic expansion of the connected moments of cyclic products of entries of the matrix as in \eqref{expEP}:
$$
\mathbb{E}^{\circ}[\mathcal{P}_{\lambda}(A_N)] = \sum_{g \geq 0} N^{2 - 2g - n - |\lambda|}\,\kappa_{g;\lambda_1,\ldots,\lambda_n}^{{\rm A}} + o(N^{-\infty})\,.
$$
For the deterministic matrix, this does not apply because the law of $B_N$ is not unitarily invariant, but it applies to $B_N' = V_N B_N V_N^{-1}$ where $V_N$ is a Haar-distributed random unitary matrix independent of $A_N$. Since $B_N'$ has the same moments as $B_N$, it must have the same free cumulants and we get
\begin{equation}
\label{EkBN}
\mathbb{E}^{\circ}[\mathcal{P}_{\lambda}(B'_N)]  = \sum_{g \geq 0} N^{2 - 2g - n -|\lambda|}\,\kappa_{g;\lambda_1,\ldots,\lambda_n}^{{\rm B}} + o(N^{-\infty})\,.
\end{equation}
A similar statement can be made for $A_N + B_N'$.

Let us recall the definition of the various generating series in this context.
\begin{equation}
\begin{split}
G_{g,n}(w_1,\ldots,w_n) & = \delta_{g,0}\delta_{n,1} + \sum_{k_1,\ldots,k_n \geq 1} F_{g;k_1,\ldots,k_n} \prod_{i = 1}^n w_i^{k_i}\,, \\
W_{g,n}(x_1,\ldots,x_n) & = \frac{\delta_{g,0}\delta_{n,1}}{x_1} + \sum_{k_1,\ldots,k_n \geq 1} \frac{F_{g;k_1,\ldots,k_n}}{\prod_{i = 1}^{n} x_i^{k_i + 1}}\,, \\
G^{\vee}_{g,n}(w_1,\ldots,w_n) & = \delta_{g,0}\delta_{n,1} + \sum_{k_1,\ldots,k_n \geq 1} \kappa_{g;k_1,\ldots,k_n} \prod_{i = 1}^{n} w_i^{k_i}\,, \\
X_{g,n}(w_1,\ldots,w_n) & = \frac{\delta_{g,0}\delta_{n,1}}{w_1} + \sum_{k_1,\ldots,k_n \geq 1} \kappa_{g;k_1,\ldots,k_n} \prod_{i = 1}^{n} w_i^{k_i - 1}\,.
\end{split}
\end{equation}
For completeness and comparison with previous formulae, let us also mention
$$
X(w) = \frac{1}{X_{0,1}(w)} = \frac{w}{G_{0,1}^{\vee}(w)}\,.
$$
In this section we will focus on $W_{g,n}$ because of their relation to the Stieltjes transform, from which (the large-$N$ asymptotic expansion of) correlations of the spectral measure can be reconstructed. As we explain in the next paragraphs, we are in the situation of Section~\ref{sec:analyticprop} where there is a compact Riemann surface $\Sigma$ and a special point $q_{\infty}$ such that we have
\begin{itemize}
\item two meromorphic functions $\mathsf{x}$ and $\mathsf{w}$ on $\Sigma$ such that $W_{0,1} \circ \mathsf{x}$ is the series expansion of $\mathsf{w}$ near $q_{\infty}$ and $X_{0,1} \circ \mathsf{w}$ is the series expansion of $\mathsf{x}$ near $q_{\infty}$;
\item meromorphic $n$-forms $\Omega_{g,n}$ with series expansion near $(q_{\infty},\ldots,q_{\infty})$ given by
$$
W_{g,n}(\mathsf{x}_1,\ldots,\mathsf{x}_n)\prod_{i = 1}^{n} \dd \mathsf{x}_i\,;
$$ 
 \item meromorphic $n$-forms $\Omega_{g,n}^{\vee}$ on $\Sigma^n$ with series expansion near $(q_{\infty},\ldots,q_{\infty})$ given by
 $$
X_{g,n}(\mathsf{w}_1,\ldots,\mathsf{w}_n)\prod_{i = 1}^{n} \dd \mathsf{w}_i\,.
 $$
 \end{itemize}
 Here $\mathsf{x}_i$ refers to the function $\mathsf{x}$ applied to the $i$-th factor in $\Sigma^n$, and by $W_{g,n}(\mathsf{x}_1,\ldots,\mathsf{x}_n)$ we mean the composition $W_{g,n} \circ (\mathsf{x}_1,\ldots,\mathsf{x}_n)$, etc.  These notations will be used for the three cases A, B, and A + B.

\subsection{Spectral curves and analytic continuation of generating series}

For the GUE, the coefficient $F_{g;k_1,\ldots,k_n}^{{\rm A}}$ is the number of maps of genus $g$ with $n$ rooted boundary faces labelled from $1$ to $n$, having respective degrees $k_1,\ldots,k_n$, and with no internal faces, multiplied by $s^{k_1 + \cdots + k_n}$. On the other hand, $\kappa_{g;\lambda_1,\ldots,\lambda_n}^{{\rm A}}$ is the same count restricted to fully simple maps \cite{BG-F18}. As the only fully simple map without internal faces consists of a single edge drawn on the sphere, we have
\begin{equation}
\label{kappatrivGUE}
\kappa_{g;\lambda_1,\ldots,\lambda_n}^{{\rm A}} = \delta_{g,0}\delta_{n,1}\delta_{\lambda_1,2}\, s^2\,.
\end{equation}
For $(g,n)=(0,1)$, we deduce
$$
X_{0,1}^{{\rm A}}(w) = \frac{1}{w} + s^2 w\,,
$$
from which we obtain by functional inversion:
$$
W_{0,1}^{{\rm A}}(x) = \frac{x - \sqrt{x^2 - 4s^2}}{2s^2}\,,
$$
with the determination of the square root chosen such that $\sqrt{x^2 - 4s^2} \sim x$ when $x \rightarrow \infty$. This is described by the spectral curve $(\Sigma^{{\rm A}},\mathsf{x}^{A},\mathsf{w}^{{\rm A}})$ where $\Sigma^{{\rm A}}$ is the Riemann sphere, $\mathsf{w}^{{\rm A}} = w$ is the global coordinate on the Riemann sphere, and $\mathsf{x}^{{\rm A}}(w) = \frac{1}{w} + s^2 w$. The marked point is $q_{\infty}^{{\rm A}} = 0$ in the $w$-coordinate. For $2g - 2 + n \geq 0$, since $\Omega_{g,n}^{{\rm A},\vee} = 0$ for $2g - 2 + n \geq 0$, these admit analytic continuation as meromorphic $n$-forms $\Omega_{g,n}^{{\rm A},\vee}$ on $(\Sigma^{{\rm A}})^n$, and by Proposition~\ref{pr313}, so do the non-trivial $\omega_{g,n}^{{\rm A}}$ as meromorphic $n$-forms $\Omega_{g,n}^{{\rm A}}$ on $(\Sigma^{{\rm A}})^n$.

For $(B_N)_N$, the situation is opposite to the GUE. Being deterministic, the all-order moments are straightforward to compute: we have
\begin{equation}
\label{FgtrivBN}
F_{g;k_1,\ldots,k_n}^B = \delta_{g,0}\delta_{n,1} t L^{k_1}\,,
\end{equation}
but the all-order free cumulants however are not trivial. For $(0,1)$ we deduce
$$
W_{0,1}^{{\rm B}}(x) = \frac{t}{x - L} + \frac{1 - t}{x}\,,
$$
from which we obtain by functional inversion:
$$
X_{0,1}^{{\rm B}}(w) = \frac{1}{2}\Bigg(L + \frac{1}{w} + \sqrt{L^2 + \frac{2L(2t - 1)}{w} + \frac{1}{w^2}}\Bigg)\,,
$$
with the determination of the square root chosen such that $$\sqrt{L^2 + \frac{2L(2t - 1)}{w} + \frac{1}{w^2}} \sim \frac{1}{w}$$ as $w \rightarrow 0$. This is described by the spectral curve $(\Sigma^{{\rm B}},\mathsf{x}^{{\rm B}},\mathsf{w}^{{\rm B}})$, where $\Sigma^{{\rm B}}$ is the Riemann sphere, $\mathsf{x}^{{\rm B}} = x$ is the global coordinate on the Riemann sphere, and $\mathsf{w}^{{\rm B}}(x) = \frac{t}{x - L} + \frac{1 - t}{x}$. The marked point is $q_{\infty}^{{\rm B}} = \infty$ in the $x$-coordinate. For $2g - 2 + n \geq 0$, repeating the analyticity argument shows that $\omega_{g,n}^{{\rm B},\vee}$ admits an analytic continuation as a meromorphic $n$-form $\Omega_{g,n}^{{\rm B},\vee}$ on $(\Sigma^{{\rm B}})^n$.

The asymptotic freeness of $(A_N)_N$ and $(B_N)_N$ results in 
$$
X_{0,1}^{{\rm A} + {\rm B}}(w) = X_{0,1}^{{\rm A}}(w) + X_{0,1}^{{\rm B}}(w) - \frac{1}{w}\,.
$$
This is described by the spectral curve $(\Sigma^{{\rm A + B}},\mathsf{x}^{{\rm A} + {\rm B}},\mathsf{w}^{{\rm A} + {\rm B}})$, where $\Sigma^{{\rm A}+ {\rm B}}$ is also the Riemann sphere with global coordinate $x$ (the marked point is $q_{\infty}^{{\rm A} + {\rm B}} = \infty$), with
\begin{equation}
\label{spAplusB}
\mathsf{x}^{{\rm A} + {\rm B}}(x) = x + s^2 \mathsf{w}^{{\rm B}}(x)\,,\qquad \mathsf{w}^{{\rm A} + {\rm B}}(x) = \mathsf{w}^{{\rm B}}(x)\,.
\end{equation}
We have a natural biholomorphic identification $\Sigma^{{\rm B}} \cong \Sigma^{{\rm A} + {\rm B}}$, because the two spaces share the same global coordinate $x$. It will not lead to any confusion to write
$$
\mathsf{w}(x) = \mathsf{w}^{{\rm B}}(x) = \mathsf{w}^{{\rm A} + {\rm B}}(x) = \frac{t}{x - L} + \frac{1 - t}{x}\,,
$$
which we can consider as a meromorphic function on $\Sigma^{{\rm B}}$ or on $\Sigma^{{\rm A} + {\rm B}}$.  For $2g - 2 + n \geq 0$, the all-order freeness together with the vanishing of $\Omega_{g,n}^{{\rm A},\vee}$ yields $\omega_{g,n}^{{\rm A} + {\rm B},\vee} = \omega_{g,n}^{{\rm B},\vee}$. We already know that the right-hand side has an analytic continuation as a meromorphic $n$-differential on $(\Sigma^{{\rm B}})^n$. Then, the same holds for  $\omega_{g,n}^{{\rm A} + {\rm B},\vee}$ as a meromorphic $n$-differential $\Omega_{g,n}^{{\rm A} + {\rm B},\vee}$ on $(\Sigma^{{\rm A} + {\rm B}})^n$ (by the aforementioned identification) and $\omega_{g,n}^{{\rm A} + {\rm B}}$ as a meromorphic $n$-differential $\Omega_{g,n}^{{\rm A} + {\rm B}}$ (by Proposition~\ref{pr313}).

\begin{remark} Although they all are Riemann spheres, the spaces $\Sigma^{{\rm A}},\Sigma^{{\rm B}}$ and $\Sigma^{{\rm A} + {\rm B}}$ should a priori be treated as distinct spaces, and this is why we have used different notations for them. It is specific to this model that we have a biholomorphic identification $\Sigma^{{\rm B}} \cong \Sigma^{{\rm A} + {\rm B}}$, but there is no such identification between $\Sigma^{{\rm A}}$ and $\Sigma^{{\rm B}}$; instead, the natural map $\mathsf{w} : \Sigma^{{\rm A} + {\rm B}} \rightarrow \Sigma^{{\rm A}}$ is a ramified covering of degree $2$.
\end{remark}

In the following computations we shall present the analytic continuations $\Omega_{g,n}^{*}$ and $\Omega_{g,n}^{*,\vee}$ (cf.~\eqref{omegdiff}, to which we add $\Omega_{0,1}^* = \mathsf{w}^*\dd\mathsf{x}^*$ and $\Omega_{0,1}^{*,\vee} = \mathsf{x}^*\dd \mathsf{w}^*$) as meromorphic $n$-forms on $(\Sigma^*)^n$ using the global coordinate of $\Sigma^*$ ($w$ for A and $x$ for B and A + B). The analytic continuation will be denoted with the same symbol. 
They store the asymptotic expansion as $N \rightarrow \infty$ of the polynomial observables in the matrix ensembles in the following way (we only write the non-trivial ones). If $f_1,\ldots,f_n$ are polynomials
\begin{equation}
\label{contoursp}
\begin{split}
& \mathbb{E}^{\circ}\big[{\rm Tr}\, f_1(A_N), \ldots, {\rm Tr}\, f_n(A_N)\big] \\
& \quad = \sum_{g \geq 0} N^{2 - 2g - n}  (-1)^n \mathop{{\rm Res}}_{w_1 = 0} \cdots \mathop{{\rm Res}}_{w_n = 0} \Omega_{g,n}^{{\rm A}}(w_1,\ldots,w_n) \prod_{i = 1}^n f_i\big(\mathsf{x}^{{\rm A}}(w_i)\big) 
\\
&
\qquad
+ o(N^{-\infty})\,, \\
& \mathbb{E}^{\circ}\big[{\rm Tr}\, f_1(A_N + B_N), \ldots, {\rm Tr}\, f_n(A_N + B_N)\big] \\
& \quad = \sum_{g \geq 0} N^{2 - 2g - n} (-1)^n \mathop{{\rm Res}}_{x_1 = \infty} \cdots \mathop{{\rm Res}}_{x_n = \infty} \Omega_{g,n}^{{\rm A} + {\rm B}}(x_1,\ldots,x_n) \prod_{i  = 1}^n f_i\big(\mathsf{x}^{{\rm A} + {\rm B}}(x_i)\big) 
\\ 
&
\qquad
+ o(N^{-\infty})\,,
\end{split}
\end{equation}
while if $\lambda \vdash d$ has length $n$
\begin{equation}
\label{kcontour}
\begin{split}
& \mathbb{E}^{\circ}\big[\mathcal{P}_{\lambda}(B_N')\big] = 
\\ & \quad \sum_{g \geq 0} N^{2 - 2g - n - d} \mathop{{\rm Res}}_{x_1 = \infty} \cdots \mathop{{\rm Res}}_{x_n = \infty} \Omega_{g,n}^{{\rm B},\vee}(x_1,\ldots,x_n) \prod_{i = 1}^n \big(\mathsf{w}(x_i)\big)^{-\lambda_i}\,, \\
& \mathbb{E}^{\circ}\big[\mathcal{P}_{\lambda}(A_N + B_N')\big] = 
\\ & \quad 
\sum_{g \geq 0} N^{2 - 2g - n - d} \mathop{{\rm Res}}_{x_1 = \infty} \cdots \mathop{{\rm Res}}_{x_n = \infty} \Omega_{g,n}^{{\rm A} + {\rm B},\vee}(x_1,\ldots,x_n) \prod_{i = 1}^n \big(\mathsf{w}(x_i)\big)^{-\lambda_i} \,.
\end{split}
\end{equation}
These identities are the direct consequences of \eqref{EANBN}-\eqref{EANplusBN}-\eqref{EkBN}, the definition of the differential form-valued generating series and of their analytic continuation, and Cauchy residue formula.

\subsection{Moment generating series for the sum}

\begin{definition}
\label{nabladeg}
Let $E_k$ be the elementary symmetric polynomial in $x_1,\ldots,x_n$ of degree $k$ (used below for $n = 2$ and $n = 3$). Let $\nabla$ be the operator acting on rational functions of $x$ by
$$
\nabla f(x) \coloneqq \frac{\dd f(x)}{\dd \mathsf{x}^{{\rm A} + {\rm B}}(x)}\,.
$$
The relevance of this operator will appear in the proof of Proposition~\ref{prexp11}.
\end{definition}
The following proposition will be rephrased in distributional form in Section~\ref{measuredis} (Propositions \ref{prexp11}-\ref{prexp03}).
\begin{proposition}\label{proncun}
For $(g,n) = (0,2)$, we have
\begin{equation}
\label{02res}
\begin{split}
& \Omega_{0,2}^{{\rm A} + {\rm B}}(x_1,x_2)    = \frac{\dd x_1 \dd x_2}{(x_1 - x_2)^2} - \frac{\dd \mathsf{x}^{{\rm A} + {\rm B}}(x_1) \dd \mathsf{x}^{{\rm A} + {\rm B}}(x_2)}{\big(\mathsf{x}^{{\rm A} + {\rm B}}(x_1) - \mathsf{x}^{{\rm A} + {\rm B}}(x_2)\big)^2} \\
& = \frac{(1-t)s^2L\big[L(L^2 - ts^2) - 2L^2E_1 + L(E_1^2 + 2E_2) - 2E_1E_2\big] + s^2E_2^2}{\big((1-t)L^2s^2 - (1-t)Ls^2E_1 - (L^2 - s^2)E_2  + LE_1E_2 - E_2^2\big)^2} \dd x_1\dd x_2\,.
 \end{split}
\end{equation}
For $(g,n) = (0,3)$, we have
\begin{equation}
\label{03res}
\begin{split}
\Omega_{0,3}^{{\rm A} + {\rm B}}(x_1,x_2,x_3) & =  -s^2 \dd_{x_1}\dd_{x_2}\dd_{x_3}\bigg(\sum_{i =1}^{3} \frac{\dd x_i}{(x_i - x_j)(x_i - x_k)\dd \mathsf{x}^{{\rm A} + {\rm B}}(x_i)}\bigg) \\
& = \dd_{x_1} \dd_{x_2} \dd_{x_3}\bigg(s^4 F(x_1,x_2,x_3) \prod_{i = 1}^3 \frac{1}{x_i^2(x_i - L)^2}\,\frac{\dd x_i}{\dd \mathsf{x}^{{\rm A} + {\rm B}}(x_i)}\bigg)\,,
\end{split}
\end{equation}
where 
	\begin{equation*} 
\begin{split} 
& F(x_1,x_2,x_3)  = \\
& \quad -(1-t)^2L^4s^2(E_1 - L)^2 + (1-t)L\Big[L^3(3s^2t - L^2 - 2s^2)E_2 \\
& \quad + 2L^2\big(L^2 + s^2(1-t)\big)E_1E_2 + 2(1-3t)L^2s^2E_3 - L^3E_1^2E_2 
\\
& \quad - L(2L^2 + s^2)E_2^2  - 2(1-2t)Ls^2E_1E_3 + 2L^2E_1E_2^2 + 2(L^2 + s^2)E_2 E_3 
\\
& \quad 
- L\big(2E_1E_2E_3 + E_2^3\big) + 2E_2^2E_3\Big]  
- (s^2 + 3tL^2)E_3^2 + 2tLE_1E_3^2  - E_2E_3^2\,,
\end{split}
\end{equation*}
and in all terms of the sum we have $\{i,j,k\} = \{1,2,3\}$. For $(g,n) = (1,1)$, we have
\begin{equation}
\label{11res}
\Omega_{1,1}^{{\rm A} + {\rm B}}(x) = - \frac{s^4}{24}\,\dd\Bigg[\nabla^2\bigg(\frac{t}{(x - L)^2} + \frac{1 - t}{x^2}\bigg)\Bigg] = \frac{s^2}{24} \dd \Bigg[\nabla^2 \bigg(\frac{\dd \mathsf{x}^{{\rm A} + {\rm B}}}{\dd x}\bigg)\Bigg]\,.
\end{equation}
\end{proposition}

\begin{proof}
For the GUE, \eqref{kappatrivGUE} tells us that $\Omega_{0,2}^{{\rm A},\vee} = 0$. From \eqref{ome20nu}-\eqref{om02tildeequl} we deduce first
$$
\widetilde{\Omega}_{0,2}^{{\rm A}}(w_1,w_2) = \frac{\dd w_1 \dd w_2}{(w_1 - w_2)^2} \,,
$$
and subsequently:
$$
\Omega_{0,2}^{{\rm A}}(w_1,w_2) = \frac{\dd w_1\dd w_2}{(w_1 - w_2)^2} - \frac{\dd \mathsf{x}^{{\rm A}}(w_1)\dd \mathsf{x}^{{\rm A}}(w_2)}{(\mathsf{x}^{{\rm A}}(w_1) - \mathsf{x}^{{\rm A}}(w_2))^2} = \frac{s^2\,\dd w_1\dd w_2}{(1 - s^2 w_1w_2)^2}\,.
$$
For $(B_N)_N$, \eqref{FgtrivBN} tells us that $\Omega_{0,2}^{{\rm B}} = 0$. From \eqref{ome20nu}-\eqref{om02tildeequl} we deduce that
$$
\widetilde{\Omega}_{0,2}^{{\rm B}}(x_1,x_2) = \frac{\dd x_1 \dd x_2}{(x_1 - x_2)^2}\,,
$$
and consequently
\begin{equation}
\label{om02same}
\begin{split}
\Omega_{0,2}^{{\rm B},\vee}(x_1,x_2) & = \frac{\dd x_1 \dd x_2}{(x_1 - x_2)^2} - \frac{\dd \mathsf{w}(x_1) \dd \mathsf{w}(x_2)}{(\mathsf{w}(x_1) - \mathsf{w}(x_2))^2} 
\\
& = \frac{-L^2 t(1-t)\dd x_1\dd x_2}{(x_1x_2 - L(1-t)(x_1 + x_2) + L^2(1-t))^2}\,.
\end{split}
\end{equation}

For $(A_N + B_N)_N$, second-order asymptotic freeness together with the vanishing of the $(0,2)$-free cumulants of the GUE yields
$$
\Omega_{0,2}^{{\rm A} + {\rm B},\vee}(x_1,x_2) = \Omega_{0,2}^{{\rm B},\vee}(x_1,x_2) = \frac{\dd x_1 \dd x_2}{(x_1 - x_2)^2} - \frac{\dd \mathsf{w}(x_1)\dd \mathsf{w}(x_2)}{(\mathsf{w}(x_1) - \mathsf{w}(x_2))^2}\,.
$$
Hence
\begin{equation}\label{Abom2tilde2}
\widetilde{\Omega}_{0,2}^{{\rm A} + {\rm B},\vee}(x_1,x_2) =\widetilde{\Omega}_{0,2}^{{\rm A} + {\rm B}}(x_1,x_2) = \frac{\dd x_1 \dd x_2}{(x_1 - x_2)^2}\,,
\end{equation}
and
\begin{equation}
\label{Abom2tilde}\begin{split}
\Omega_{0,2}^{{\rm A} + {\rm B}}(x_1,x_2) & = \frac{\dd x_1 \dd x_2}{(x_1 - x_2)^2} - \frac{\dd \mathsf{x}^{{\rm A} + {\rm B}}(x_1)\dd \mathsf{x}^{{\rm A} + {\rm B}}(x_2)}{(\mathsf{x}^{{\rm A} + {\rm B}}(x_1) - \mathsf{x}^{{\rm A} + {\rm B}}(x_2))^2}\,.
\end{split}
\end{equation}
This evaluates to the announced expression \eqref{02res}.

The strategy is the same for other $(g,n)$. Equation~\eqref{FgtrivBN} tells us that $\Omega_{0,3}^{{\rm B}} = 0$. The functional relation \eqref{03form} then implies
$$   
\Omega_{0,3}^{{\rm B},\vee}(x_1,x_2,x_3) = \sum_{i = 1}^{3} \dd_{x_i}\bigg(\frac{\dd x_i}{\dd \mathsf{w}(x_i)}\,\frac{\dd x_j\,\dd x_k}{(x_i - x_j)^2(x_i - x_k)^2}\bigg)\,.
$$ 
Equation \eqref{kappatrivGUE} tells us that $\Omega_{0,3}^{{\rm A},\vee} = 0$. Since $(A_N)_N$ and $(B_N)_N$ are free at order $(0,3)$, we get $\Omega_{0,3}^{{\rm A} + {\rm B},\vee} = \Omega_{0,3}^{{\rm B},\vee}$. With the expression of $\widetilde{\Omega}_{0,2}^{{\rm A} + {\rm B},\vee}$ found in \eqref{Abom2tilde2}, we can use \eqref{03form} in the opposite direction and deduce
\begin{equation*}
\begin{split}
& \Omega_{0,3}^{{\rm A} + {\rm B}}(x_1,x_2,x_3) 
\\
& = - \Omega_{0,3}^{{\rm B},\vee}(x_1,x_2,x_3) 
+ \sum_{i = 1}^{3} \dd_{x_i}\bigg(\frac{(\dd x_i)^2}{\dd \mathsf{x}^{{\rm A} + {\rm B}}(x_i) \dd \mathsf{w}(x_i)}\,\frac{\dd x_j\,\dd x_k}{(x_i - x_j)^2(x_i - x_k)^2}\bigg) \\
& = \sum_{i = 1}^3  \dd_{x_i}\Bigg[\bigg(\frac{\dd x_i}{\dd \mathsf{x}^{{\rm A} + {\rm B}}(x_i)} - 1\bigg)\frac{\dd x_i}{\dd \mathsf{w}(x_i)}\,\frac{\dd x_j\,\dd x_k}{(x_i - x_j)^2(x_i - x_k)^2}\Bigg] \\
& = \sum_{i = 1}^3 \dd_{x_i}\bigg(\frac{-s^2\dd x_i}{\dd \mathsf{x}^{{\rm A} + {\rm B}}(x_i)}\,\frac{\dd x_j\,\dd x_k}{(x_i - x_j)^2(x_i - x_k)^2}\bigg) \\
& =  \dd_{x_1}\dd_{x_2}\dd_{x_3}\bigg(\sum_{i = 1}^{3} \frac{-s^2\dd x_i}{\dd \mathsf{x}^{{\rm A} + {\rm B}}(x_i)}\,\frac{1}{(x_i - x_j)(x_i - x_k)}\bigg)\,.
\end{split}
\end{equation*}
One can observe that the expression in the parenthesis has no poles at coinciding points, although the three terms individually do. Its evaluation yields \eqref{03res}.

Next, $\Omega_{1,1}^{{\rm B}} = 0$ and the functional relation \eqref{11formanal} yield
$$
\Omega_{1,1}^{{\rm B},\vee}(x) = \dd\bigg[\frac{1}{2}\,\frac{\Omega_{0,2}^{{\rm B},\vee}(x,x)}{\dd x \dd \mathsf{w}(x)} - \frac{1}{24 \dd x}\dd\bigg(\frac{\dd\big(\frac{\dd x}{\dd \mathsf{w}(x)}\big)}{\dd x}\bigg)\bigg] \,.
$$
The vanishing of $\Omega_{1,1}^{{\rm A},\vee}$ and the freeness at order $(1,1)$ of $(A_N)_N$ and $(B_N)_N$ imply $\Omega_{1,1}^{{\rm A}+{\rm B},\vee} = \Omega_{1,1}^{{\rm B},\vee}$. Using the functional relation \eqref{11formanal} in the opposite direction we obtain
\begin{equation*}
\begin{split}
&  \Omega_{1,1}^{{\rm A} + {\rm B}}(x) =\\
& \quad \dd\Bigg[\frac{1}{2}\,\frac{\Omega_{0,2}^{{\rm B},\vee}(x,x)}{\dd \mathsf{x}^{{\rm A} + {\rm B}}(x) \dd \mathsf{w}(x)} - \frac{1}{2}\,\frac{\Omega_{0,2}^{{\rm B},\vee}(x,x)}{\dd x \dd \mathsf{w}(x)} + \frac{1}{24 \dd x}\dd\bigg(\frac{\dd\big(\frac{\dd x}{\dd \mathsf{w}(x)}\big)}{\dd x}\bigg) 
\\
& \quad - \frac{1}{24 \dd \mathsf{x}^{{\rm A} + {\rm B}}(x)}\dd\bigg(\frac{\dd \big(\frac{\dd \mathsf{x}^{{\rm A} + {\rm B}}(x)}{\dd \mathsf{w}(x)}\big)}{\dd \mathsf{x}^{{\rm A} + {\rm B}}(x)}\bigg)\Bigg]\,. 
\end{split}
\end{equation*}
We can compute $\Omega_{0,2}^{{\rm B},\vee}$ from \eqref{om02same}. With help of \textsc{Mathematica} the result for $\Omega_{1,1}^{{\rm A} + {\rm B}}(x)$ can be presented in the simpler form  \eqref{11res}.
\end{proof}

\subsection{Structure of the spectral curve}

\label{dstrucsp}

The expressions found in Proposition~\ref{proncun} are sufficient to compute the polynomial moments of $A_N + B_N$ up to $o(1/N)$, using  \eqref{contoursp}. Yet, we have a further goal: the spectral measure of $A_N + B_N$ has support in $\mathbb{R}$, and we would like to extract an approximation for it (in the sense of an asymptotic expansion of cumulant expectation values when tested against a suitable class of functions) in terms of measures (if possible) or distributions (in general) with support in $\mathbb{R}$. This will be completed in Section~\ref{measuredis}, by converting the contour integrals in \eqref{contoursp} into integrals over $\mathbb{R}$. This procedure exploits the analyticity of $\Omega_{g,n}^{{\rm A} + {\rm B}}$ to move the contours on the spectral curve and get integrals over (a subset of) the preimage in $\Sigma$ of the real line in the $\mathsf{x}^{{\rm A} + {\rm B}}$-plane. We can carry this out only after understanding better the structure of the branched covering $\mathsf{x}^{{\rm A} + {\rm B}} : \Sigma^{{\rm A} + {\rm B}} \rightarrow \widehat{\mathbb{C}}$, given by
\begin{equation}
\label{xAplusB} \mathsf{x}^{{\rm A} + {\rm B}}(x) = x + \frac{ts^2}{x - L} + \frac{(1-t)s^2}{x}\,.
\end{equation}
In particular, we need to describe the $\mathsf{x}^{{\rm A} + {\rm B}}$-preimage of $\mathbb{R} \cup \{\infty\}$, which we call $\mathcal{R}$. This is the goal of this (technical) section.

The map $\mathsf{x}^{{\rm A} + {\rm B}}$ has degree $3$ and thus four ramification points (counted with multiplicity), which are the roots of 
$$
Q(x) = \frac{\dd\mathsf{x}^{{\rm A} + {\rm B}}(x)}{\dd x} = 1 - \frac{ts^2}{(x-L)^2} - \frac{(1-t)s^2}{x^2}\,.
$$
Since the quartic polynomial
\begin{equation}
\label{Qtildedef}
\tilde{Q}(x) \coloneqq x^2(x-L)^2Q(x) = x^4 - 2Lx^3  + (L^2 -s^2)x^2   + 2Ls^2(1-t)x - L^2s^2(1-t)
\end{equation}
has positive leading coefficient, and $\tilde{Q}(0)$ and $\tilde{Q}(L)$ are negative,  we always have at least two real ramification points such that $0$ and $L$ lie between them. We will use generically the letter $r$ for ramification points, and $b = \mathsf{x}^{{\rm A} + {\rm B}}(r)$ for the corresponding branchpoints. For $x \in \mathbb{C}$, we observe that
\begin{equation}
\label{Imdc} {\rm Im}\big(\mathsf{x}^{{\rm A} + {\rm B}}(x)\big) = {\rm Im}(x)\bigg(1 - \frac{ts^2}{|x - L|^2} - \frac{(1-t)s^2}{|x|^2}\bigg)\,.
\end{equation}
Therefore, $\mathcal{R} = \mathbb{R} \cup \Gamma$, where $\Gamma$ is the vanishing locus of the second factor in \eqref{Imdc}. We see that $\Gamma$ intersects the real axis at the real ramification points and can be parametrised by $\xi = {\rm Re}(x)$ in the form $\xi \pm {\rm i}\Upsilon_{\epsilon}(\xi)$, with
\begin{equation}
\label{Upsilon0}
\begin{split}
& \Upsilon_{\epsilon}(\xi) = 
\textstyle \sqrt{\frac{s^2 - \xi^2 - (\xi - L)^2 + \epsilon\sqrt{(L^2 + s^2)^2 - 4ts^2L^2 - 4(L^2 + s^2 - 2ts^2)L\xi + 4L^2\xi^2}}{2}}\,,
\end{split}
\end{equation}
for some $\epsilon \in \{\pm 1\}$, where the square roots in this expression are chosen to have nonnegative values. The topology of $\Gamma$, the range of the parametrisation and the allowed choices of $\epsilon$ depend on the number of real ramification points. The latter is determined by the sign of the discriminant of $\tilde{Q}(x)$:
$$
\mathsf{Disc} = 16t(1-t)L^2s^{4}\big((L^2 - s^2)^3 - 27t(1-t)L^2s^4\big)\,.
$$
At fixed  $t \in (0,1)$, the last factor in $\mathsf{Disc}$ has exactly two real roots at $\frac{L}{s} = \pm \lambda(t)$, with
$$
\lambda(t) = \sqrt{1 + 3t^{2/3}(1-t)^{1/3} + 3t^{1/3}(1-t)^{2/3}} \,\in (1,2]
$$
reaching its maximum value $2$ for $t = \frac{1}{2}$, and four non-real roots related to each other by a sign or/and complex conjugation. This leads to the following alternative, illustrated in Figure~\ref{fig:spcurvecircle}. In the following, $\mathbb{H}_{\pm} = \{x  \in \mathbb{C}\,\,|\,\,\pm {\rm Im}\,x > 0\}$ stands for the upper (resp.~lower) half-plane. We describe the three regimes in the case where $L>0$. 

\begin{figure}[h!]
\begin{center}
\includegraphics[width=0.32\textwidth]{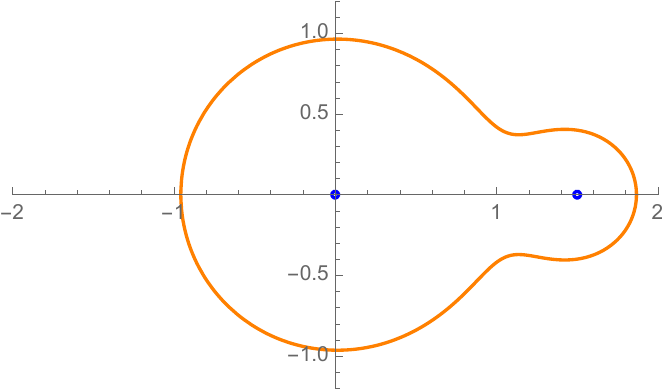}
\includegraphics[width=0.32\textwidth]{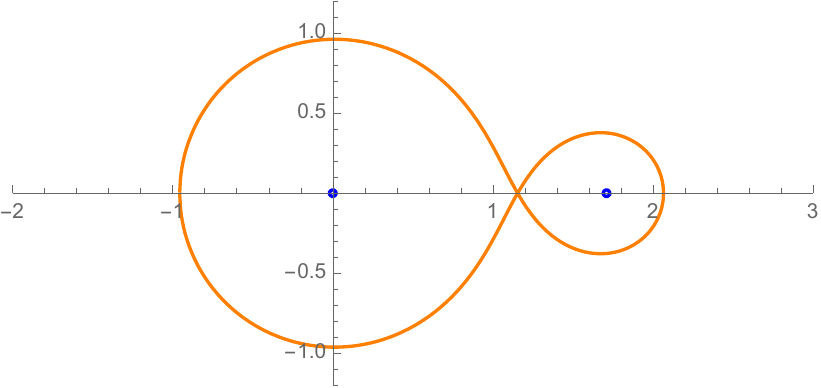}
 \includegraphics[width=0.32\textwidth]{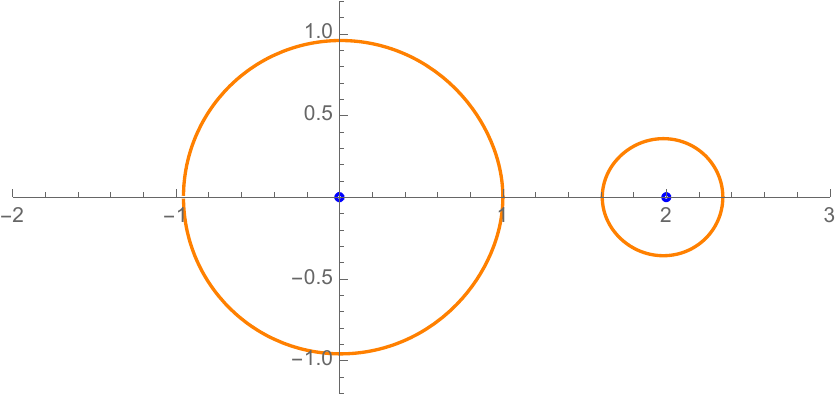}
 \caption{\label{fig:spcurvecircle} For $s = 1$, $t = 0.1$ and for the values (from left to right) $L = 1.5$, $L = \lambda(0.1) \approx 1.7094$ and $L = 2$, we display in the $x$-plane the locus $\Gamma$ in orange and the two finite poles of $\mathsf{x}^{{\rm A} + {\rm B}}$ (at $0$ and $L$) in blue.}
 \end{center}
 \end{figure}
 
\medskip

\noindent $\bullet$ \textit{If  $L < s \lambda(t)$}, we have $\mathsf{Disc} < 0$. Then we have two real ramification points satisfying
$$
r_-(t) < 0 < L < r_+(t)
$$
and a pair of complex conjugate ones. As $t \rightarrow 0$ they behave as
\begin{equation}
\label{rpmtum}
 r_{\pm}(t) = \pm s + \frac{sL(2s  \mp L)}{2(s \mp L)^2} t + O(t^2)\,,\qquad b_{\pm}(t) = \pm 2s + \frac{sL}{s \mp L} t + O(t^2)\,.
\end{equation}
The set $\Gamma$ is a topological circle and the two non-real ramification points belong\footnote{\label{note2}This is justified as follows. For a ramification point we have $$Q(r) = 1 - \frac{ts^2}{(r - L)^2} - \frac{(1-t)s^2}{r^2}\,.$$ If $r$ is not real, the ratio of the last two terms cannot be a positive real number. So, we have a strict triangular inequality $1 < \frac{ts^2}{|r - L|^2} + \frac{(1-t)s^2}{|r|^2}$, hence $r \notin \Gamma$. Since $\Gamma$ is a Jordan curve, $\mathbb{C} \setminus \Gamma$ has two connected components, exactly one of them being bounded. The two connected components can be distinguished by the sign  that the expression $\big(1 - \frac{ts^2}{|x - L|^2} - \frac{(1-t)s^2}{|x|^2}\big)$ assumes on each of them.  Letting $x \rightarrow \infty$ we see that the sign is positive, so the bounded connected component must correspond to the negative sign. Hence $r$ belongs to the bounded connected component.}
 to the bounded connected component of $\Sigma^{{\rm A} + {\rm B}} \setminus \Gamma$. Letting $I = (r_-(t),r_+(t))$, we have a parametrisation $\gamma_{\pm} : I \rightarrow \Gamma \cap \mathbb{H}_{\pm}$ given by $\gamma_{\pm}(\xi) = \xi \pm {\rm i}\Upsilon_{+}(\xi)$. We call $J = (b_-(t),b_+(t))$.
 
\medskip

\noindent $\bullet$ \textit{If $L = s \lambda(t)$}, we have $\mathsf{Disc} = 0$. Then we have three real ramification points\footnote{\label{note1}The fact that they are ordered in this way can be proved by checking it holds for $t$ small enough using their given asymptotics as $t \rightarrow 0$. Since $0$ and $L$ are not ramified it must remain true for all $t \in (0,1)$ by continuity.} satisfying
$$
r_-(t) <  0 < r_0(t) < L < r_+(t)\,,
$$
where $r_0(t)$ is double. As $t \rightarrow 0$, the two extreme ones behave as in \eqref{rpmtum} and the double one as
$$ 
r_0(t) = s\Big(1 + \frac{t^{1/3}}{2} + O(t^{2/3})\Big)\,,\qquad b_0(t) = 2s\Big(1 - \frac{3t^{2/3}}{8} + O(t)\Big)\,.
$$
The set $\Gamma$ is a union of two topological circles meeting at $r_0(t)$. We can still parametrise $\Gamma \cap \mathbb{H}_{\pm}$ with $\gamma_{\pm}$ given above, except that the range is now $I = (r_-(t),r_+(t)) \setminus \{r_0(t)\}$. We call $J = (b_-(t),b_+(t)) \setminus \{b_0(t)\}$.

\medskip

\noindent $\bullet$ \textit{If $L > s \lambda(t)$}, we have $\mathsf{Disc} > 0$. Then we have four real ramification points satisfying\footnote{See Footnote \ref{note1}.}
$$
r_-(t) < 0 < r_+(t) < \tilde{r}_-(t) < L < \tilde{r}_+(t)\,.
$$
As $t \rightarrow 0$, the two first ones behave as in \eqref{rpmtum}, while the two last ones behave like
\[
\begin{split}
 \tilde{r}_{\pm}(t) & = L \pm \frac{sL}{\sqrt{L^2 - s^2}}\sqrt{t} - \frac{s^4 L}{(L^2-s^2)^2}t+ O(t^{3/2})\,,\qquad \quad
\\ 
\tilde{b}_{\pm}(t) & = \Big(L + \frac{s^2}{L}\Big) \pm \frac{2s\sqrt{L^2 - s^2}}{L} \sqrt{t} + O(t)\,.	
\end{split}
\]
The set $\Gamma$ is a disjoint union of two topological circles. Letting $I = (r_-(t),r_+(t)) \cup (\tilde{r}_-(t),\tilde{r}_+(t))$, we have a parametrisation $\gamma_{\pm} : I \rightarrow \Gamma \cap \mathbb{H}_{\pm}$ with the same formula as above. We call $J = (b_-(t),b_+(t)) \cup (\tilde{b}_-(t),\tilde{b}_+(t))$.

\begin{definition}[Parametrisations]
\label{paramxc} Recall that the parametrisation $\gamma_{\pm} : I \rightarrow \Gamma \cap \mathbb{H}_\pm$ is given by
$$
\gamma_{\pm}(\xi) = \xi \pm {\rm i}\Upsilon_+(\xi)\,,
$$
with $\Upsilon_+(\xi)$ defined in \eqref{Upsilon0}. We introduce the parametrisation $c\coloneqq \mathsf{x}^{{\rm A} + {\rm B}} \circ \gamma_+ : I \rightarrow J$, which is equal to
\begin{equation}
\label{cixi} c(\xi) = 2\xi - \frac{L}{2} + \frac{s^2 - \sqrt{(L^2 + s^2)^2 - 4ts^2L^2 - 4(L^2 + s^2 - 2ts^2)L\xi + 4L^2\xi^2}}{2(2\xi - L)}\,.
\end{equation}
Although it is not manifest, this expression is regular at $\xi = L/2$.  By design, for $\xi \in I$ the arguments inside each square root in  $\gamma_+$ and $\Upsilon_+$ \eqref{Upsilon0} are positive and we choose their nonnegative determination.
\end{definition}

\begin{lemma}[Properties of the parametrisations]
\label{propparam} For any real ramification point $r$ and $t \in (0,1)$, $c'(r)$ is finite and $c'(r) \neq 0$ if $r$ is simple while $c'$ has a zero of order $2$ at  $r$ if $r$ is double. There exists a constant $c_r > 0$ such that
$$
\Upsilon_+(\xi) \,\,\mathop{\sim}_{\substack{\xi \in I \\ \xi \rightarrow r}} c_r |\xi - r|^{\frac{m}{2}}\,,
$$
where $m = 1$ if $r$ is simple and $m = 2$ if $r$ is double. Besides, this asymptotic equivalence can be differentiated with respect to $\xi$.
\end{lemma}
\begin{proof} The first part is a computation. For the second part we remark that $\Upsilon_+(\xi)^2= -A(\xi) + \sqrt{B(\xi)}$, where $A,B$ are polynomials that can be read from \eqref{Upsilon0} satisfying $A(\xi) \leq \sqrt{B(\xi)}$. We have $\Upsilon_+(r) = 0$ but one can check that for $t \in (0,1)$ the roots of $A$ are not roots of $\tilde{Q}$. By comparing with \eqref{Qtildedef} we find $B(\xi) - A^2(\xi) = -\tilde{Q}(\xi)$. Hence
$$
\Upsilon_+(\xi)^2 = \frac{B(\xi) - A(\xi)^2}{A(\xi) + \sqrt{B(\xi)}}
$$
is analytic near $\xi = r$ with a zero of order $m$.
\end{proof}
 
\begin{lemma}[Principal sheet]
Let $\Sigma^{\circ}$ be the unbounded connected component of $\Sigma^{{\rm A} + {{\rm B}}} \setminus \Gamma$ (the component containing $\infty$ in the $x$-coordinate) and $\Sigma_{\pm}^{\circ} = \Sigma^{\circ} \cap \mathbb{H}_{\pm}$. Then $\mathsf{x}^{{\rm A} + {\rm B}}(\Sigma^{\circ}_{\pm})\subseteq \mathbb{H}_{\pm}$ and $\mathsf{x}^{{\rm A} + {\rm B}}$ restricts to a biholomorphic map from $\Sigma^{\circ}$ to $\widehat{\mathbb{C}} \setminus \overline{J}$.
\end{lemma}

\begin{proof}
This is an elementary topological argument.

By construction, for $M > 0$ large enough we have $\pm {\rm i}M  \in \Sigma^{\circ}_{\pm}$ while \eqref{Imdc} implies that  the sign of ${\rm Im}\,\mathsf{x}^{{\rm A} + {\rm B}}$ is $\pm 1$. The map ${\rm sgn}({\rm Im}\,\mathsf{x}^{{\rm A} + {\rm B}})$ is continuous on $\Sigma^{\circ}_{\pm}$ and each of these two domains are connected, so we deduce that ${\rm sgn}({\rm Im}\,\mathsf{x}^{{\rm A} + {\rm B}}) = \pm 1$ on $\Sigma^{\circ}_{\pm}$, justifying the first claim.

Denote $\partial\Sigma^{\circ}_+$ the boundary of $\Sigma^{\circ}_+$ in the Riemann sphere $\Sigma^{{\rm A} + {\rm B}}$. Then, pick  $x_0^{\pm} \in \mathbb{H}_\pm \setminus \overline{\Sigma^{\circ}_\pm}$, introduce $i_{\pm}(x) = \frac{1}{x - x_0^{\pm}}$ that we consider as an automorphism of the Riemann sphere, and set $\Sigma^{\circ\,\prime}_\pm = i_{\pm}(\Sigma^{\circ}_{\pm})$. The map $f_{\pm} = i_{\pm} \circ \mathsf{x}^{{\rm A} + {\rm B}} \circ i_{\pm}^{-1}$ is continuous on the bounded domain $\overline{\Sigma^{\circ\,\prime}_\pm} \subset \mathbb{C}$, holomorphic in $\Sigma^{\circ\,\prime}_\pm$ with image included in $i_{\pm}(\mathbb{H}_{\pm})$, and $\partial\Sigma^{\circ\,\prime}_{\pm} = i_{\pm}(\partial\Sigma^{\circ}_\pm)$ is a Jordan curve. After we justify that the restriction of $\mathsf{x}^{{\rm A} + {\rm B}}$ to $\partial\Sigma^{\circ\,\prime}_\pm$ is a homeomorphism onto $i_{\pm}(\mathbb{R} \cup \{\infty\})$, a classical theorem of complex analysis  \cite[Chapter IV, \S 11]{Cana} implies that $f_{\pm}$ gives a biholomorphic map from $\Sigma^{\circ\,\prime}_{\pm}$ onto $i_{\pm}(\mathbb{H}_{\pm})$. Hence, the restriction of $\mathsf{x}^{{\rm A} + {\rm B}}$ to $\Sigma^{\circ}_\pm$ is biholomorphic onto $\mathbb{H}_{\pm}$. Together with $\mathsf{x}^{{\rm A} + {\rm B}}(\Sigma^{\circ}_{\pm}) \subset \mathbb{H}_{\pm}$ and the fact that $\mathsf{x}^{{\rm A} + {\rm B}}(\Sigma^{\circ} \cap \mathbb{R}) = \mathbb{R} \setminus \overline{J}$ (direct consequence of the variations of the restriction of $\mathsf{x}^{{\rm A} + {\rm B}}$ on $\mathbb{R}$), this implies that $\mathsf{x}^{{\rm A} + {\rm B}}$ is a biholomorphic map from $\Sigma^{\circ}$ to $\widehat{\mathbb{C}}  \setminus \overline{J}$.
 
To justify the injectivity on the boundary, we remark that the map $x \mapsto {\rm Re}\,x$ restricts to a homeomorphism $\psi_\pm$ from $\partial\Sigma^{\circ}_{\pm} \setminus \{\infty\}$ to $\mathbb{R}$, which is real-analytic away from the ramification points. Then $h_\pm = \mathsf{x}^{{\rm A} + {\rm B}} \circ \psi^{-1}_{\pm} :  \mathbb{R} \rightarrow \mathbb{R}$ is continuous, and away from the ramification points it is real-analytic. As $\Gamma$ does not contain ramification points (cf.~Footnote~\ref{note2}) and the differential of $\mathsf{x}^{{\rm A} + {\rm B}}$ is $\mathbb{C}$-linear, $h_{\pm}'$ does not vanish away from ramification points. Hence, $h_\pm$ is strictly monotone and $\mathsf{x}^{{\rm A} + {\rm B}}$ is injective on $\partial\Sigma^{\circ}_\pm \setminus \{\infty\}$. Since the latter does not contain the poles of $\mathsf{x}^{{\rm A} + {\rm B}}$ located at $0,L,\infty$, we conclude that $\mathsf{x}^{{\rm A} + {\rm B}}$ is injective on $\partial\Sigma^{\circ}_\pm$, as desired.
\end{proof}

\subsection{Distributional form}
\label{measuredis}

It is well-known that the $(0,1)$-moments come from a positive measure on $\mathbb{R}$ (called \emph{equilibrium measure}). The BBP phase transition manifests itself as $t \rightarrow 0$ by a change of the topology of the support: if $L < s$ the support is a segment close to $[-2s,2s]$, while if $L > s$ it also contains a second segment of width $O(\sqrt{t})$ around $L + \frac{s^2}{L}$. Section~\ref{dstrucsp} revisited this transition and contains all the necessary information to describe explicitly the equilibrium measure for $t \in (0,1)$.

\begin{lemma}[Law of large numbers]
\label{01lemmueq} For any polynomial $f$, we have
$$
\mathbb{E}\big[{\rm Tr}\,f (A_N + B_N)] \mathop{=}_{N \rightarrow \infty} N \int_{J} f \dd \mu_{0,1} + O\Big(\frac{1}{N}\Big)\,,
$$
where $\mu_{0,1}$ is the probability measure with support $\overline{J}$, whose pullback by the parametrisation $c : I \rightarrow J$ is
$$
c^*\mu_{0,1} = \frac{\Upsilon_{+}(\xi)}{\pi s^2}\mathbf{1}_{I}(\xi) \dd c(\xi) \,.
$$
If $L \leq s \lambda(t)$, the support consists of one segment. If $L > s\lambda(t)$, the support consists of two segments, and $\mu_{0,1}$ assigns mass $t$ to the rightmost one.
\end{lemma} 

\begin{remark}
From \eqref{Upsilon0}-\eqref{cixi} we can determine the behaviour of the equilibrium measure as $t \rightarrow 0$. As expected, it becomes semi-circular at the macroscopic level:
$$
\lim_{t \rightarrow 0} \frac{\dd\mu_{0,1}(\mathsf{x})}{\dd \mathsf{x}} = \frac{\sqrt{4s^2 - \mathsf{x}^2}}{2\pi s^2} \, \mathbf{1}_{[-2s,2s]}(\mathsf{x})
$$
uniformly for $\mathsf{x} \in \mathbb{R}$. Besides, if $L >  s\lambda(0) = s$, the measure also has in its support a segment with size of order $\sqrt{t}$ and mass $t$, where it also looks semi-circular:
$$
\frac{\dd\mu_{0,1}(\mathsf{x}(\mathsf{v}))}{\dd \mathsf{v}} \mathop{=}_{t \rightarrow 0} \frac{2t\sqrt{1 - \mathsf{v}^2}}{\pi}\mathbf{1}_{[-1,1]}(\mathsf{v})  + O(t^{\frac{5}{4}}),
$$
where
$$
\mathsf{x}(\mathsf{v}) = L + \frac{s^2}{L} + \frac{2s}{L}\sqrt{L^2-s^2}\sqrt{t}\,\mathsf{v}
$$
uniformly for real $\mathsf{v}$ on any compact set.
\end{remark}

\begin{proof}
From \eqref{contoursp} we get for $N \rightarrow \infty$
\[
\mathbb{E}\big[{\rm Tr}\,f(A_N + B_N)\big] \mathop{=}_{N \rightarrow \infty}\,\, N\,\mathcal{M}_{0,1}[f] + O\Big(\frac{1}{N}\Big),\;
\]
 with
\[\mathcal{M}_{0,1}[f] \coloneqq  - \mathop{{\rm Res}}_{x = \infty} f\big(\mathsf{x}^{{\rm A} + {\rm B}}(x)\big) \mathsf{w}(x)\dd \mathsf{x}^{{\rm A} + {\rm B}}(x)\,.	
\]
We use the formula \eqref{spAplusB} for $\big(\mathsf{x}^{{\rm A} + {\rm B}}(x),\mathsf{w}(x)\big)$ and move the contour in $\Sigma^{{\rm A} + {\rm B}}$ to rewrite
\begin{equation}
\label{M091f}\begin{split}
\mathcal{M}_{0,1}[f] & =  \frac{1}{s^2} \mathop{{\rm Res}}_{x = \infty} f\big(\mathsf{x}^{{\rm A} + {\rm B}}(x)\big) \big(x - \mathsf{x}^{{\rm A} + {\rm B}}(x)\big) \dd \mathsf{x}^{{\rm A} + {\rm B}}(x)  
\\
&
= \frac{1}{s^2}\mathop{{\rm Res}}_{x = \infty} f\big(\mathsf{x}^{{\rm A} + {\rm B}}(x)\big)\,x\,\dd \mathsf{x}^{{\rm A} + {\rm B}}(x)  \\
& = \frac{-1}{2{\rm i}\pi s^2} \oint_{\Gamma} f\big(\mathsf{x}^{{\rm A} + {\rm B}}(x)\big)\,x\,\dd \mathsf{x}^{{\rm A} + {\rm B}}(x)\,,
\end{split}
\end{equation}
where $\Gamma$ is equipped with the counterclockwise orientation (see Figure~\ref{fig:spcurvecircle}, the contour surrounding $\infty$ being moved to the curve $\Gamma$ shown in orange). Between the first and second line we used the fact that the residue at $\infty$ of a polynomial function is zero. Recall that $x \in \Gamma$ is parametrised by $\xi \in I \subset \mathbb{R}$ as $x = \gamma_{\pm}(\xi) = \xi \pm {\rm i}\Upsilon_+(\xi)$, and that $c(\xi) = \mathsf{x}^{{\rm A} + {\rm B}}(\xi \pm {\rm i}\Upsilon_+(\xi))$ given in \eqref{cixi} does not depend on the sign $\pm$. In consequence, \eqref{M091f} can be converted to an integral over $I$ (oriented from left to right) 
\[
\begin{split}
\mathcal{M}_{0,1}[f] & = \frac{-1}{2{\rm i}\pi s^2} \int_{I} f(c(\xi))\Big(\big(\xi - {\rm i}\Upsilon_{+}(\xi)\big) - \big(\xi + {\rm i}\Upsilon_+(\xi)\big)\Big)\dd c(\xi) 
\\ & = \int_{I} f(c(\xi))\,\frac{\Upsilon_+(\xi) \dd c(\xi)}{\pi s^2}\,,
\end{split}
\]
and we recognise the announced result.

The support of this measure is $c(\overline{I}) = \overline{J}$. It was described in Section~\ref{dstrucsp} and we saw that it depends on the relative position of $L$ and $s\lambda(t)$. When $L > s\lambda(t)$, it consists of two segments, and the rightmost one (call it $J_2$) is the $\mathsf{x}^{{\rm A} + {\rm B}}$-image of the component of $\Gamma$ surrounding $x = L$ (call it $\Gamma_2$). We can repeat the previous manipulations to compute the mass of $J_2$ as the contour integral in the $x$-plane
$$
\mu_{0,1}[J_2] = \frac{1}{2{\rm i}\pi} \oint_{\Gamma_2}  \mathsf{w}(x)\,\dd \mathsf{x}^{{\rm A} + {\rm B}}(x)\,.
$$
In view of \eqref{spAplusB}, the integrand is meromorphic inside $\Gamma_2$ with a pole at $x = L$, so the integral is evaluated by a residue at this point, and we find $\mu_{0,1}[J_2] = t$.  
\end{proof}

We go further by demonstrating that $(0,2)$-moments come from a measure, while the $(g,n)$-moments for $2g - 2 + n > 0$ come from a finite-order distribution. The support of these measures or distributions is independent of $(g,n)$ and corresponds to $\overline{J}$ introduced in Section~\ref{dstrucsp}.

\begin{proposition}[Subleading correction to the law of large numbers] \label{prexp11}
For any polynomial $f$, we have as $N \rightarrow \infty$
$$
\mathbb{E}\big[{\rm Tr}\,f(A_N + B_N)\big] = N \int_{J} f\,\dd \mu_{0,1} + \frac{1}{N} \int_{J} f'''\, \dd \mu_{1,1} + O\Big(\frac{1}{N^3}\Big)\,,
$$
where $\mu_{1,1}$ is the signed measure of mass $0$ on $\overline{J}$ given by
$$
c^*\mu_{1,1} = - \frac{s^2}{24\pi}\,\frac{c'(\xi) \Upsilon_+'(\xi)}{1 + (\Upsilon_+'(\xi))^2}\ \mathbf{1}_{I}(\xi)\dd c(\xi)\,.
$$
If $L > s\lambda(t)$, then $\mu_{1,1}$ gives mass $\frac{s^6t(1-t)}{6L^3}$ to the rightmost segment. Besides, as $t \rightarrow 0$
$$
\lim_{t \rightarrow 0} \frac{\dd\mu_{1,1}(\mathsf{x})}{\dd \mathsf{x}} =  \frac{\mathsf{x}\sqrt{4s^2 - \mathsf{x}^2} }{48\pi}\mathbf{1}_{[-2s,2s]}(\mathsf{x})
$$
when integrated against a polynomial. Moreover, if $L > s\lambda(0) = s$, setting $\mathsf{x}(\mathsf{v}) = L + \frac{s^2}{L} + \frac{2s}{L}\sqrt{L^2-s^2}\sqrt{t}\,\mathsf{v}$, we have
$$
\frac{\dd \mu_{1,1}(\mathsf{x}(\mathsf{v}))}{\dd \mathsf{v}} \mathop{=}_{t \rightarrow 0}  \frac{s^3(L^2 - s^2)^{3/2}}{6\pi L^3}\,\mathsf{v}\sqrt{1-\mathsf{v}^2} \mathbf{1}_{[-1,1]}(\mathsf{v})\sqrt{t} + O(t^{\frac{3}{4}})
$$
uniformly for $\mathsf{v}$ in any compact of $\mathbb{R}$.
\end{proposition}

\begin{proposition}[Central limit theorem]  \label{prexp02}
Let $f$ be a polynomial. The random variable
$$
\varkappa_N[f] \coloneqq {\rm Tr}\,f(A_N + B_N) - N \int_{J} f \dd \mu_{0,1}
$$
converges in law to a centered Gaussian random variable as $N\to\infty$, with variance
$$
\frac{1}{2\pi^2} \int_{I^2} \ln\Bigg|\frac{\gamma_+(\xi_1) - \overline{\gamma_+(\xi_2)}}{\gamma_+(\xi_1) - \gamma_+(\xi_2)}\Bigg|\,f'(c(\xi_1))f'(c(\xi_2))\dd c(\xi_1)\dd c(\xi_2)\,.
$$
In other words, for any polynomial $f$, the random variable $\varkappa_N[f]$ converges in law to $\frac{1}{\sqrt{\pi}}\langle\mathscr{G},f' \circ \mathsf{x}^{{\rm A} + {\rm B}} \rangle$, where $\mathscr{G}$ is the restriction to $\Gamma \cap \mathbb{H}_+$ of the Gaussian free field in $\mathbb{H}_+$ with Dirichlet boundary condition.
\end{proposition}

\begin{proposition}[Leading asymptotics of third-order cumulants] \label{prexp03}
For any polynomial $f$, we have as $N \rightarrow \infty$
\begin{equation}
\label{Ecirc03}
\begin{split}
& \mathbb{E}^{\circ}\big[{\rm Tr}\,f(A_N + B_N),{\rm Tr}\,f(A_N + B_N),{\rm Tr}\,f(A_N + B_N)\big] = 
\\
& \quad \frac{1}{N} \int_{J^3} f' \otimes f' \otimes f'\, \dd \mu_{0,3} + O\Big(\frac{1}{N^3}\Big)\,,
\end{split}
\end{equation}
where $\mu_{0,3}$ is a finite signed measure supported on $\overline{J}^3$ such that
\begin{equation}
\label{ccc03}
(c,c,c)^*\mu_{0,3}=  \frac{s^2}{8{\rm i}\pi^3} \sum_{\substack{\epsilon_1,\epsilon_2,\epsilon_3 = \pm 1 \\ i = 1,2,3}}  \frac{\epsilon_1\epsilon_2 \epsilon_3 \gamma_{\epsilon_i}'(\xi_i)\,\mathbf{1}_{I^3}(\xi_1,\xi_2,\xi_3)\dd c(\xi_1)\dd c(\xi_2) \dd c(\xi_3)}{c'(\xi_i)\big(\gamma_{\epsilon_i}(\xi_i) - \gamma_{\epsilon_j}(\xi_j)\big)\big(\gamma_{\epsilon_i}(\xi_i) - \gamma_{\epsilon_k}(\xi_k)\big)}
\end{equation}
(the explicit expression is long and not particularly enlightening).
\end{proposition}

\begin{proof}[Proof of Proposition~\ref{prexp11}]
We first observe that if $h$ is a meromorphic function on $\Sigma^{{\rm A} + {\rm B}}$ without pole at $\infty$ and $f$ is a polynomial, we have by integration by parts
$$
\mathop{{\rm Res}}_{x = \infty} f\big(\mathsf{x}^{{\rm A} + {\rm B}}(x)\big) \dd h(x)  = -\mathop{{\rm Res}}_{x = \infty} f'\big(\mathsf{x}^{{\rm A} + {\rm B}}(x)\big)  h(x) \dd \mathsf{x}^{{\rm A} + {\rm B}}(x)\,.
$$
Recalling Definition~\ref{nabladeg} for the operator $\nabla$, we have likewise for $n \in \mathbb{Z}_{> 0}$
\begin{equation}
\label{nablamoin}
\mathop{{\rm Res}}_{x = \infty} f\big(\mathsf{x}^{{\rm A} + {\rm B}}(x)\big) \dd\big( \nabla^n h(x) \big) = (-1)^{n+1}\mathop{{\rm Res}}_{x = \infty} f^{(n+1)}\big(\mathsf{x}^{{\rm A} + {\rm B}}(x)\big) h(x) \dd \mathsf{x}^{{\rm A} + {\rm B}}(x)\,.
\end{equation}
The term of order $1/N$ in the $N \rightarrow \infty$ asymptotic expansion of $\mathbb{E}\big[{\rm Tr}\,f(A_N + B_N)\big]$ is 
$$
\mathcal{M}_{1,1}[f] = - \mathop{{\rm Res}}_{x = \infty} f\big(\mathsf{x}^{{\rm A} + {\rm B}}(x)\big)\,\Omega_{1,1}^{{\rm A} + {\rm B}}(x)\,.
$$
As $\Omega_{1,1}^{{\rm A} + {\rm B}}$ has poles at the ramification points, we cannot directly move the contour to $\Gamma$ to rewrite it as an integral over $J$. Instead, we first use integration by parts. With the formula \eqref{11res} for $\Omega_{1,1}^{{\rm A} + {\rm B}}$, using the identity \eqref{nablamoin}, we get
\begin{equation}
\label{M11ffirst}
\mathcal{M}_{1,1}[f] = \frac{s^2}{24} \mathop{{\rm Res}}_{x = \infty} f'''\big(\mathsf{x}^{{\rm A} + {\rm B}}(x)\big) \frac{(\dd \mathsf{x}^{{\rm A} + {\rm B}})^2}{\dd x}\,.
\end{equation}
Now, recalling the expression \eqref{xAplusB} for $\mathsf{x}^{{\rm A} + {\rm B}}$, the integrand has poles only at $x = 0,L,\infty$, so we can move the contour from $\infty$ to $\Gamma$, and proceed as in the proof of Lemma~\ref{01lemmueq} to obtain
\begin{equation}
\begin{split}
\mathcal{M}_{1,1}[f] & = \frac{s^2}{48{\rm i}\pi}\,\int_{I} f'''(c(\xi)) \frac{(\dd c(\xi))^2}{\dd \xi}\Big(\frac{1}{1 + {\rm i}\Upsilon_+'(\xi)} - \frac{1}{1 - {\rm i}\Upsilon_+'(\xi)}\Big) \\
& =-\frac{s^2}{24\pi} \int_{I} f'''(c(\xi))\,\frac{c'(\xi)\Upsilon_+'(\xi)}{1 + (\Upsilon_+'(\xi))^2}\,\dd c(\xi)\,.
\end{split}
\end{equation}
Replacing $f'''$ with $1$ in \eqref{M11ffirst} computes $\mu_{1,1}[J]$, and it is equal to $0$ because the integrand has vanishing residue at $\infty$. If $L > s\lambda(t)$, we can distinguish the mass of the rightmost segment
$$
\mu_{1,1}[J_2] = \frac{s^4}{24} \mathop{{\rm Res}}_{x = L} \Big(\frac{t}{(x - L)^2} + \frac{1 - t}{x^2}\Big) \dd \mathsf{x}^{{\rm A} + {\rm B}}(x) = \frac{s^6 t(1-t)}{6L^3}\,.
$$
\end{proof}

\begin{proof}[Proof of Proposition~\ref{prexp02}]
In \eqref{contoursp} we read that  the variance of $\varkappa_N[f]$ as $N \rightarrow \infty$ is
$$
{\rm Var}[\varkappa_N[f]\big] = \mathcal{M}_{0,2}[f] + O(N^{-2})\,,
$$
where 
\begin{equation}
\label{M02resubis}
\mathcal{M}_{0,2}[f] \coloneqq \mathop{{\rm Res}}_{x_1 = \infty}  \mathop{{\rm Res}}_{x_2 = \infty} \Omega_{0,2}^{{\rm A} + {\rm B}}(x_1,x_2) \prod_{j = 1}^{2} f\big(\mathsf{x}^{{\rm A} + {\rm B}}(x_j)\big)\,.
\end{equation}
Furthermore, from \eqref{contoursp} we know that $\varkappa_N[f]$ has mean $O(N^{-1}) = o(1)$, and for $n \geq 3$, that its $n$-th cumulant is $O(N^{2 - n}) = o(1)$. This implies that the moments of $\varkappa_N[f]$ converge to the moments of a centered Gaussian random variable with variance $\mathcal{M}_{0,2}[f]$. By the method of moments, this implies convergence in law of $\varkappa_N[f]$ towards this Gaussian.

Recall from Proposition~\ref{proncun} the formula
\begin{equation}
\Omega_{0,2}^{{\rm A} + {\rm B}}(x_1,x_2) = \frac{\dd x_1 \dd x_2}{(x_1 - x_2)^2} - \frac{\dd \mathsf{x}^{{\rm A} + {\rm B}}(x_1)\dd \mathsf{x}^{{\rm A} + {\rm B}}(x_2)}{\big(\mathsf{x}^{{\rm A} + {\rm B}}(x_1) - \mathsf{x}^{{\rm A} + {\rm B}}(x_2)\big)^2}\,.
\end{equation}
The residues in \eqref{M02resubis} can be realised by integration along $x_j$ on clockwise-oriented circles with large radii $\rho_1$ and $\rho_2$, and we can choose $\rho_2 > \rho_1$. Calling $C_1,C_2$ those circles with counterclockwise orientation, we get
\[
\begin{split}
&
\mathcal{M}_{0,2}[f] = 
\\
&
\quad
\frac{-1}{4\pi^2} \oint_{C_1} \oint_{C_2} \bigg(\frac{\dd x_1 \dd x_2}{(x_1 -x_2)^2} - \frac{\dd \mathsf{x}^{{\rm A} + {\rm B}}(x_1) \dd \mathsf{x}^{{\rm A} + {\rm B}}(x_2)}{\big(\mathsf{x}^{{\rm A} + {\rm B}}(x_1) - \mathsf{x}^{{\rm A} + {\rm B}}(x_2)\big)^2}\bigg) \prod_{j= 1}^{2} f\big(\mathsf{x}^{{\rm A} + {\rm B}}(x_j)\big)\,.
\end{split}
\] 
As each term in the integrand is regular on $C_1 \times C_2$, we can split the integral into a sum of two. For the second term, we can change variable to $\mathsf{x}_j = \mathsf{x}^{{\rm A} + {\rm B}}(x_j)$. With respect to $\mathsf{x}_1$ the integrand is holomorphic in the bounded connected component of $\mathbb{C} \setminus \mathsf{x}^{{\rm A} + {\rm B}}(C_1)$, therefore the integral vanishes. For the first term, we can move $C_1$ to $\Gamma$, and then move $C_2$ to a contour $\Gamma^{+\eta}$ slightly larger than $\Gamma$, at distance $\eta$ from the latter.

For simplicity of the computations, we assume $L< s \lambda(t)$ so that $C_1$ and $C_2$ have only one component (the argument transposes easily to the case $L\geq s \lambda(t)$). We use the notation $C_j^{\pm}\coloneqq C_j\cap \mathbb{H}_{\pm}$ for the oriented arcs ($-C_j^{\pm}$ is the arc $C_j^{\pm}$ with the opposite orientation); the endpoints of $C_1^{\pm}$ are $r_{\pm}(t)$ (see Section \ref{dstrucsp}) and the endpoints of $C_2^{\pm}$ are $r_{\pm}(t)\pm \eta$.

The strategy is to perform integration by parts with respect to $x_2$ and $x_1$. This makes a complex logarithm appear. In order to avoid discontinuities in the definition of complex logarithms, we will try to use real logarithms. For this purpose we decompose the integral into four pieces:
\[\begin{split}
\mathcal{M}_{0,2}[f]& = \frac{-1}{4\pi^2}\Bigg(\int_{C_1^{+}}\int_{C_2^{+}}+\int_{C_1^{-}}\int_{C_2^{-}}+\int_{C_1^{+}}\int_{C_2^{-}}
\\
& \quad +\int_{C_1^{-}}\int_{C_2^{+}}\Bigg)\frac{\dd x_1 \dd x_2}{(x_1-x_2)^2}\prod\limits_{j=1}^{2}f\big(\mathsf{x^{{\rm A}+ {\rm B}}}(x_j)\big) \\
& = \frac{-1}{4\pi^2}\int_{-C_1^{+}}\int_{-C_2^{+}}\Bigg(\frac{\dd x_1 \dd x_2}{(x_1-x_2)^2}\prod\limits_{j=1}^{2}f\big(\mathsf{x^{{\rm A}+ {\rm B}}}(x_j)\big)
\\
& \quad +\frac{\dd \overline{x}_1 \dd \overline{x}_2}{(\overline{x}_1-\overline{x}_2)^2}\prod\limits_{j=1}^{2}f\big(\mathsf{x^{{\rm A}+ {\rm B}}}(\overline{x}_j)\big)
\\
& \quad  - \frac{\dd x_1 \dd \overline{x}_2}{(x_1-\overline{x}_2)^2}f\big(\mathsf{x^{{\rm A}+ {\rm B}}}(x_1)\big)f\big(\mathsf{x^{{\rm A}+ {\rm B}}}(\overline{x}_2)\big)
\\
& \quad -\frac{\dd \overline{x}_1 \dd x_2}{(\overline{x}_1-x_2)^2}f\big(\mathsf{x^{{\rm A}+ {\rm B}}}(\overline{x}_1)\big)f\big(\mathsf{x^{{\rm A}+ {\rm B}}}(x_2)\big)\Bigg)\,,
\end{split}\]
where we have used that $\int_{C_j^{-}}h(x_j)\dd x_j = \int_{-C_j^{+}}h(\overline{x}_j)\dd \overline{x}_j$ to get the second equality. We claim that we can replace $f\big(\mathsf{x}^{{\rm A}+ {\rm B}}(\overline{x}_j)\big)$ with $f\big(\mathsf{x}^{{\rm A}+ {\rm B}}(x_j)\big)$ in these integrals. For $x_1 \in C_1^+$ this is clear: we have $\mathsf{x}^{{\rm A}+ {\rm B}}(\overline{x}_1)=\mathsf{x}^{{\rm A}+ {\rm B}}(x_1)$ hence $f\big(\mathsf{x}^{{\rm A}+ {\rm B}}(\overline{x}_1)\big)=f\big(\mathsf{x}^{{\rm A}+ {\rm B}}(x_1)\big)$. For $x_2 \in C_2^{+} = \Gamma^{+\eta}$, we only have
$$
\underset{\eta\to 0}{\lim}\, \big(f(\mathsf{x}^{{\rm A}+ {\rm B}}(\overline{x}_2))-f(\mathsf{x}^{{\rm A}+ {\rm B}}(x_2))\big)=0\,.
$$
and we have to justify that
\begin{equation}
\label{underst}
\underset{\eta\to 0}{\lim}\int_{-C_1^+}\int_{-C_2^{+}} \frac{\dd \overline{x}_1\dd \overline{x}_2}{(\overline{x}_1-\overline{x}_2)^2} f\big(\mathsf{x}^{{\rm A}+{\rm B}}(x_1)\big)\Big[f\big(\mathsf{x}^{{\rm A}+{\rm B}}(\overline{x}_2)\big)-f\big(\mathsf{x}^{{\rm A}+{\rm B}}(x_2)\big)\Big] = 0\,.
\end{equation}
We start with the integral in the left-hand side for fixed $\eta > 0$, and first perform integration by parts with respect to $x_2$. This does not create boundary terms because the endpoints of $\Gamma^{+\eta} \cap \mathbb{H}_{\pm}$ are on the real line and the last factor in \eqref{underst} vanishes there. We then perform integration by parts with respect to $x_1$. It creates a boundary term and makes a complex logarithm appear. For $(x_1,x_2)\in C_1^+\times C_2^+$ we have $(\overline{x}_1-\overline{x}_2) \notin -{\rm i}\mathbb{R}$, so we can choose the complex logarithm having a cut on $-{\rm i}\mathbb{R}_{\geq 0}$ and coinciding with the real logarithm on $\mathbb{R}_{> 0}$. This rewrites the left-hand side of \eqref{underst} before taking the limit as
\begin{equation*}
\begin{split}
& \int_{-C_1^+}\int_{-C_2^{+}} \ln (\overline{x}_1-\overline{x}_2)f'\big(\mathsf{x}^{{\rm A}+{\rm B}}(x_1)\big)\Big[f'\big(\mathsf{x}^{{\rm A}+{\rm B}}(\overline{x}_2)\big)\dd \mathsf{x}^{{\rm A}+{\rm B}}(\overline{x}_2)\\
& \qquad \qquad \qquad -f'\big(\mathsf{x}^{{\rm A}+{\rm B}}(x_2)\big)\dd \mathsf{x}^{{\rm A}+{\rm B}}(x_2)\Big]\dd \mathsf{x}^{{\rm A} + {\rm B}}(x_1) \\
& - \sum_{\epsilon = \pm 1} \int_{-C_2^+} \ln(b_{\epsilon}(t) - \overline{x}_2) f(b_\epsilon(t))\Big[f'\big(\mathsf{x}^{{\rm A}+{\rm B}}(\overline{x}_2)\big)\dd \mathsf{x}^{{\rm A}+{\rm B}}(\overline{x}_2)
\\
& \qquad \qquad \qquad -f'\big(\mathsf{x}^{{\rm A}+{\rm B}}(x_2)\big)\dd \mathsf{x}^{{\rm A}+{\rm B}}(x_2)\Big]\,.
\end{split}
\end{equation*}
We can now apply dominated convergence by putting absolute values around the integrands and we get that the limit as $\eta\to 0$ is zero.

Therefore, we are led to carry out the following integrals:
\begin{equation}\label{eq:M02split}
\begin{split}
& \mathcal{M}_{0,2}[f] = 
\\
&  \quad 
\underset{\eta\to 0}{\lim}\, \Bigg[\frac{-1}{4\pi^2}\int_{-C_1^{+}}\int_{-C_2^{+}}\Bigg(\frac{\dd x_1 \dd x_2}{(x_1-x_2)^2}+\frac{\dd \overline{x}_1 \dd \overline{x}_2}{(\overline{x}_1-\overline{x}_2)^2}\Bigg)\prod\limits_{j=1}^{2}f\big(\mathsf{x^{{\rm A}+ {\rm B}}}(x_j)\big) \\
& \quad +\frac{1}{4\pi^2}\int_{-C_1^{+}}\int_{-C_2^{+}}\Bigg(\frac{\dd x_1 \dd \overline{x}_2}{(x_1-\overline{x}_2)^2}+\frac{\dd \overline{x}_1 \dd x_2}{(\overline{x}_1-x_2)^2} \Bigg)\prod\limits_{j=1}^{2}f\big(\mathsf{x^{{\rm A}+ {\rm B}}}(x_j)\big)\Bigg]\,.
\end{split}
\end{equation} 

Again, we perform integration by parts with respect to $x_2$ and $x_1$. It turns out that the boundary terms in the two lines of \eqref{eq:M02split} cancel each other. In the end, we obtain:
\[\begin{split}
\mathcal{M}_{0,2}[f]=& \underset{\eta \to 0}{\lim}\, \Bigg[\frac{-1}{4\pi^2}\int_{-C_1^{+}}\int_{-C_2^{+}} \ln |x_1-x_2|^2 \prod\limits_{j=1}^{2}f'\big(\mathsf{x^{{\rm A}+ {\rm B}}}(x_j)\big)\dd \mathsf{x^{{\rm A}+ {\rm B}}}(x_j)\\
& + \frac{1}{4\pi^2}\int_{-C_1^{+}}\int_{-C_2^{+}} \ln |x_1-\overline{x}_2|^2 \prod\limits_{j=1}^{2}f'\big(\mathsf{x^{{\rm A}+ {\rm B}}}(x_j)\big)\dd \mathsf{x^{{\rm A}+ {\rm B}}}(x_j)\Bigg]\\
= & \frac{1}{4\pi^2}\int_{-C_1^{+}}\int_{-C_1^{+}} \ln \left|\frac{x_1-\overline{x}_2}{x_1-x_2}\right|^2\prod\limits_{j=1}^{2}f'\big(\mathsf{x^{{\rm A}+ {\rm B}}}(x_j)\big)\dd \mathsf{x^{{\rm A}+ {\rm B}}}(x_j)\,.
\end{split}\]
Notice that thanks to the way the integral was split, only the real logarithm is involved here and every term is integrable. Using the change of variable $x_j = \gamma_+(\xi_j)$, we get the desired expression. 
\end{proof}

\begin{proof}[Proof of Proposition~\ref{prexp03}]
The term we want to compute is
\begin{equation}
\label{M30ciung}
\begin{split}
 \mathcal{M}_{0,3}[f] & \coloneqq - \mathop{{\rm Res}}_{x_1 = \infty} \mathop{{\rm Res}}_{x_2 = \infty}  \mathop{{\rm Res}}_{x_3 = \infty}  \Omega_{0,3}^{{\rm A} + {\rm B}}(x_1,x_2,x_3) \prod_{j = 1}^{3} f\big(\mathsf{x}^{{\rm A} + {\rm B}}(x_j)\big) \\
& = -\mathop{{\rm Res}}_{x_1 = \infty} \mathop{{\rm Res}}_{x_2 = \infty}  \mathop{{\rm Res}}_{x_3 = \infty} M_{0,3}(x_1,x_2,x_3) \prod_{j = 1}^{3} f'\big(\mathsf{x}^{{\rm A} + {\rm B}}(x_j)\big)\dd \mathsf{x}^{{\rm A} + {\rm B}}(x_j)\,,
\end{split}
\end{equation}
where we have introduced
$$
M_{0,3}(x_1,x_2,x_3) = s^2 \bigg(\sum_{i = 1}^{3} \frac{\dd x_i}{\dd \mathsf{x}^{{\rm A} + {\rm B}}(x_i)\,(x_i - x_j)(x_i -x_k)}\bigg)\,.
$$
As in the previous proofs, we would like to move the three integration contours to $\Gamma$ and rewrite the resulting expression as an integral over $I^3$, but this can be done only when the function
\[
\begin{split}
& \Delta M_{0,3}(\xi_1,\xi_2,\xi_3) \coloneqq 
\\
& \quad -\frac{1}{(2{\rm i}\pi)^3} \sum_{\epsilon_1,\epsilon_2,\epsilon_3 = \pm 1} \epsilon_1 \epsilon_2 \epsilon_3\,M_{0,3}\big(\xi_1 + {\rm i}\epsilon_1\Upsilon_+(\xi_1),\xi_2 + {\rm i}\epsilon_2 \Upsilon_+(\xi_2),\xi_3 + {\rm i}\epsilon_3\Upsilon_+(\xi_3)\big)
\end{split}
\]
is integrable on $I^3$. The function $M_{0,3}$ has an alternative expression found in \eqref{03res}: it is of the form $\frac{F(x_1,x_2,x_3)}{\tilde{Q}(x_1)\tilde{Q}(x_2)\tilde{Q}(x_3)}$, where $F$ is a polynomial. We only need to check the local integrability of $\Delta M_{0,3}$ when $\xi_i$ approaches an endpoint of $I$, since these correspond to ramification points, i.e. roots of $\tilde{Q}$. 

Assume $L \neq s \lambda(t)$. In this case we learned in Section~\ref{dstrucsp} that all ramification points are simple. Let us take a triple $\mathbf{r} = (r^{(1)},r^{(2)},r^{(3)})$ of ramification points (possibly equal). With a partial fraction expansion we can find two triples of constants $\boldsymbol{\alpha}$, $\boldsymbol{\beta}$ such that
$$
M_{0,3}(x_1,x_2,x_3) =  \prod_{j= 1}^{3} \Big(\alpha^{(j)} + \frac{\beta^{(j)}}{x_j - r^{(j)}}\Big) + \tilde{M}_{0,3}(x_1,x_2,x_3)\,.
$$
Then $\Delta M_{0,3}(\xi_1,\xi_2,\xi_3) = -\prod_{j = 1}^{3} \beta^{(j)} K_{r^{(j)}}(\xi_j) + \Delta \tilde{M}_{0,3}(\xi_1,\xi_2,\xi_3)$, where
$$
K_r(\xi) = \frac{1}{2{\rm i}\pi}\bigg(\frac{1}{\xi + {\rm i}\Upsilon_+(\xi) - r} - \frac{1}{\xi - {\rm i}\Upsilon_+(\xi) - r}\bigg) = \frac{-1}{\pi}\,\frac{\Upsilon_+(\xi)}{(\xi -r)^2 + \Upsilon_+(\xi)^2}\,.
$$
As a consequence of Lemma~\ref{propparam}, $K_r(\xi) = O\big(|\xi - r|^{-1/2}\big)$ as $\xi \rightarrow r$, so it is locally integrable near $r$. Therefore, $\Delta M_{0,3}$ is integrable on $I^3$. Following the procedure of contour deformation in \eqref{M30ciung} we obtain
$$
\mathcal{M}_{0,3}[f] = \int_{I^3} \Delta M_{0,3}(\xi_1,\xi_2,\xi_3)  \prod_{j =1}^{3} f'\big(c(\xi_j))\dd c(\xi_j) \,.
$$

If $L = s\lambda(t)$, the ramification point $r_0(t)$ is double, and Lemma~\ref{propparam} implies that $\Delta M_{0,3}(\xi_1,\xi_2,\xi_3)$ can develop a double pole at $\xi = r_0(t)$, which is compensated by the double zero of $c'$ at $r_0(t)$ (see Lemma~\ref{propparam}) since the integrand contains $\dd c(\xi_j) = c'(\xi_j) \dd \xi_j$. So the previous argument applies as well. \end{proof}

\newpage

\appendix

\section{M\"obius functions on partitioned permutations}
\label{App:Moebius}

We return to Definition~\ref{def:PSPmudel}. The M\"obius function $\mu_{\hbar} : PS(d) \rightarrow \mathbb{Z}\llbracket \hbar \rrbracket$ is the unique multiplicative function such that
$$
\mu_{\hbar} \circledast \zeta_{\hbar} = \delta\,,
$$
where the non-zero values of $\zeta_{\hbar}$ are $\zeta_{\hbar}(\mathbf{0}_{\alpha},\alpha) = \hbar^{|\alpha|}$ for $\alpha \in S(d)$. We can give it an explicit formula by writing $\mu_{\hbar} \circledast (\delta + \zeta_{\hbar} - \delta) = \delta$ and expanding in geometric series
\begin{equation}
\label{museries}
\mu_{\hbar} = \sum_{k \geq 0} (-1)^k \underbrace{(\zeta_{\hbar} - \delta) \circledast \cdots \circledast (\zeta_{\hbar} - \delta)}_{k\,\,{\rm times}}\,.
\end{equation}
The $k = 0$ term should be interpreted as $\delta$. Only the terms with $k \leq l$ contribute to the coefficient of $\hbar^l$, so that the series yields a well-defined formal power series in $\hbar$.

To compute $\mu_{\hbar}$ explicitly for low values of $d$, we can take the following route. For $d = 1$, $\zeta_{\hbar} = \delta = \mu_{\hbar}$ is the unit function. For $d \geq 2$, identifying functions $PS(d) \rightarrow \mathbb{Z}\llbracket\hbar\rrbracket$ with elements of the monoid algebra $\mathbb{Z}\llbracket \hbar \rrbracket[PS(d)]$, we have the decomposition
\begin{equation}
\label{zetahbardelta}\zeta_{\hbar} - \delta = \sum_{\substack{\lambda \vdash d \\ \lambda \ne (1^{d})}}  \hbar^{d - \ell(\lambda)} \left( \sum_{\alpha \in C_{\lambda}} (\mathbf{0}_{\alpha},\alpha)\right).
\end{equation}
This function is invariant under simultaneous conjugation. Since the extended convolution preserves this property, we can decompose for any $k \geq 1$
\begin{equation}
\label{zetaconv}
(\zeta_{\hbar} - \delta)^{\circledast k} =  \sum_{\substack{\lambda \vdash d \\ \mathbf{0}_{\lambda} \leq \mathcal{L} \\ (\mathcal{L},\pi_\lambda) \neq (\mathbf{0}_{{\rm id}},{\rm id})}} z_{\hbar,k}(\mathcal{L},\pi_{\lambda})\left(\sum_{\alpha = g\pi_{\lambda}g^{-1} \in C_{\lambda}} (g\mathcal{L},\alpha)\right),
\end{equation}
where all coefficients are in $\mathbb{Z}\llbracket\hbar\rrbracket$. Equation~\eqref{zetahbardelta} gives us the initial condition $z_{\hbar,1}(\mathbf{0}_{\lambda},\pi_{\lambda}) =  \hbar^{d - \ell(\lambda)}$ and $z_{\hbar,1}(\mathcal{L},\pi_{\lambda}) = 0$ if $\mathcal{L} \neq \mathbf{0}_{\lambda}$. Convoluting \eqref{zetahbardelta} from the left with \eqref{zetaconv} gives an expression of the same form. This operation is represented by a matrix $\mathbf{C}_{\hbar}$ with rows and columns indexed by $(\mathcal{L},\lambda)$ which only reflects the $\odot$ product on $PS(d)$. Calling $\mathbf{z}_{\hbar,k}$ the column vector of all coefficients $z_{\hbar,k}(\mathcal{L},\pi_{\lambda})$ appearing in \eqref{zetaconv}, we get
\begin{equation}
\label{linasy}  
\mathbf{z}_{\hbar,k + 1} = \mathbf{C}_{\hbar} \mathbf{z}_{\hbar,k}\,,
\end{equation}
Calling $\boldsymbol{\mu}_{\hbar}$ the column vector of values of $\mu_{\hbar}(\mathcal{L},\pi_{\lambda})$ for $(\mathcal{L},\lambda)$ appearing in \eqref{zetaconv}, we get from \eqref{museries}
$$ 
\boldsymbol{\mu}_{\hbar} = -({\rm Id} + \mathbf{C}_\hbar)^{-1}\mathbf{z}_{\hbar,1}.
$$ 
 
In the following examples, we use notations like $\big[\begin{smallmatrix} \overline{12},\,\overline{345} \\ (12)(34) \end{smallmatrix}\big]_4$ to mean $\mathcal{L} = \{\{1,2\},\{3,4,5\}\}$, $\pi_{\lambda} = (12)(34)$, that is $\lambda = (2,2,1)$, and the index in the bottom-right corner indicates that the colength of $(\mathcal{L},\pi_{\lambda})$ is $4$.

\begin{example}
For $d = 2$, index rows and columns by
$$
\big[\begin{smallmatrix} \overline{12} \\ (12) \end{smallmatrix}\big]_1\,,\qquad \big[\begin{smallmatrix} \overline{12} \\ {\rm id} \end{smallmatrix}\big]_2\,.
$$ 
We find
$$
\mathbf{z}_{\hbar,k = 1} = \left(\begin{array}{cc} \hbar \\ 0 \end{array}\right)
,\qquad \mathbf{C}_{\hbar} = \left(\begin{array}{cc} 0 & \hbar \\ \hbar & 0 \end{array}\right).
$$
Therefore
\begin{equation}
\begin{split}
\mu_{\hbar}\big[\begin{smallmatrix} \overline{12} \\ (12) \end{smallmatrix}\big]_1 = -\frac{\hbar}{1 - \hbar^2}\,,\qquad 
\mu_{\hbar}\big[\begin{smallmatrix} \overline{12} \\ {\rm id} \end{smallmatrix}\big]_2 = \frac{\hbar^2}{1 - \hbar^2}\,.
\end{split}
\end{equation}
\end{example}

\begin{example} \label{Mobd3} For $d = 3$, index rows and columns in the following order
$$
\big[\begin{smallmatrix}  \overline{12},\,\overline{3} \\ (12) \end{smallmatrix}\big]_1\,,\quad \big[\begin{smallmatrix}  \overline{12},\,\overline{3} \\ {\rm id} \end{smallmatrix}\big]_2\,,\quad \big[\begin{smallmatrix} \overline{123} \\ (123)  \end{smallmatrix}\big]_2\,,\quad \big[\begin{smallmatrix} \overline{123} \\ (12) \end{smallmatrix}\big]_3\,,\quad \big[\begin{smallmatrix} \overline{123} \\ {\rm id} \end{smallmatrix}\big]_4\,.
$$
We have
$$
\mathbf{z}_{\hbar,1} = \left(\begin{array}{c} \hbar \\ 0 \\ \hbar^2 \\ 0 \\ 0 \end{array}\right),\qquad \mathbf{C}_{\hbar} = \left(\begin{array}{ccccc} 0 & \hbar & 0 & 0 & 0 \\ \hbar & 0 & 0 & 0 & 0 \\ 3\hbar & 3\hbar^2 & \hbar^2 & 3\hbar & \hbar^2 \\ 2\hbar^2 & 2\hbar & 2\hbar & 2\hbar^2 & \hbar \\ 0 & 0 & 2\hbar^2 & 3\hbar & 0 \end{array}\right). 
$$
Equation~\eqref{linasy} leads to
\begin{equation*}
\begin{split}
\mu_{\hbar}\big[\begin{smallmatrix}  \overline{12},\,\overline{3} \\ (12) \end{smallmatrix}\big]_1 & = \frac{-\hbar}{1 - \hbar^2}\,, \qquad \mu_{\hbar}\big[\begin{smallmatrix}  \overline{12},\,\overline{3} \\ {\rm id} \end{smallmatrix}\big]_2 = \frac{\hbar^2}{1 - \hbar^2}  \,,\\
\mu_{\hbar}\big[\begin{smallmatrix} \overline{123} \\ (123)  \end{smallmatrix}\big]_2 & = \frac{2\hbar^2}{(1 - \hbar^2)(1 -4\hbar^2)} \,, \quad\,\,\, \mu_{\hbar}\big[\begin{smallmatrix} \overline{123} \\ (12) \end{smallmatrix}\big]_3 = -\frac{4\hbar^3}{(1 - \hbar^2)(1 - 4\hbar^2)}\,,\\
\mu_{\hbar}\big[\begin{smallmatrix} \overline{123} \\ {\rm id} \end{smallmatrix}\big]_4 & = \frac{8\hbar^4}{(1 - \hbar^2)(1 - 4\hbar^2)}\,.
\end{split}
\end{equation*}
\end{example}

\begin{remark} By multiplicativity of $\mu_{\hbar}$, the coefficients with $\mathcal{L}$ having more than one block are  polynomial expressions in the coefficients of the same function on $PS(d')$ with $d' < d$. In Example~\ref{Mobd3}, multiplicativity is manifest:
$$
\mu_{\hbar}\big[\begin{smallmatrix} \overline{12},\,\overline{3} \\ {\rm id} \end{smallmatrix}\big]_2 = \mu_{\hbar}\big[\begin{smallmatrix} \overline{12} \\ {\rm id} \end{smallmatrix}\big]_2 \cdot \mu_{\hbar}\big[\begin{smallmatrix} \overline{3} \\ {\rm id} \end{smallmatrix}\big]_0\,,\qquad \mu_{\hbar}\big[\begin{smallmatrix} \overline{12},\,\overline{3} \\ (12) \end{smallmatrix}\big]_1 = \mu_{\hbar}\big[\begin{smallmatrix} \overline{12} \\ (12) \end{smallmatrix}\big]_1 \cdot \mu_{\hbar}\big[\begin{smallmatrix} \overline{3} \\ {\rm id} \end{smallmatrix}\big]_0\,.
$$
\end{remark}

\begin{remark}The M\"obius function for the convolution $\ast$ is retrieved by
$$
\mu(\mathcal{A},\alpha) = \lim_{\hbar \rightarrow 0} \hbar^{-|(\mathcal{A},\alpha)|} \mu_{\hbar}(\mathcal{A},\alpha)\,.
$$
For instance, we get
$$
\mu\big[\begin{smallmatrix} \overline{123} \\ (123) \end{smallmatrix}\big]_2 = 2 = {\rm Cat}(2) \,,\qquad \mu\big[\begin{smallmatrix} \overline{123} \\ (12) \end{smallmatrix}\big]_3 = -4 = - {\rm Cat}(2,1) \,,
$$
with
\begin{equation*}
\begin{split}
{\rm Cat}(d) & = \frac{1}{d + 1} \binom{2d}{d}, \\
{\rm Cat}(d_1,d_2) & = \frac{2d_1d_2}{d_1 + d_2} \binom{2d_1 - 1}{d_1} \binom{2d_2 - 1}{d_2}
\end{split}
\end{equation*}
being the numbers of non-crossing disc partitions and annular non-crossing permutations, respectively, in agreement with the known values of the M\"obius function \cite{Kreweras,CMSS}.
\end{remark}

\section{About the GUE}

Consider again the SPS on $\mathscr{A} = \mathbb{C}[a]$ given by the topological expansion of a GUE matrix with $s = 1$, cf. Section~\ref{GUESetting}. As observed in \eqref{kappatrivGUE}, almost all free cumulants are zero and we get
\[
G_{g,n}^{\vee}(w_1,\ldots,w_n) = \delta_{g,0}\delta_{n,1}(1 + w_1^2)\,,
\]
and Theorem~\ref{thm:R-transform-HigherGenera} then gives a formula for the GUE moments. There are only contributions from the graphs $\Gamma \in \mathcal{G}_{n}$ in which black vertices have valency $2$ and cannot be connected to the same white vertex. Erasing the black vertices, this amounts to considering the set ${\rm Gr}_{n}$ of connected graphs having $n$ vertices labelled from $1$ to $n$ without loops. We observe $\varsigma(\hbar u w \partial_w) w^k =  \varsigma(k\hbar u)w^k$ for $k > 0$. The weight of a hyperedge between the $i$-th and $j$-th vertices becomes
\[ 
\mathsf{c}^{\vee}(u_i,w_i,u_j,w_j) = \ln\bigg(\frac{(e^{\hbar u_i/2}w_i - e^{\hbar u_j/2}w_j)(e^{-\hbar u_i/2}w_i - e^{-\hbar u_j/2}w_j)}{(e^{\hbar u_i/2}w_i - e^{-\hbar u_j/2}w_j)(e^{-\hbar u_i/2}w_i - e^{\hbar u_j/2}w_j)}\bigg)\,,
\]
while the operator weight on a vertex is
\begin{equation*}
\begin{split}
\vec{\mathsf{O}}^{\vee}(w) & = \sum_{m \geq 0} (P^{\vee}(w)w\partial_w)^m
\\
& \quad P^{\vee}(w) \cdot [v^m] \,\, \sum_{r \geq 0} \Big(\partial_y + \frac{v}{y}\Big)^{r} \exp\bigg(v\,\frac{\varsigma(\hbar v\partial_y)}{\varsigma(\hbar \partial_y)}\ln y - v\ln y\bigg)\Big|_{y = 1 + w^2} \\
& \quad \cdot [u^r] \,\,\frac{\exp\big(u(\varsigma(2\hbar u) - 1)w^2\big)}{\hbar u \varsigma(\hbar u)}\,.
\end{split}
\end{equation*}
For $n \geq 2$ and $2g - 2 + n > 0$, the formula is then
\[
\begin{split}
& G_{g,n}(X_1,\ldots,X_n) = 
\\
& \quad [\hbar^{2g - 2 + n}] \sum_{\Gamma \in {\rm Gr}_n} \frac{1}{\# {\rm Aut}(\Gamma)} \prod_{i = 1}^n \vec{\mathsf{O}}^{\vee}(w_i) \prod_{\{i,j\} = {\rm edge}} \mathsf{c}^{\vee}(u_i,w_i,u_j,w_j)\,,
\end{split} 
\]
with the substitutions:
\[
X_i = \frac{w_i}{1 + w_i^2}\,,\qquad P^{\vee}(w_i) = \frac{1 + w_i^2}{1-w_i^2}\,.
\]
For $n = 1$ and $g \geq 1$, as loops are forbidden, the vertex with no edge is the only graph:
\[
G_{g,1}(X_1) = \Delta_g^{\vee}(X_1) + [\hbar^{2g - 1}]\,\, \vec{\mathsf{O}}^{\vee}(w_1)\,.
\]
The resulting formula is essentially the Harer--Zagier formula, as shown in \cite{DaniHZ} where the proof is a special case of the general method we have used in Section~\ref{Sec3}. 

\section{Hock formula} \label{sec:Hock}

For completeness, we also state an equivalent version of the key formula of Theorem~\ref{thm:R-transform-GenusZero} that was found by Hock in~\cite{hock2022xy}.

\begin{theorem}
	\label{thm:R-transform-GenusZero-Hock} Let $f,f^{\vee}\colon PS \rightarrow R$ be multiplicative functions and $G_{0,n},G_{0,n}^{\vee}$ their respective $n$-point functions. Suppose that $f= \zeta \ast f^{\vee}$. Then, under the substitutions
	\[
	X_i = \frac{w_i}{G_{0,1}^{\vee}(w_i)}\,,\qquad P^{\vee}(w_i) =  \frac{\dd \ln w_i}{\dd \ln X_i}\,,
	\]
	we have for any $n \geq 3$:
	\[
\begin{aligned}
&\left(\prod_{i=1}^n X_i\right)G_{0,n}(X_1,\ldots,X_n)
\\
&\quad=
\sum_{r_1,\ldots,r_n\geq0}
\prod_{i=1}^{n}
\bigl(X_i^2\partial_{X_i}\bigr)^{r_i}
X_iw_iP^\vee(w_i)
\sum_{T\in\mathcal G_{0,n}(\mathbf r+1)}
\prod_{I\in\mathcal I(T)}'
\frac{G^\vee_{0,\#I}(w_I)}
     {\prod_{i\in I}w_i},
\end{aligned}
\]
where $\prod'$ means that one should replace each occurrence of $G_{0,2}^{\vee}(w_i,w_j)$ with $\widetilde{G}_{0,2}^{\vee}(w_i,w_j)$. 
\end{theorem}

\begin{remark} Here $X^2\partial_X$ can be rewritten in the coordinate $w$ as 
	\[
	X^2\partial_X = P^\vee(w) wX \partial_w = \frac{P^\vee(w)w^2}{G_{0,1}^\vee(w)} \partial_w,
	\]
	and the factor $XwP^\vee(w)$ can alternatively be represented as 
	\[
	XwP^{\vee}(w) = \frac{w^2P^{\vee}(w)}{G^{\vee}_{0,1}(w)}. 
	\]
\end{remark}

\newpage

\printbibliography
\end{document}